\documentclass[%
  onecolumn
]{mpi2015-cscpreprint}

\usepackage{algorithm, algorithmicx, algpseudocode}
\newcommand{\SComment}[1]{\Comment{{\scriptsize #1}}} 
\usepackage{amsthm} 
\theoremstyle{plain}
	\newtheorem{theorem}{Theorem}
	\newtheorem{corollary}{Corollary}
	\newtheorem{lemma}{Lemma}
	
\theoremstyle{definition}

\theoremstyle{remark}
	
\usepackage{amsmath, amscd, amssymb, bm, mathrsfs, mathtools}
\usepackage{array, booktabs, makecell} 
\newcolumntype{P}[1]{>{\centering\arraybackslash}p{#1}}
\newcolumntype{M}[1]{>{\centering\arraybackslash}m{#1}}
\usepackage[american]{babel}
\usepackage{enumerate}
\usepackage{graphicx}
\usepackage{hyperref}
\usepackage{soul}
\usepackage{url}
\usepackage{tikz}
\usepackage[backgroundcolor=violet!60,textcolor=white,linecolor=gray]{todonotes}
\usepackage{xcolor}
\usepackage{xspace}

\newcommand{\tr}[1]{\text{tr}\left(#1\right)}

\newcommand{\chol}{\texttt{chol}\xspace}

\newcommand{\Proj}[1]{\texttt{Proj}\left(#1\right)}
\newcommand{\QR}[1]{\texttt{QR}\left(#1\right)}

\newcommand{\spR}{\mathbb{R}}				

\newcommand{\vv}{\bm{v}}

\newcommand{\vy}{\bm{y}}

\newcommand{\vC}{\bm{C}}

\newcommand{\vD}{\bm{D}}

\newcommand{\vE}{\bm{E}}
\newcommand{\vEproj}{\Delta\vE_{\text{proj}}}

\newcommand{\vG}{\bm{G}}
\newcommand{\vGtil}{\tilde{\vG}}
\newcommand{\vGbar}{\bar{\vG}}
\newcommand{\vH}{\bm{H}}
\newcommand{\vHtil}{\tilde{\vH}}
\newcommand{\vHbar}{\bar{\vH}}

\newcommand{\vM}{\bm{M}}

\newcommand{\vQ}{\bm{Q}}

\newcommand{\vQbar}{\bar{\vQ}}
\newcommand{\vQtil}{\widetilde{\vQ}}

\newcommand{\vU}{\bm{U}}

\newcommand{\vUbar}{\bar{\vU}}
\newcommand{\vV}{\bm{V}}
\newcommand{\vVbar}{\bar{\vV}}
\newcommand{\vVtil}{\widetilde{\vV}}
\newcommand{\vW}{\bm{W}}
\newcommand{\vWbar}{\bar{\vW}}

\newcommand{\vWtil}{\widetilde{\vW}}
\newcommand{\vX}{\bm{X}}

\newcommand{\vY}{\bm{Y}}

\newcommand{\vZ}{\bm{Z}}
\newcommand{\vZbar}{\bar{\vZ}}

\newcommand{\Pbar}{\bar{P}}

\newcommand{\Rbar}{\bar{R}}

\newcommand{\Sbar}{\bar{S}}

\newcommand{\Tbar}{\bar{T}}
\newcommand{\Ombar}{\bar{\Omega}}

\newcommand{\Ybar}{\bar{Y}}

\newcommand{\RR}{\mathcal{R}}
\newcommand{\RRbar}{\bar{\RR}}
\renewcommand{\SS}{\mathcal{S}}
\newcommand{\SSbar}{\bar{\SS}}

\newcommand{\TT}{\mathcal{T}}

\newcommand{\YY}{\mathcal{Y}}
\newcommand{\YYbar}{\bar{\mathcal{Y}}}

\newcommand{\bQQ}{\bm{\mathcal{Q}}}
\newcommand{\bQQbar}{\bar{\bQQ}}
\newcommand{\bQQbarnew}{\bar{\bQQ}_{new}}
\newcommand{\bQQbarprev}{\bar{\bQQ}_{prev}}
\newcommand{\bQQnew}{\bQQ_{new}}
\newcommand{\bQQprev}{\bQQ_{prev}}

\newcommand{\bUU}{\bm{\mathcal{U}}}

\newcommand{\bXX}{\bm{\mathcal{X}}}

\newcommand{\inv}{{-1}}
\newcommand{\tinv}{{-T}}

\newcommand{\ihalf}{{-1/2}}
\newcommand{\bmat}[1]{\begin{bmatrix}#1\end{bmatrix}}

\newcommand{\bigO}[1]{\mathcal{O}\left(#1\right)}

\newcommand{\sigmin}{\sigma_{\min}}

\newcommand{\norm}[1]{\left\lVert#1\right\rVert}

\newcommand{\normF}[1]{\norm{#1}_{\text{F}}}

\makeatletter
\newsavebox{\@brx}
\newcommand{\llangle}[1][]{\savebox{\@brx}{\(\m@th{#1\langle}\)}%
	\mathopen{\copy\@brx\kern-0.5\wd\@brx\usebox{\@brx}}}
\newcommand{\rrangle}[1][]{\savebox{\@brx}{\(\m@th{#1\rangle}\)}%
	\mathclose{\copy\@brx\kern-0.5\wd\@brx\usebox{\@brx}}}
\makeatother

\newcommand{\eps}{\varepsilon}

\newcommand{\monomial}{\texttt{monomial}\xspace}

\newcommand{\piled}{\texttt{piled}\xspace}

\newcommand{\BCGS}{\hyperref[alg:BCGSA]{\texttt{BCGS}}\xspace}	
\newcommand{\BCGSA}{\hyperref[alg:BCGSA]{\texttt{BCGS-A}}\xspace}	
\newcommand{\BCGSIRO}{\hyperref[alg:BCGSIROA]{\texttt{BCGSI+}}\xspace}	
\newcommand{\BCGSIROA}{\hyperref[alg:BCGSIROA]{\texttt{BCGSI+A}}\xspace}	
\newcommand{\BCGSIROAthree}{\hyperref[alg:BCGSIROA3S]{\texttt{BCGSI+A-3S}}\xspace}
\newcommand{\BCGSIROAtwo}{\hyperref[alg:BCGSIROA2S]{\texttt{BCGSI+A-2S}}\xspace}
\newcommand{\BCGSIROAone}{\hyperref[alg:BCGSIROA1S]{\texttt{BCGSI+A-1S}}\xspace}
\newcommand{\BCGSIROLS}{\texttt{BCGSI+LS}\xspace}	
\newcommand{\BMGS}{\texttt{BMGS}\xspace}	
\newcommand{\BCGSIROF}{\texttt{BCGSI+F}\xspace}	
\newcommand{\BCGSCP}{\texttt{BCGS-CP}\xspace}	
\newcommand{\IO}[1]{\texttt{IO}\left(#1\right)}	
\newcommand{\IOnoarg}{\texttt{IO}\xspace}	
\newcommand{\IOnoargs}{\texttt{IO}s\xspace} 
\newcommand{\IOA}[1]{\texttt{IO}_{\mathrm{A}}\left(#1\right)}
\newcommand{\IOAnoarg}{\texttt{IO}_{\mathrm{A}}\xspace}
\newcommand{\IOone}[1]{\texttt{IO}_1\left(#1\right)}
\newcommand{\IOonenoarg}{\texttt{IO}_1\xspace}
\newcommand{\IOtwo}[1]{\texttt{IO}_2\left(#1\right)}
\newcommand{\IOtwonoarg}{\texttt{IO}_2\xspace}

\newcommand{\CGS}{\texttt{CGS}\xspace}	
\newcommand{\MGS}{\texttt{MGS}\xspace}	
\newcommand{\HouseQR}{\texttt{HouseQR}\xspace}	
\newcommand{\GivensQR}{\texttt{GivensQR}\xspace}	
\newcommand{\TSQR}{\texttt{TSQR}\xspace}	
\newcommand{\CholQR}{\texttt{CholQR}\xspace}	

\newcommand{\epsqr}{\rho_{\mathrm{qr}}}
\newcommand{\epsproj}{\psi}
\newcommand{\epsprojbcgsa}{\psi^{(\BCGSA)}}

\newcommand{\epsprojbcgsiroaone}{\psi^{(\BCGSIROAone)}}
\newcommand{\epsprojbcgsiroathree}{\psi^{(\BCGSIROAthree)}}
\newcommand{\epsQ}{\omega_{\mathrm{qr}}}
\newcommand{\epsQone}{\omega_{prev}}
\newcommand{\epsQtwo}{\omega_{new}}
\newcommand{\epsQk}{\omega_k}
\newcommand{\epsQkp}{\omega_{k-1}}

\newcommand{\epsQip}{\omega_{i-1}}
\newcommand{\epsXkp}{\rho_{k-1}} 
\newcommand{\constQ}{\alpha}
\newcommand{\constQA}{\alpha_{\mathrm{A}}}

\definecolor{plotblue}{RGB}{0, 113.9850, 188.9550}
\definecolor{plotred}{RGB}{216.7500, 82.8750, 24.9900}
\definecolor{plotpurple}{RGB}{125.9700, 46.9200, 141.7800}
\newcommand{\first}[1]{\textcolor{plotblue}{#1}}
\newcommand{\second}[1]{\textcolor{plotred}{#1}}
\newcommand{\combo}[1]{\textcolor{plotpurple}{#1}}

\begin{document}


\title{On the loss of orthogonality in low-synchronization variants of reorthogonalized block classical Gram-Schmidt}

\author[$\ast$]{Erin Carson}
\affil[$\ast$]{Department of Numerical Mathematics, Faculty of Mathematics and Physics, Charles University, Sokolovsk\'{a} 49/83, 186 75 Praha 8, Czechia\authorcr \email{\{carson, oktay\}@karlin.mff.cuni.cz}, \email{yuxin.ma@matfyz.cuni.cz}, \orcid{0000-0001-9469-7467}, \orcid{0000-0003-0761-2184}, \orcid{0000-0002-2860-0134}}

\author[$\dagger, \ddagger$]{Kathryn Lund}
\affil[$\dagger$]{Computational Mathematics Theme, Building R71, STFC Rutherford Appleton Laboratory, Harwell Oxford, Didcot, Oxfordshire, OX11 0QX, United Kingdom\authorcr \email{kathryn.lund@stfc.ac.uk}, \orcid{0000-0001-9851-6061}}
\affil[$\ddagger$]{Computational Methods in Systems and Control Theory, Max Planck Institute for Dynamics of Complex Technical Systems, Sandtorstr.\ 1, 39106 Magdeburg, Germany}

\author[$\ast$]{Yuxin Ma}

\author[$\ast,\S$]{Eda Oktay}
\affil[$\S$]{Department of Mathematics, Chemnitz University of Technology, Reichenhainer Str.\ 41, 09126 Chemnitz, Germany}

\shortauthor{E. Carson, K. Lund, Y. Ma, and E. Oktay}

\keywords{
    backward stability, Gram-Schmidt, low-synchronization, communication-avoiding, Arnoldi method, Krylov subspaces, loss of orthogonality
}

\msc{
    65-04, 65F10, 65F25, 65F50, 65G50, 65Y05
}
  
\abstract{
    Interest in communication-avoiding orthogonalization schemes for high-performance computing has been growing recently.  This manuscript addresses open questions about the numerical stability of various block classical Gram-Schmidt variants that have been proposed in the past few years.  An abstract framework is employed, the flexibility of which allows for new rigorous bounds on the loss of orthogonality in these variants. We first analyze a generalization of (reorthogonalized) block classical Gram-Schmidt and show that a ``strong'' intrablock orthogonalization routine is only needed for the very first block in order to maintain orthogonality on the level of the unit roundoff. In particular, this ``strong" first step does not have to be a reorthogonalized QR itself and subsequent steps can use less stable QR variants, thus keeping the overall communication costs low.

Then, using this variant, which has four synchronization points per block column, we remove the synchronization points one at a time and analyze how each alteration affects the stability of the resulting method. Our analysis shows that the variant requiring only one synchronization per block column, equivalent to a variant previously proposed in the literature, cannot be guaranteed to be stable in practice, as stability begins to degrade with the first reduction of synchronization points. As a negative result, we conclude that this particular block algorithm should be avoided in practice.

Our analysis of block methods also provides new, more positive theoretical results for the single-column case. In particular, it is proven that DCGS2 from [Bielich, D. et al. \emph{Par. Comput.} 112 (2022)] and CGS-2 from [\'{S}wirydowicz, K. et al, \emph{Num. Lin. Alg. Appl.} 28 (2021)] are as stable as Householder QR.  Numerical examples from the \texttt{BlockStab} toolbox are included throughout, to help compare variants and illustrate the effects of different choices of intraorthogonalization subroutines.

}

\novelty{Bounds on the loss of orthogonality are proven for a block Gram-Schmidt variant with one synchronization point.  As a by-product, bounds on the order of unit roundoff are obtained for the single-column variant, which was until now an open problem.  New bounds are also proven for block classical Gram-Schmidt, and a variety of numerical examples are provided to confirm the theoretical results.}

\maketitle



\section{Introduction} \label{sec:intro}
With the advent of exascale computing, there is a pressing need for highly parallelizable algorithms that also reduce \emph{communication}, i.e., data movement and synchronization.  An underlying kernel in diverse numerical linear applications is the orthogonalization of a matrix, whose efficiency is limited by inner products and vector normalizations involving \emph{synchronization points (sync points)}, which dominate communication costs.  In the case of an inner product or norm, a sync point arises when a vector is stored in a distributed fashion across nodes: each node locally computes part of the inner product and then must transmit its result to all other nodes so that each can assemble the full inner product.  Communication-avoiding methods such as $s$-step Krylov subspace methods have proven to be effective adaptations in practice; see, e.g., \cite{BalCDetal14, Car15, Hoe10, YamTHetal20}. Such methods implicitly rely on a stable block Gram-Schmidt (BGS) routine that should itself be communication-avoiding.  Blocking alone reduces the number of sync points, as previously vector-wise operations can instead be performed on tall-skinny matrices or \emph{block vectors}, thus replacing single inner products with block inner products and normalizations with low-sync QR factorizations.  Low-sync variants of BGS have attracted much recent attention \cite{BieLTetal22, CarLR21, CarLRetal22, CarLMetal24a, OktC23, SwiLAetal21, YamTHetal20, YamHBetal24, Zou23}, but their stability, in particular how well they preserve orthogonality between basis vectors, is often poor, which can lead to issues in downstream applications like Krylov subspace methods, least-squares, or eigenvalue solvers. Understanding the floating-point stability of low-sync BGS methods is thus imperative for their reliable deployment in exascale environments.

To be more precise, we define a \emph{block vector} $\vX \in \spR^{m \times s}$ as a concatenation of $s$ column vectors.  We are interested in computing an economic QR decomposition for the concatenation of $p$ block vectors
\[
    \bXX = \begin{bmatrix} \vX_1 & \vX_2 & \cdots \vX_p \end{bmatrix} \in \spR^{m \times ps}
\]
with $m\geq ps$.
We achieve this via a BGS procedure that takes $\bXX$ and a block size $s$ as arguments and returns an orthogonal basis $\bQQ \in \spR^{m \times ps}$, along with an upper triangular $\RR \in \spR^{ps \times ps}$ such that $\bXX = \bQQ \RR$.  Both $\bQQ$ and $\RR$ are computed block-wise, meaning that $s$ new columns of $\bQQ$ are generated per iteration, as opposed to just one column at a time.

In addition to sync points, we are also concerned with the stability of BGS, which we measure here in terms of the \emph{loss of orthogonality (LOO)},
\begin{equation} \label{eq:loo}
    \norm{I - \bQQbar^T \bQQbar},
\end{equation}
where $I$ is the $ps \times ps$ identity matrix and $\bQQbar \in \spR^{m \times ps}$ denotes a computed basis with floating-point error.  We will also consider the relative residual
\begin{equation} \label{eq:rel_res}
    \frac{\norm{\bXX - \bQQbar \RRbar}}{\norm{\bXX}}
\end{equation}
and relative Cholesky residual,
\begin{equation} \label{eq:rel_chol_res}
    \frac{\norm{\bXX^T \bXX - \RRbar^T \RRbar}}{\norm{\bXX}^2},
\end{equation}
where $\RRbar$ is the computed version of $\RR$ with floating-point error and $\norm{\cdot}$ denotes the induced 2-norm.  The residual \eqref{eq:rel_chol_res} measures how close a BGS method is to correctly computing a Cholesky decomposition of $\bXX^T \bXX$, which can provide insight into the stability pitfalls of a method; see, e.g., \cite{CarLR21, GirLRetal05}.

The ideal BGS would require one sync point per block vector and return $\bQQbar$ and $\RRbar$ such that \eqref{eq:loo}-\eqref{eq:rel_chol_res} are $\bigO{\eps}$, where $\eps$ denotes the \emph{unit roundoff}, without any conditions on $\bXX$, except perhaps that the 2-norm condition number $\kappa(\bXX)$ is no larger than $\bigO{\eps^\inv}$.  To the best of our knowledge, no such BGS method exists, and one must make trade-offs regarding the number of sync points and stability.  In practice, the acceptable level of the LOO is often well above machine precision; e.g., the Generalized Minimal Residual (GMRES) method is known to be backward stable with Arnoldi based on modified Gram-Schmidt (\MGS), whose LOO depends linearly on the condition number of $\bXX$ \cite{GreRS97}.  A similar result for block GMRES remains open, however \cite{ButHMetal24}.

A key issue affecting the stability of a BGS method is the choice of \emph{intraorthogonalization} routine, or the so-called ``local" QR factorization of a single block column; in the pseudocode throughout this manuscript, we denote this routine as ``intraortho" or simply \IOnoarg.  Traditional Householder QR (\HouseQR) or Givens QR (\GivensQR) are common choices due to their unconditional $\bigO{\eps}$ LOO, but they introduce additional sync points \cite{Hig02, GolV13}.  One-sync variants include, e.g., Tall-Skinny QR (\TSQR \cite{DemGHetal12, DemGGetal15}), also known as \texttt{AllReduceQR} \cite{MorYZ12}, and \CholQR \cite{YamNYetal15}.  \TSQR /\texttt{AllReduceQR} is known to have $\bigO{\eps}$ LOO, while that of \CholQR is bounded by $\bigO{\eps} \kappa^2(\vX)$, where $\kappa^2(\vX)$ itself should be bounded by $\bigO{\eps^\ihalf}$.

In this manuscript, we focus on reorthogonalized variants of block classical Gram-Schmidt (\BCGS).  We begin by proving LOO bounds on \BCGS, which has not been done rigorously before to the best of our knowledge (Section~\ref{sec:bcgs_iro_a}).  A key feature of our analysis is an abstract framework that highlights the effects of the projection and intraortho stages.  Furthermore, we consider a variant of \BCGS that only requires a ``strong" first step (\BCGSA).  Although this modification does not improve the numerical behavior of \BCGS, its reorthogonalized variant \BCGSIROA enjoys $\bigO{\eps}$ LOO with more relaxed assumptions on the subsequent \IOnoargs than those of Barlow and Smoktunowicz \cite{BarS13} and Barlow \cite{Bar24}(\BCGSIRO), and lower communication cost.  In Section~\ref{sec:roadmap}, we derive a BGS method with one sync point from \BCGSIROA, which is similar to the one-sync variant \BCGSIROLS (\cite[Algorithm~7]{CarLRetal22} and \cite[Figure~3]{YamTHetal20}), and is a block analogue of DCGS2 and CGS-2 with Normalization and Reorthogonalization Lags (\cite[Algorithm~2]{BieLTetal22} and \cite[Algorithm~3]{SwiLAetal21}, respectively).  Unlike \cite{BieLTetal22, SwiLAetal21}, we do not use the notion of lags or delays; instead, we view things in terms of shifting the window of the for-loop, which simplifies the mathematical analysis.  We derive the one-sync algorithm in three stages-- \BCGSIROAthree, \BCGSIROAtwo, and \BCGSIROAone-- in order to systematically demonstrate how new floating-point error is introduced with the successive removal of sync points.  We then prove stability bounds for these algorithms in Section~\ref{sec:ls_proofs}.  A summary and discussion of all the bounds is provided in Section~\ref{sec:recap}, along with an important corollary: \BCGSIROAone achieves $\bigO{\eps}$ LOO for $s=1$.  In other words, \cite[Algorithm~2]{BieLTetal22} and \cite[Algorithm~3]{SwiLAetal21} are effectively as stable as \HouseQR, which thus far has not been rigorously proven.  Section~\ref{sec:conclusions} concludes our work with an outlook for future directions.

A few remarks regarding notation are necessary before proceeding.  Generally, uppercase Roman letters ($R_{ij}, S_{ij}, T_{ij}$) denote $s \times s$ block entries of a $ps \times ps$ matrix, which itself is usually denoted by uppercase Roman script ($\RR, \SS, \TT$).  A block column of such matrices is denoted with MATLAB indexing:
\[
    \RR_{1:k-1,k} = \begin{bmatrix}
        R_{1,k} \\ R_{2,k} \\ \vdots \\ R_{k-1,k}
    \end{bmatrix}.
\]
For simplicity, we also abbreviate standard $ks \times ks$ submatrices as $\RR_k := \RR_{1:k,1:k}$.

Bold uppercase Roman letters ($\vQ_k$, $\vX_k$, $\vU_k$) denote $m \times s$ block vectors, and bold, uppercase Roman script ($\bQQ, \bXX, \bUU$) denotes an indexed concatenation of $p$ such vectors.  Standard $m \times ks$ submatrices are abbreviated as
\[
    \bQQ_k := \bQQ_{1:k} =
    \begin{bmatrix}
        \vQ_1 & \vQ_2 & \cdots & \vQ_k
    \end{bmatrix}.
\]

We will aim for bounds in terms of the induced 2-norm $\norm{\cdot}$, but we will also make use of the Frobenius norm in the analysis, denoted as $\normF{\cdot}$.  By $\tr{A}$ we denote the trace of a square matrix $A$. Furthermore, we always take $\kappa(A)$ to mean the 2-norm condition number defined as the ratio between the largest and smallest singular values of $A$.
\section{Improved stability of \texttt{BCGS} with inner reorthgonalization} \label{sec:bcgs_iro_a}
It is well known that \BCGS is itself low-sync, in the sense that only 2 sync points are required per block vector, assuming that we only employ \IOnoargs that themselves require just one sync point (e.g., \TSQR or \CholQR).  In comparison to block modified Gram-Schmidt (\BMGS), which requires $k+1$ sync points for the $k$th iteration, a fixed number of sync points per iteration can have clear performance benefits; see, e.g., \cite{BieLTetal22, Lun23, SwiLAetal21, YamTHetal20}.

We will consider a slight modification to \BCGS in the first step.  We call this variant \BCGSA, as in Algorithm~\ref{alg:BCGSA}, where ``A" here stands for ``alpha" or the German ``Anfang", as the key change is made at the beginning of the algorithm.  The idea is to require that $\IOAnoarg$ be strongly stable like \HouseQR and then allow for more flexibility in the $\IOnoarg$ used in the for-loop.

\begin{algorithm}[htbp!]
	\caption{$[\bQQ, \RR] = \BCGSA(\bXX, \IOAnoarg, \IOnoarg)$ \label{alg:BCGSA}}
	\begin{algorithmic}[1]
		\State{$[\vQ_1, R_{11}] = \IOA{\vX_1}$}
		\For{$k = 2, \ldots,p$}
		    \State{$\RR_{1:k-1,k} = \bQQ_{k-1}^T \vX_k$} \label{line:bcgs-proj-begin}
		    \State{$[\vQ_k, R_{kk}] = \IO{\vX_k - \bQQ_{k-1} \RR_{1:k-1,k}}$} \label{line:bcgs-proj-end}
		\EndFor
		\State \Return{$\bQQ = [\vQ_1, \ldots, \vQ_p]$, $\RR = (R_{ij})$}
	\end{algorithmic}
\end{algorithm}

Unfortunately, \BCGS (i.e., \BCGSA with $\IOAnoarg = \IOnoarg$) can exhibit an LOO worse than $\bigO{\eps} \kappa^2(\bXX)$, and the situation is no better for \BCGSA.  A natural solution often seen in practice is to run \BCGS twice and combine the for-loops, leading to what we call \BCGSIRO, with \texttt{I+} standing for ``inner reorthogonalization"; see Algorithm~\ref{alg:BCGSIROA} but assume $\IOAnoarg = \IOonenoarg = \IOtwonoarg$.  \BCGSIRO has 4 sync points per iteration and has been analyzed by Barlow and Smoktunowicz \cite{BarS13}, who show that as long as the $\IOnoarg$ has $\bigO{\eps}$ LOO, then the overall method also has $\bigO{\eps}$ LOO.

Figure~\ref{fig:roadmap_1} provides a comparison among \BCGS, \BCGSA and \BCGSIRO for different choices of \IOnoargs.  All plots in this section are generated by the MATLAB toolbox \texttt{BlockStab}\footnote{\url{https://github.com/katlund/BlockStab/releases/tag/v2.1.2024}} in double precision ($\eps \approx 10^{-16}$) on what are called \monomial\footnote{Each such matrix is a concatenation of $r$ block vectors $\vX_k = \bmat{\vv_k & A\vv_k & \cdots & A^{t-1} \vv_k}$,\\$k \in \{1, \ldots, r\}$, where each $\vv_k$ is is randomly generated from the uniform distribution with norm $1$, while $A$ is an $m \times m$ diagonal operator having evenly distributed eigenvalues in $(0.1, 10)$.  A sequence of such matrices with growing condition number is generated by varying $r$ and $t$; in particular, $rt = ps$, but it is not necessary that $r = p$ or $t = s$.} matrices; see the script \texttt{test\_roadmap.m}.  We have used MATLAB 2024a on a Dell laptop running Windows 10 with an 12th Gen Intel Core i7-1270P processor and 32GB of RAM.  We use notation like $\texttt{BGS} \circ \IOnoarg$ to denote the composition of the outer block ``skeleton" and intraorthogonalizing ``muscle".  We have fixed $\IOAnoarg = \HouseQR$ for all ``A" methods, and only the choice of $\IOnoarg = \IOonenoarg = \IOtwonoarg$ is reported in the legends.  Note that \CholQR is implemented without a fail-safe for violating positive definiteness.\footnote{See the \texttt{chol\_free} subroutine in \texttt{BlockStab}, based on \cite[Algorithm~10.2]{Hig02}.}

\begin{figure}[htbp!]
	\begin{center}
	    \begin{tabular}{cc}
	         \resizebox{.35\textwidth}{!}{\includegraphics[trim={0 0 180pt 0},clip]{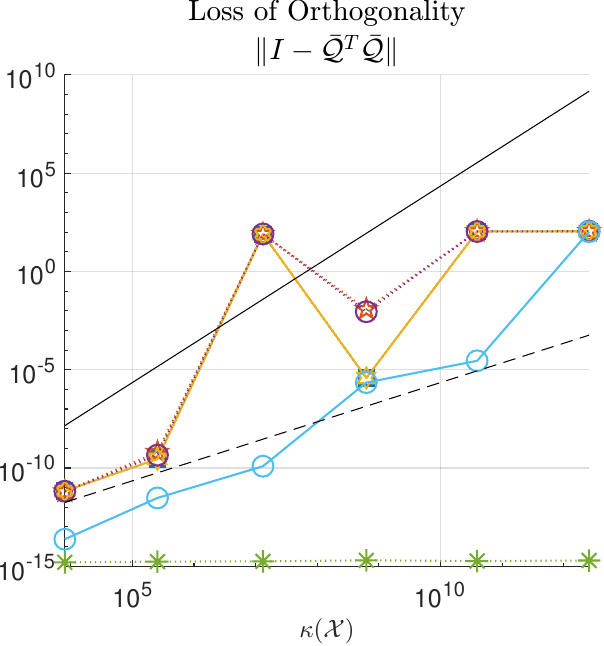}} &
	         \resizebox{.57\textwidth}{!}{\includegraphics{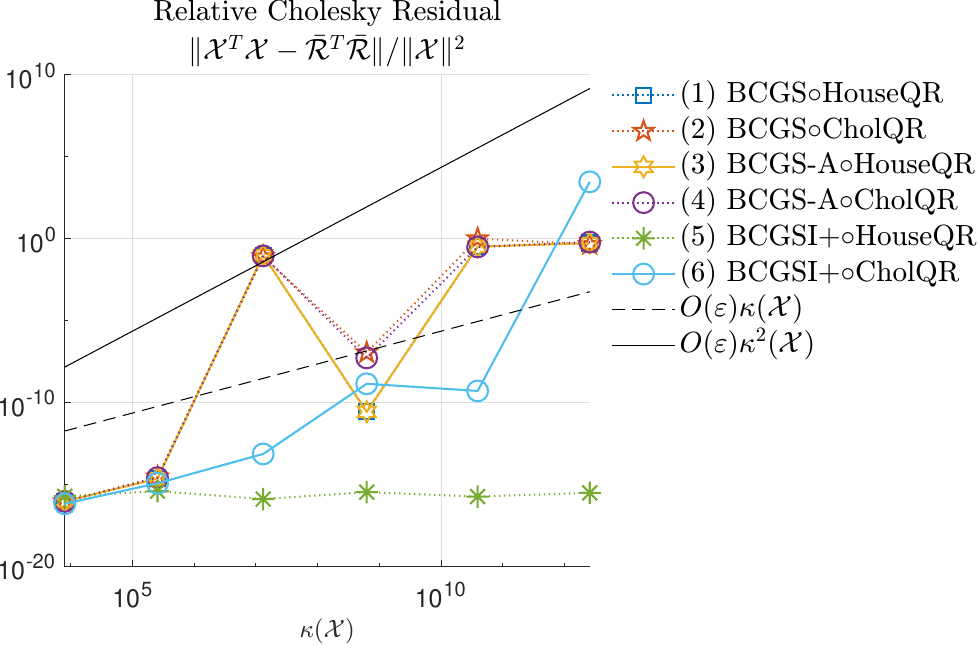}}
	    \end{tabular}
	\end{center}
	\caption{Comparison among \BCGS, \BCGSA, and \BCGSIRO on a class of \monomial matrices from \texttt{BlockStab}. \label{fig:roadmap_1}}
\end{figure}

It is clear from Figure~\ref{fig:roadmap_1} that \BCGSIRO does not exhibit $\bigO{\eps}$ LOO for $\IOnoarg = \CholQR$.  However, by requiring the first block vector to be orthogonalized by something as stable as \HouseQR, we can relax the requirement for subsequent \IOnoargs and prove a stronger result than in \cite{BarS13}.  We denote this modified algorithm \BCGSIROA, given as Algorithm~\ref{alg:BCGSIROA}.  \BCGSIROA can also be interpreted as a generalization of the \BCGSIROF approach introduced in \cite{Zou23}.

The way Algorithm~\ref{alg:BCGSIROA} is written, it would appear that we store three auxiliary matrices $\SS, \TT \in \spR^{ps \times ps}$ and $\bUU \in \spR^{m \times ps}$, where $\SS_{ij} = (S_{ij})$, $\TT_{ij} = (T_{ij})$, and $\bUU = \bmat{\vU_1 & \vU_2 & \cdots & \vU_p}$.  In practice, the entire matrices need not be stored and built, but their theoretical construction is helpful for proving stability bounds.  Further note the three different colors used for different sections of the algorithm: \first{blue} for the first BGS step, \second{red} for the second, and \combo{purple} for combining quantities from each step to finalize entries of $\RR$.  These colors may help some readers in Section~\ref{sec:roadmap} when we derive variants with fewer sync points.

\begin{algorithm}[htbp!]
	\caption{$[\bQQ, \RR] = \BCGSIROA(\bXX, \IOAnoarg, \IOonenoarg, \IOtwonoarg)$ \label{alg:BCGSIROA}}
	\begin{algorithmic}[1]
		\State{$[\vQ_1, R_{11}] = \IOA{\vX_1}$} 
		\For{$k = 2, \ldots,p$}
		    \State \first{$\SS_{1:k-1,k} = \bQQ_{k-1}^T \vX_k$}  \SComment{step k.1.1 -- first projection} \label{BCGSIROAS}
		    \State \first{$[\vU_k, S_{kk}] = \IOone{\vX_k - \bQQ_{k-1} \SS_{1:k-1,k}}$} \SComment{step k.1.2 -- first intraortho} \label{BCGSIROA0}
		    \State \second{$\TT_{1:k-1,k} = \bQQ_{k-1}^T \vU_k$}  \SComment{step k.2.1 -- second projection} \label{BCGSIROA2}
		    \State \second{$[\vQ_k, T_{kk}] = \IOtwo{\vU_k - \bQQ_{k-1} \TT_{1:k-1,k}}$} \SComment{step k.2.2 -- second intraortho} \label{BCGSIROA1}
		    \State \combo{$\RR_{1:k-1,k} = \SS_{1:k-1,k} + \TT_{1:k-1,k} S_{kk}$} \SComment{step k.3.1 -- form upper $\RR$ column}
		    \State \combo{$R_{kk} = T_{kk} S_{kk}$} \SComment{step k.3.2 -- form $\RR$ diagonal entry}
		\EndFor
		\State \Return{$\bQQ = [\vQ_1, \ldots, \vQ_p]$, $\RR = (R_{ij})$}
	\end{algorithmic}
\end{algorithm}

Figure~\ref{fig:roadmap_2} demonstrates the improved behavior of \BCGSIROA relative to \BCGSIRO.  In particular, note that the relative Cholesky residual is restored to $\bigO{\eps}$ for $\IOnoarg = \CholQR$, indicating that \BCGSIROA returns a reliable Cholesky factor $\RR$ for a wider range of \IOnoargs than \BCGSIRO.  Practically speaking, \BCGSIROA only needs an expensive (but stable) \IOnoarg once at the beginning, and less expensive (and even less stable) \IOnoargs for the remaining iterations.

\begin{figure}[htbp!]
	\begin{center}
	    \begin{tabular}{cc}
	         \resizebox{.355\textwidth}{!}{\includegraphics[trim={0 0 190pt 0},clip]{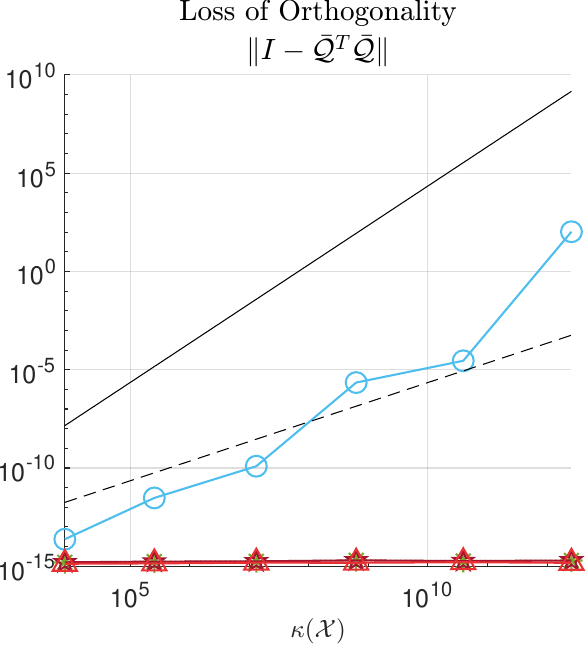}} &
	         \resizebox{.58\textwidth}{!}{\includegraphics{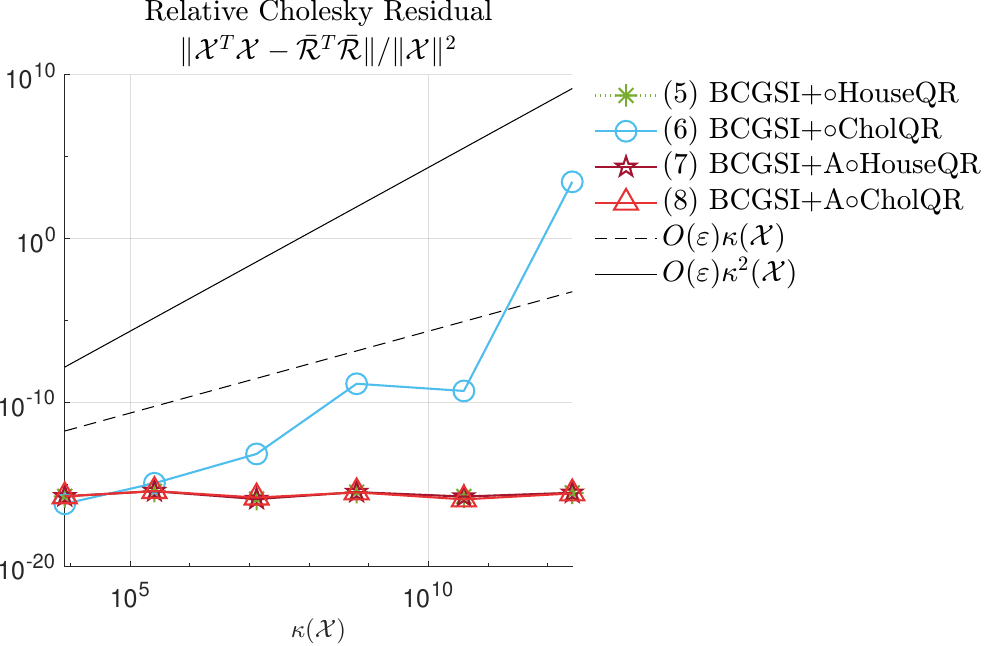}}
	    \end{tabular}
	\end{center}
	\caption{Comparison between \BCGSIRO (i.e., Algorithm~\ref{alg:BCGSIROA} with all \IOnoargs equal) and \BCGSIROA ($\IOAnoarg = \HouseQR$ and $\IOonenoarg = \IOtwonoarg$) on a class of \monomial matrices. \label{fig:roadmap_2}}
\end{figure}

In the following subsections, we introduce an abstract framework for handling the stability analysis of a general BGS routine by splitting it into projection and intraortho stages.  We then prove bounds for \BCGSA and \BCGSIROA encompassing a wide variety of configurations.

\subsection{An abstract framework for block Gram-Schmidt} \label{sec:framework}
For a general BGS procedure, given $\bQQprev$ satisfying $\bQQprev^T\bQQprev = I$ and $\vX$, each iteration aims to compute an orthogonal basis $\vQ$ of the next block vector $\vX$ that satisfies $\bQQnew^T\bQQnew = I$ with $\bQQnew = [\bQQprev, \vQ ]$. Furthermore, the iteration can be divided into two stages:
\begin{equation}
    \begin{split}
        & \text{projection stage:} \quad \vG = \Proj{\vX, \bQQprev}; \mbox{ and}\\
        & \text{intraortho stage:} \quad [\vQ, R] = \QR{\vG}.
    \end{split}
\end{equation}
In practice, $\Proj{\cdot}$ represents an algorithmic choice that has different rounding error consequences.  Generally, the intraortho stage will be identical to the choice(s) of \IOnoarg\footnote{Indeed, the intraortho stage could be replaced by a more general factorization than one that returns an orthogonal basis; see, e.g., \cite{Zou23}.}.

Let $\vGtil$ denote the exact result of $\Proj{\vX, \bQQbarprev}$, where $\bQQbarprev$ is the computed version of $\bQQprev$ from a given algorithm.  Taking rounding errors into account, the computed quantity of the projection stage would then satisfy
\begin{equation} \label{eq:epsproj}
    \vGbar = \vGtil + \vEproj, \quad \norm{\vEproj} \leq \epsproj,
\end{equation}
where $\epsproj > 0$ is related to the unit roundoff $\eps$, the condition number $\kappa(\vX)$, and the definition of \texttt{Proj} itself.  Meanwhile, the computed quantities of the intraortho stage satisfy
\begin{equation} \label{eq:epsqr}
    \begin{split}
        & \vGbar + \Delta \vG = \vQbar \Rbar, \quad \norm{\Delta \vG} \leq \epsqr \norm{\vGbar}, \\
        & \norm{I - \vQbar^T \vQbar} \leq \epsQ,
    \end{split}
\end{equation}
where $\epsqr \in (0,1]$ and $\epsQ < 1$ depend on the \IOnoarg.  For example, \eqref{eq:epsqr} holds for $\HouseQR$ and $\GivensQR$, which satisfy
\begin{equation} \label{eq:epsqr-1}
    \begin{split}
        & \vGbar + \Delta \vGtil = \vQtil \Rbar, \quad \norm{\Delta \vGtil} \leq \bigO{\eps} \norm{\vGbar}, \\
        & \vQbar = \vQtil+ \Delta\vQ,
        \quad \norm{\Delta\vQ} \leq \bigO{\eps},
    \end{split}
\end{equation}
with $\vQtil$ exactly satisfying $\vQtil^T \vQtil = I$, because from~\eqref{eq:epsqr-1} we have
\begin{equation}
    \begin{split}
        & \vGbar + \Delta \vGtil + \Delta\vQ \Rbar = \vQbar \Rbar, \\
        & \norm{\Rbar} \leq (1 + \bigO{\eps}) \norm{\vGbar} + \bigO{\eps} \norm{\Rbar} \Rightarrow \norm{\Rbar} \leq \frac{1 + \bigO{\eps}}{1 - \bigO{\eps}} \norm{\vGbar}.
    \end{split}
\end{equation}
Furthermore, $\HouseQR$ and $\GivensQR$ satisfy
\begin{equation}
    \begin{split}
        & \vGbar + \Delta \vG = \vQbar \Rbar, \quad \norm{\Delta \vG} \leq  \bigO{\eps} \norm{\vGbar}, \\
        & \norm{I - \vQbar^T \vQbar} \leq \bigO{\eps}.
    \end{split}
\end{equation}

From~\cite[Theorems~19.4, 19.10, 19.13]{Hig02} and the discussion above, we show example values of $\epsqr$ and $\epsQ$ for different methods in Table~\ref{tab:muscles}.
\begin{table}[!tbp]
    \centering
    \begin{tabular}{ccc} \hline
        Methods     & $\epsqr$      & $\epsQ$ \\ \hline
        $\HouseQR$  & $\bigO{\eps}$ & $\bigO{\eps}$ \\
        $\GivensQR$ & $\bigO{\eps}$ & $\bigO{\eps}$ \\
        $\MGS$      & $\bigO{\eps}$ & $\bigO{\eps} \kappa$ \\
        $\CholQR$   & $\bigO{\eps}$ & $\bigO{\eps} \kappa^2$ \\ \hline
    \end{tabular}
    \caption{Values of $\epsqr$ and $\epsQ$ for common \IOnoarg choices. \label{tab:muscles}}
\end{table}

In the remainder of this section, we abstract the induction process appearing in a typical BGS analysis through a series of lemmas.  Given $\bQQbarprev$ satisfying
\begin{equation} \label{eq:epsQkp}
    \norm{I - \bQQbarprev^T \bQQbarprev} \leq \epsQone,
\end{equation}
for some $0 < \epsQone \leq 1$, it follows that
\begin{equation} \label{eq:Qkp-2norm}
    \norm{\bQQbarprev} \leq 1 + \epsQone \leq \bigO{1}.
\end{equation}
We aim to derive a $\epsQtwo$ in relation to $\epsQone$, $\epsproj$, $\epsqr$, and $\epsQ$ such that 
\[
    \norm{I - \bQQbarnew^T \bQQbarnew} \leq \epsQtwo,
\]
and then bound the LOO of the final result of the iteration by induction.  The following lemma addresses the impact of $\epsQone$, $\epsproj$, $\epsqr$, and $\epsQ$ on $\epsQtwo$.

\begin{lemma} \label{lem:epsQk-relation}
    Assume that $\vGbar$, $\vGtil$, $\vQbar$, $\Rbar$, and $\bQQbarprev$ satisfy~\eqref{eq:epsproj}, \eqref{eq:epsqr}, and~\eqref{eq:epsQkp}, and that 
    \[
        \frac{(1 + \epsqr) \epsproj}{\sigmin(\vGtil)} + \epsqr \kappa(\vGtil) < 1
    \]
    is satisfied. Furthermore, assume that $\Rbar$ is nonsingular.  Then 
    \begin{equation}
        \norm{\Rbar^\inv} \leq \frac{1 + \epsQ}{\sigmin(\vGtil) -(1 + \epsqr) \epsproj- \epsqr \norm{\vGtil}}
    \end{equation}
    and $\bQQbarnew = [\bQQbarprev, \vQbar]$ satisfies
    \begin{equation} \label{eq:lem:epsQk-relation:epsQk-relation}
        \begin{split}
            \norm{I - \bQQbarnew^T \bQQbarnew} \\
            \leq{} & \epsQone + 2 \norm{\bQQbarprev^T \vGtil \Rbar^\inv} + \epsQ\\
            & + \frac{2(1 + \epsQone)(1 + \epsQ) \left(\epsqr \kappa(\vGtil) + (1 + \epsqr) \epsproj/\sigmin(\vGtil) \right)}{1 - (1 + \epsqr) \epsproj/\sigmin(\vGtil) - \epsqr \kappa(\vGtil)}.
        \end{split}
    \end{equation}
\end{lemma}

Lemma~\ref{lem:epsQk-relation} illustrates how \texttt{Proj} and \texttt{QR} influence the LOO of $\bQQbarnew$.  It also implies that we only need to estimate $\norm{\bQQbarprev^T \vGtil \Rbar^\inv}$, $\epsproj$, $\epsqr$, and $\epsQ$ to assess the LOO of $\bQQbarnew$.

Before proving Lemma~\ref{lem:epsQk-relation}, we first give the following lemma to be used in its proof.

\begin{lemma} \label{lem:norm-R}
    Assume that for $\vW$, $\vU$, $\Delta \vW \in \spR^{m \times s}$ and $R \in \spR^{s \times s}$ nonsingular,
    \[
    \vW + \Delta \vW = \vU R.
    \]
    Then
    \begin{align}
        \norm{R^\inv} \leq \frac{\norm{\vU}}{\sigmin(\vW) - \norm{\Delta \vW}}.
    \end{align}
\end{lemma}

\begin{proof}
    By the perturbation theory of singular values~\cite[Corollary~8.6.2]{GolV13}, we have
    \begin{equation*}
        \sigmin(\vW) - \norm{\Delta \vW} \leq \sigmin(\vU R).
    \end{equation*}
    Together with
    \[
        \sigmin^2(\vU R)= \min_{\vy} \left(\frac{\norm{\vU R \vy}^2}{\norm{\vy}^2} \right)
        \leq \min_{\vy} \left(\frac{\norm{\vU}^2 \norm{R \vy}^2}{\norm{\vy}^2} \right)
        \leq \norm{\vU}^2 \sigmin^2(R),
    \]
    we can bound $\norm{R^\inv}$ as
    \[
        \norm{R^\inv}
        = \frac{1}{\sigmin(R)}
        \leq \frac{\norm{\vU}}{\sigmin(\vW) - \norm{\Delta \vW}}.
    \]
\end{proof}

Note that we do not assume that $\vU$ in Lemma~\ref{lem:norm-R} is orthogonal.

\begin{proof}[Proof of Lemma~\ref{lem:epsQk-relation}]
    Notice that
    \[
    \norm{I - \bQQbarnew^T \bQQbarnew}
    \leq \norm{I - \bQQbarprev^T \bQQbarprev}
    + 2 \norm{\bQQbarprev^T \vQbar}
    + \norm{I - \vQbar^T \vQbar}.
    \]
    Together with~\eqref{eq:epsproj}, \eqref{eq:epsqr}, \eqref{eq:epsQkp} and~\eqref{eq:Qkp-2norm}, we obtain
    \begin{equation} \label{eq:norm-Qk-1TQk}
        \begin{split}
            \norm{\bQQbarprev^T \vQbar}
            & = \norm{\bQQbarprev^T (\vGtil + \vEproj + \Delta \vG) \Rbar^\inv} \\
            & \leq \norm{\bQQbarprev^T \vGtil \Rbar^\inv} + (1 + \epsQone) \epsproj\norm{\Rbar^\inv}\\
            & \quad + (1 + \epsQone) \epsqr \norm{\vGbar} \norm{\Rbar^\inv},
        \end{split}
    \end{equation}
    and furthermore,
    \begin{equation*}
        \begin{split}
            \norm{I - \bQQbarnew^T \bQQbarnew}
            & \leq \epsQone + \epsQ + 2 \norm{\bQQbarprev^T \vGtil \Rbar^\inv} \\
            & \quad + 2(1 + \epsQone) \left(\epsproj + \epsqr \norm{\vGbar} \right) \norm{\Rbar^\inv}.
        \end{split}
    \end{equation*}

    Next, we estimate $\norm{\Rbar^\inv}$, which satisfies
    \begin{equation*}
        \vGtil + \vEproj + \Delta\vG = \vQbar \Rbar.
    \end{equation*}
    By Lemma~\ref{lem:norm-R}, $\norm{\Rbar^\inv}$ can be bounded as
    \begin{equation} \label{eq:norm-Rbarinv}
        \begin{split}
            \norm{\Rbar^\inv}
            & \leq \frac{\norm{\vQbar}}{\sigmin(\vGtil) -(\norm{\vEproj} + \norm{\Delta \vG})} \\
            & \leq \frac{1 + \epsQ}{\sigmin(\vGtil) -(1 + \epsqr) \epsproj- \epsqr \norm{\vGtil}}.
        \end{split}
    \end{equation}
    The conclusion follows because
    \begin{equation*}
        \begin{split}
            \norm{\vGbar} \norm{\Rbar^\inv}
            & \leq \frac{(1 + \epsQ) \left(\norm{\vGtil} + \epsproj\right)}{\sigmin(\vGtil) -(1 + \epsqr) \epsproj- \epsqr \norm{\vGtil}} \\
            & \leq \frac{(1 + \epsQ) \left(\kappa(\vGtil)+ \epsproj/\sigmin(\vGtil) \right)}{1 -(1 + \epsqr) \epsproj/\sigmin(\vGtil) - \epsqr \kappa(\vGtil)}.
        \end{split}
    \end{equation*}
\end{proof}

\subsection{Loss of orthogonality of \texttt{BCGS-A}} \label{sec:BCGSA}
According to Algorithm~\ref{alg:BCGSA}, the projection and intraortho stages can be written respectively as
\begin{equation}
    \vG = \Proj{\vX, \bQQ} := \vX - \bQQ\bQQ^T\vX
    \quad\text{and}\quad \QR{\vG} := \IO{\vG}. 
\end{equation}
Specifically, for the $k$th inner loop of Algorithm~\ref{alg:BCGSA},
\begin{align}
    & \vV_k = \Proj{\vX_k, \bQQ_{k-1}} = \vX_k - \bQQ_{k-1} \bQQ_{k-1}^T \vX_k, \mbox{ and} \label{eq:one-proj} \\ 
    & [\vQ_k, R_{kk}]= \QR{\vV_k} = \IO{\vV_k}. \label{eq:bcgs-orth}
\end{align}
Then we define
\begin{equation} \label{eq:definition-Vtilk}
    \vVtil_k = (I - \bQQbar_{k-1} \bQQbar_{k-1}^T) \vX_k,
\end{equation}
where $\bQQbar_{k-1}$ satisfies
\begin{equation} \label{eq:def-Vtilk-epsQkp}
        \norm{I - \bQQbar_{k-1}^T\bQQbar_{k-1}}\leq \epsQkp.
\end{equation}
Furthermore, $\vVbar_k$ denotes the computed result of $\vVtil_k$.

To use Lemma~\ref{lem:epsQk-relation} to analyze the $k$th inner loop of Algorithm~\ref{alg:BCGSA}, we first estimate $\sigmin(\vVtil_k)$ and $\kappa(\vVtil_k)$ in the following lemma, and then analyze the specific $\epsproj := \epsprojbcgsa_k$ satisfying $\norm{\vVbar_k - \vVtil_k}\leq \epsprojbcgsa_k$, as well as the quantity $\norm{\bQQbar_{k-1}^T \vVtil_k \Rbar_{kk}^\inv}$, which are related only to the projection stage in Lemma~\ref{lem:quantities-one-projection}.

\begin{lemma} \label{lem:norm-Wk0}
    Let $\vVtil_k$ and $\bQQbar_{k-1}$ satisfy~\eqref{eq:definition-Vtilk} and~\eqref{eq:def-Vtilk-epsQkp}.
    Assume that
    \begin{equation}\label{eq:lem-norm-Wk0:assump}
        \bXX_{k-1} + \Delta \bXX_{k-1} = \bQQbar_{k-1} \RRbar_{k-1},
        \quad \norm{\Delta \bXX_{k-1}} \leq \epsXkp \norm{\bXX_{k-1}},
    \end{equation}
    with $\epsXkp \kappa(\bXX_k)<1$ and $\RRbar_{k-1}$ nonsingular.
    Then
    \begin{align}
        & \norm{\vVtil_k} \leq (1 + \epsQkp) \norm{\vX_k}, \label{lem:norm-Wk0:normVtil} \\
        & \sigmin(\vVtil_k) \geq \sigmin(\bXX_k) - \epsXkp \norm{\bXX_{k-1}}, \mbox{ and}\label{lem:norm-Wk0:sigminVtil} \\
        & \kappa(\vVtil_k) \leq \frac{(1 + \epsQkp) \kappa(\bXX_k)}{1 - \epsXkp \kappa(\bXX_k)}. \label{lem:norm-Wk0:condVtil}
    \end{align}
\end{lemma}

\begin{proof}
    Recalling the definition~\eqref{eq:definition-Vtilk} of $\vVtil_k$ and following~\cite[Equations (39)]{CarLMetal24a}, it is easy to verify \eqref{lem:norm-Wk0:sigminVtil}. Note that the two symmetric matrices $\bQQbar_{k-1} \bQQbar_{k-1}^T$ and $\bQQbar_{k-1}^T \bQQbar_{k-1}$ have the same nonzero eigenvalues. From the assumption~\eqref{eq:def-Vtilk-epsQkp},
    \begin{equation} \label{eq:lem:norm-I-QQT}
        \begin{split}
            \norm{I - \bQQbar_{k-1} \bQQbar_{k-1}^T} & = \max\{\norm{I - \bQQbar_{k-1}^T \bQQbar_{k-1}}, 1\}
            \leq \max\{\epsQkp, 1\} \leq 1 + \epsQkp,
        \end{split}
    \end{equation}
    which, combined with the definition~\eqref{eq:definition-Vtilk} of \(\vVtil_k\), gives \eqref{lem:norm-Wk0:normVtil}.  Combining \eqref{lem:norm-Wk0:sigminVtil} with~\eqref{eq:lem:norm-I-QQT}, we have
    \begin{equation} \label{eq:kappa-tildeW-one-proj}
        \kappa(\vVtil_k) = \frac{\norm{\vVtil_k}}{\sigmin(\vVtil_k)}
        \leq \frac{(1 + \epsQkp) \norm{\vX_k}}{\sigmin(\bXX_k) - \epsXkp \norm{\bXX_{k-1}}}
        \leq \frac{(1 + \epsQkp) \kappa(\bXX_k)}{1 - \epsXkp \kappa(\bXX_k)},
    \end{equation}
    giving \eqref{lem:norm-Wk0:condVtil}.
\end{proof}

\begin{lemma} \label{lem:quantities-one-projection}
    Let $\vVtil_k$ and $\bQQbar_{k-1}$ satisfy~\eqref{eq:definition-Vtilk} and~\eqref{eq:def-Vtilk-epsQkp}.
    Then for the projection stage~\eqref{eq:one-proj} computed by lines~\ref{line:bcgs-proj-begin}-\ref{line:bcgs-proj-end} in Algorithm~\ref{alg:BCGSA}, it holds that

        \begin{align}
            & \norm{\vVbar_k - \vVtil_k} \leq \epsprojbcgsa_k \leq \bigO{\eps} \norm{\vX_k},\label{lem:quantities-one-projection:normVbarVtil}\\
            \quad
            &\norm{\bQQbar_{k-1}^T \vVtil_k \Rbar_{kk}^\inv}
            \leq (1 + \epsQkp) \epsQkp \norm{\vX_k} \norm{\Rbar_{kk}^\inv}.\label{lem:quantities-one-projection:normQVR}
        \end{align}
\end{lemma}

\begin{proof}
    By~\eqref{eq:Qkp-2norm} it follows that
    \begin{equation} \label{eq:error-QtX}
        \RRbar_{1:k-1,k} = \bQQbar_{k-1}^T \vX_k+ \Delta\RRbar_{1:k-1,k},
        \quad \norm{\Delta\RRbar_{1:k-1,k}}
        \leq \bigO{\eps} \norm{\vX_k}.
    \end{equation}
    Furthermore, we have
    \begin{equation}\label{eq:Vbar_k}
        \vVbar_k = \vX_k - \bQQbar_{k-1}\bQQbar_{k-1}^T\vX_k + \Delta \vVtil_k = \vVtil_k + \Delta \vVtil_k,
        \quad \norm{\Delta \vVtil_k}\leq \bigO{\eps} \norm{\vX_k},
    \end{equation}
    which gives \eqref{lem:quantities-one-projection:normVbarVtil}.
    Then by the assumption~\eqref{eq:def-Vtilk-epsQkp}, we find the bound \eqref{lem:quantities-one-projection:normQVR} as follows:
    \begin{align*}
        \norm{\bQQbar_{k-1}^T \vVtil_k \Rbar_{kk}^\inv}
        & = \norm{\bQQbar_{k-1}^T \left(I - \bQQbar_{k-1} \bQQbar_{k-1}^T \right) \vX_k \Rbar_{kk}^\inv} \\
        & \leq \norm{\left(I - \bQQbar_{k-1}^T \bQQbar_{k-1} \right) \bQQbar_{k-1}^T \vX_k \Rbar_{kk}^\inv} \\
        & \leq \norm{I - \bQQbar_{k-1}^T \bQQbar_{k-1}}
        \norm{\bQQbar_{k-1}} \norm{\vX_k} \norm{\Rbar_{kk}^\inv} \\
        & \leq(1 + \epsQkp) \epsQkp \norm{\vX_k} \norm{\Rbar_{kk}^\inv}.
    \end{align*}
\end{proof}

Lemma~\ref{lem:epsQk-relation} indicates that we only need to estimate $\norm{\bQQbar_{k-1}^T \vVtil_k \Rbar_{kk}^\inv}$, $\epsprojbcgsa_k$, $\epsqr$, $\sigmin(\vVtil_k)$, and $\kappa(\vVtil_k)$ to evaluate the LOO of $\bQQbar_k$. Already, $\norm{\bQQbar_{k-1}^T \vVtil_k \Rbar_{kk}^\inv}$ and $\epsprojbcgsa_k$ have been determined in Lemma~\ref{lem:quantities-one-projection} together with the estimation of $\norm{\Rbar_{kk}^\inv}$ shown in~\eqref{eq:norm-Rbarinv}, while $\sigmin(\vVtil_k)$ and $\kappa(\vVtil_k)$ were addressed in Lemma~\ref{lem:norm-Wk0}. It is important to note that $\epsqr$ is dependent on $\IOonenoarg$. Therefore, the subsequent lemma utilizes both Lemma~\ref{lem:quantities-one-projection} and Lemma~\ref{lem:norm-Wk0} to describe the behavior of the $k$th inner loop of $\BCGSA$, guided by Lemma~\ref{lem:epsQk-relation}.

\begin{lemma} \label{lem:bcgs-kinnerloop}
    Assume that $\bQQbar_{k-1}$ satisfies~\eqref{eq:def-Vtilk-epsQkp}, and that
    \[
        \bXX_{k-1} + \Delta \bXX_{k-1} = \bQQbar_{k-1} \RRbar_{k-1},
        \quad \norm{\Delta \bXX_{k-1}}
        \leq \epsXkp \norm{\bXX_{k-1}}.
    \]
    Suppose further that for all $\vX \in \spR^{m \times s}$ with $\kappa(\vX) \leq \kappa(\bXX)$, the following hold for $[\vQbar, \Rbar] = \IOA{\vX}$:
    \begin{equation*}
        \begin{split}
            & \vX + \Delta \vX = \vQbar \Rbar,
            \quad \norm{\Delta \vX} \leq \bigO{\eps} \norm{\vX}, \\
            & \norm{I - \vQbar^T \vQbar} \leq \bigO{\eps} \kappa^{\constQA}(\vX),
        \end{split}
    \end{equation*}
    and for $[\vQbar, \Rbar] = \IO{\vX}$:
    \begin{equation*}
        \begin{split}
            & \vX + \Delta \vX = \vQbar \Rbar,
            \quad \norm{\Delta \vX} \leq \bigO{\eps} \norm{\vX}, \\
            & \norm{I - \vQbar^T \vQbar} \leq \bigO{\eps} \kappa^{\constQ}(\vX).
        \end{split}
    \end{equation*}
    Assume as well that $\bigO{\eps} \kappa(\bXX_k) \leq \frac{1}{2}$ and $\epsXkp \kappa(\bXX_k) < 1$. Then for the $k$th inner loop of Algorithm~\ref{alg:BCGSA} with any $k \geq 2$,
    \begin{equation} \label{eq:epsQk-relation}
        \norm{I - \bQQbar_k^T \bQQbar_k}
        \leq 
        \epsQkp + \frac{2 \epsQkp(1 + \epsQkp)(1 + \epsQ)+ \bigO{\eps}}{1 - \bigO{\eps} \kappa(\bXX_k)} \kappa(\bXX_k)+ \epsQ.
    \end{equation}
\end{lemma}

By induction on $k$, we obtain the following result on the LOO of $\BCGSA$.

\begin{theorem} \label{thm:bcgs}
    Let $\bQQbar$ and $\RRbar$ denote the computed results of Algorithm~\ref{alg:BCGSA}. Assume that for all $\vX \in \spR^{m \times s}$ with $\kappa(\vX) \leq \kappa(\bXX)$, the following hold for $[\vQbar, \Rbar] = \IOA{\vX}$:
    \begin{equation*}
        \begin{split}
            & \vX + \Delta \vX = \vQbar \Rbar,
            \quad \norm{\Delta \vX} \leq \bigO{\eps} \norm{\vX}, \\
            & \norm{I - \vQbar^T \vQbar} \leq \bigO{\eps} \kappa^{\constQA}(\vX).
        \end{split}
    \end{equation*}
    Assume as well that for $[\vQbar, \Rbar] = \IO{\vX}$, it holds that
    \begin{equation*}
        \begin{split}
            & \vX + \Delta \vX = \vQbar \Rbar,
            \quad \norm{\Delta \vX} \leq \bigO{\eps} \norm{\vX}, \\
            & \norm{I - \vQbar^T \vQbar} \leq \bigO{\eps} \kappa^{\constQ}(\vX).
        \end{split}
    \end{equation*}
    If $\bigO{\eps} \kappa(\bXX) \leq \frac{1}{2}$,
    then
    \begin{equation} \label{eq:thm:bcgs:res}
         \bXX + \Delta\bXX = \bQQbar \RRbar, \quad \norm{\Delta\bXX} \leq \bigO{\eps} \norm{\bXX}
    \end{equation}
    and
    \begin{equation} \label{eq:thm:bcgs_loo}
        \norm{I - \bQQbar^T \bQQbar} \leq \bigO{\eps} \left(\kappa(\bXX) \right)^{p-2 + \max\{\constQA + 1, \constQ\}}.
    \end{equation}
\end{theorem}

\begin{proof}
    First, we prove a bound on the residual of $\BCGS$ by an inductive approach on the block vectors of $\bXX$.
    For the base case, the assumptions of $\IOAnoarg$ directly give
    \begin{equation*}
        \bXX_{1} + \Delta \bXX_{1} = \bQQbar_{1} \RRbar_1,
        \quad \norm{\Delta \bXX_{1}} \leq \bigO{\eps} \norm{\bXX_{1}}.
    \end{equation*}
    Then we prove that it holds for $k$ provided it holds for $k-1$.
    Assume that $\bXX_{k-1} + \Delta \bXX_{k-1} = \bQQbar_{k-1} \RRbar_{k-1}$ with $\norm{\Delta \bXX_{k-1}} \leq \bigO{\eps} \norm{\bXX_{k-1}}$.
    Then by~\eqref{eq:epsqr}, \eqref{eq:error-QtX}, \eqref{eq:Vbar_k}, and Lemma~\ref{lem:quantities-one-projection},
    \begin{equation*}
        \begin{split}
            \bXX_k + \Delta \bXX_k 
            & = \bmat{\bXX_{k-1} + \Delta \bXX_{k-1}& \vX_k + \Delta\vX_k} \\
            & = \bmat{\bQQbar_{k-1}& \vQbar_k} \bmat{\RRbar_{k-1}& \RRbar_{1:k-1,k} \\0& \Rbar_{kk}} \\
            & = \bmat{\bQQbar_{k-1} \RRbar_{k-1}& \bQQbar_{k-1} \RRbar_{1:k-1,k} + \vQbar_k \Rbar_{kk}} \\
            & = \bmat{\bQQbar_{k-1} \RRbar_{k-1}& \bQQbar_{k-1} \left(\bQQbar_{k-1}^T \vX_k + \Delta\RRbar_{1:k-1,k}\right)
            + \vX_k - \bQQbar_{k-1} \bQQbar_{k-1}^T \vX_k + \Delta \vVtil_k + \Delta \vG_k} \\
            & = \bmat{\bQQbar_{k-1} \RRbar_{k-1}& \vX_k + \bQQbar_{k-1}^T \Delta\RRbar_{1:k-1,k} + \Delta \vVtil_k + \Delta \vG_k}.
        \end{split}
    \end{equation*}
    Thus, $\Delta\vX_k = \bQQbar_{k-1}^T \Delta\RRbar_{1:k-1,k} + \Delta \vVtil_k + \Delta \vG_k$ satisfies $\norm{\Delta\vX_k} \leq \bigO{\eps} \norm{\vX_k}$ and further we have
    \begin{equation}
        \norm{\Delta\bXX_k}
        \leq \norm{\Delta\bXX_{k-1}} + \norm{\Delta\vX_k}
        \leq \bigO{\eps} \norm{\bXX_k},
    \end{equation}
    which proves~\eqref{eq:thm:bcgs:res} by induction.

    Next it remains to prove the LOO~\eqref{eq:thm:bcgs_loo} using Lemma~\ref{lem:bcgs-kinnerloop} via induction.  Note that from~\eqref{eq:thm:bcgs:res} we have already verified that the assumption of the residual in Lemma~\ref{lem:bcgs-kinnerloop} is satisfied.  The assumptions on $\IOAnoarg$ directly give the base case for the LOO, i.e., $\omega_1\leq \bigO{\eps} \kappa^{\constQA}(\bXX_1)\leq \bigO{\eps}\bigl(\kappa(\bXX_1)\bigr)^{\max\{\constQA + 1, \constQ\} - 1}$.
    Now we assume that $\epsQkp\leq \bigO{\eps}\bigl(\kappa(\bXX_{k-1})\bigr)^{\max\{\constQA + 1, \constQ\} + k - 3}$ and then bound \(\epsQk\).  Using Lemma~\ref{lem:bcgs-kinnerloop} and the assumption of $\IOnoarg$ we conclude the proof because
    \begin{align*}
        \norm{I - \bQQbar_k^T \bQQbar_k}
        & \leq \bigO{\eps}\bigl(\kappa(\bXX_{k-1})\bigr)^{\max\{\constQA + 1, \constQ\} + k - 2} + \bigO{\eps} \kappa^{\constQ}(\vX_k) \\
        & \leq \bigO{\eps}\bigl(\kappa(\bXX_k)\bigr)^{\max\{\constQA + 1, \constQ\} + k - 2}.
    \end{align*}
\end{proof}

If $\epsQkp$, $\epsQ\leq \bigO{1}$ and $\epsQ\leq \epsQkp$ for some $k-1$, we can simplify~\eqref{eq:epsQk-relation} in Lemma~\ref{lem:bcgs-kinnerloop} as
\begin{equation*}
    \norm{I - \bQQbar_k^T \bQQbar_k}
    \leq \bigO{1}\norm{I - \bQQbar_{k-1}^T \bQQbar_{k-1}}\kappa(\bXX_k).
\end{equation*}
In other words, $\IOAnoarg$ and \IOnoarg have essentially no effect from one iteration to the next on the orthogonality of the basis.  Theorem~\ref{thm:bcgs} makes this observation more precise.  In particular, \eqref{eq:thm:bcgs_loo} shows that in the best of circumstances, $\kappa(\bXX)$ still has an exponent of at least $p-1$ in the bound, which is the same as what Kie{\l}basi{\'n}ski and Schwetlick proved for column-wise \texttt{CGS} \cite[pp~284]{KieS88} (i.e., for $s=1$.  Table \ref{tab:LOO-bcgsa} compares provable LOO bounds for various choices of $\IOAnoarg$ and \IOnoarg.  Although we have been unable to find numerical examples that attain this bound, Figure~\ref{fig:roadmap_1} clearly shows that $p > 2$, even when $\IOAnoarg = \IOnoarg = \HouseQR$.

\begin{table}[!htbp]
    \centering
    \begin{tabular}{ccccc} \hline
        $\IOAnoarg$     & $\IOnoarg$        & $\norm{I - \bQQbar^T \bQQbar}$ \\ \hline
        \HouseQR        & \HouseQR or \MGS  & $\bigO{\eps}(\kappa(\bXX))^{p-1}$ \\
        \HouseQR or \MGS& \CholQR           & $\bigO{\eps}(\kappa(\bXX))^{p}$ \\
        \CholQR         & \CholQR           & $\bigO{\eps}(\kappa(\bXX))^{p+1}$ \\ \hline
    \end{tabular}
    \caption{Upper bounds on the loss of orthogonality for $\BCGSA$ with different $\IOnoarg$ analyzed in Theorem~\ref{thm:bcgs}. \label{tab:LOO-bcgsa}}
\end{table}

\subsection{Loss of orthogonality of \texttt{BCGSI+A}} \label{sec:BCGSIROA}
For Algorithm~\ref{alg:BCGSIROA}, the projection and intraortho stages can be written respectively as follows:
\begin{equation}
    \vG = \Proj{\vX, \bQQ} := (I - \bQQ \bQQ^T)(I - \bQQ \bQQ^T) \vX \mbox{ and }
    \QR{\vG} := \IOtwo{\vG}. 
\end{equation}
Specifically, for the $k$th inner loop of Algorithm~\ref{alg:BCGSIROA},
\begin{align}
    & \vH_k = \Proj{\vX_kS_{kk}^\inv, \bQQ_{k-1}} = (I - \bQQ_{k-1} \bQQ_{k-1}^T)(I - \bQQ_{k-1} \bQQ_{k-1}^T) \vX_kS_{kk}^\inv, \label{eq:two-proj+one-orth-proj} \\
    &[\vQ_k, R_{kk}]= \QR{\vH_k} = \IOtwo{\vH_k}, \label{eq:bcgsiroa3-orth}
\end{align}
where $\vU_k S_{kk} = \IOone{(I - \bQQ_{k-1} \bQQ_{k-1}^T) \vX_k}$.

\begin{lemma} \label{lem:bcgsiro-property}
    Let $\vVtil_k = (I - \bQQbar_{k-1} \bQQbar_{k-1}^T) \vX_k$ and $\vHtil_k = (I - \bQQbar_{k-1} \bQQbar_{k-1}^T) \vUbar_k$, where $\bQQbar_{k-1}$ satisfies $\norm{I - \bQQbar_{k-1}^T\bQQbar_{k-1}}\leq \epsQkp$.
    Assume that for all $\vX \in \spR^{m \times s}$ with $\kappa(\vX) \leq \kappa(\bXX)$, the following hold for $[\vQbar, \Rbar] = \IOone{\vX}$:
    \begin{equation*}
        \begin{split}
            & \vX + \Delta \vX = \vQbar \Rbar,
            \quad \norm{\Delta \vX} \leq \bigO{\eps} \norm{\vX}, \\
            & \norm{I - \vQbar^T \vQbar} \leq \bigO{\eps} \kappa^{\alpha_1}(\vX).
        \end{split}
    \end{equation*}
    Similarly, assume the following hold for $[\vQbar, \Rbar] = \IOtwo{\vX}$:
    \begin{equation*}
        \begin{split}
            & \vX + \Delta \vX = \vQbar \Rbar,
            \quad \norm{\Delta \vX} \leq \bigO{\eps} \norm{\vX}, \\
            & \norm{I - \vQbar^T \vQbar} \leq \bigO{\eps} \kappa^{\alpha_2}(\vX).
        \end{split}
    \end{equation*}
    Furthermore, if
    \begin{equation}
        \begin{split}
            & \bXX_{k-1} + \Delta \bXX_{k-1} = \bQQbar_{k-1} \RRbar_{k-1},
            \quad \norm{\Delta \bXX_{k-1}}
            \leq \epsXkp \norm{\bXX_{k-1}}
        \end{split}
    \end{equation}
    as well as
    \begin{equation}\label{eq:lem-bcgsiro-property:assump}
        \bigO{\eps} \kappa^{\alpha_1}(\bXX_k)+2 \left(\epsQkp + \bigO{\eps} \right)(1 + \epsQkp)^2 \kappa(\bXX_k) + \epsXkp \kappa(\bXX_k) \leq \frac{1}{2},
    \end{equation}
    then the quantities computed by Algorithm~\ref{alg:BCGSIROA} satisfy the following, for any $k \geq 2$:
    \begin{align}
        & \SSbar_{1:k-1,k} = \bQQbar_{k-1}^T \vX_k + \Delta\SSbar_{1:k-1,k},
        \quad \norm{\Delta\SSbar_{1:k-1,k}} \leq \bigO{\eps} \norm{\vX_k}; \label{eq:lem-bcgsiro-property-SS} \\
        & \vUbar_k \Sbar_{kk} = \vVtil_k + \Delta\vVtil_k,
        \quad \norm{\Delta\vVtil_k} \leq \bigO{\eps} \norm{\vX_k}; \label{eq:lem-bcgsiro-property-orth1} \\
        & \norm{I - \vUbar_k^T \vUbar_k} \leq \bigO{\eps} \kappa^{\alpha_1}(\bXX_k),
        \quad \norm{\vUbar_k} \leq 1 + \bigO{\eps} \kappa^{\alpha_1}(\bXX_k); \label{eq:lem-bcgsiro-property-lossUk} \\
        & \norm{\vX_k} \norm{\Sbar_{kk}^\inv} \leq 2 \kappa(\bXX_k); \\
        & \bar{\TT}_{1:k-1,k} = \bQQbar_{k-1}^T \vUbar_k + \Delta\bar{\TT}_{1:k-1,k}, \label{eq:lem-bcgsiro-property-ETT} \\
        & \quad \norm{\Delta\bar{\TT}_{1:k-1,k}} \leq \bigO{\eps}(1 + \bigO{\eps} \kappa^{\alpha_1}(\bXX_k)); \notag \\
        & \vQbar_k \Tbar_{kk} = \vHtil_k + \Delta\vHtil_k,
        \quad \norm{\Delta\vHtil_k} \leq \bigO{\eps}(1 + \bigO{\eps} \kappa^{\alpha_1}(\bXX_k)); \label{eq:lem-bcgsiro-property-orth2} \\
        & \norm{I - \vQbar_k^T \vQbar_k} \leq \bigO{\eps},
        \quad \norm{\vQbar_k} \leq 1 + \bigO{\eps}; \label{eq:lem-bcgsiro-property-Qkloss} \\
        & \RRbar_{1:k-1,k} = \SSbar_{1:k-1,k} + \bar{\TT}_{1:k-1,k} \Sbar_{kk} + \Delta\RRbar_{1:k-1,k}, \label{eq:lem-bcgsiro-property-ERR} \\
        & \quad \norm{\Delta\RRbar_{1:k-1,k}} \leq \bigO{\eps} \norm{\vX_k}; \mbox{ and} \notag \\
        & \Rbar_{kk} = \Tbar_{kk} \Sbar_{kk} + \Delta\Rbar_{kk},
        \quad \norm{\Delta\Rbar_{kk}} \leq \bigO{\eps} \norm{\vX_k}.
    \end{align}
\end{lemma}

\begin{proof}
    Similarly to~\eqref{eq:error-QtX} and~\eqref{eq:Vbar_k} in the proof of Lemma~\ref{lem:quantities-one-projection}, it is easy to verify~\eqref{eq:lem-bcgsiro-property-SS}, \eqref{eq:lem-bcgsiro-property-orth1}, and
    \begin{align}
        & \bar{\TT}_{1:k-1,k} = \bQQbar_{k-1}^T \vUbar_k + \Delta\bar{\TT}_{1:k-1,k},
        \quad \norm{\Delta\bar{\TT}_{1:k-1,k}} \leq \bigO{\eps} \norm{\vUbar_k}, \label{eq:proof-lem-bcgsiro-property-TT} \\
        & \vQbar_k \Tbar_{kk} = \vHtil_k + \Delta\vHtil_k,
        \quad \norm{\Delta\vHtil_k} \leq \bigO{\eps} \norm{\vUbar_k}. \label{eq:proof-lem-bcgsiro-property-QT}
    \end{align}
    By the assumption on $\IOonenoarg$ and by applying Lemma~\ref{lem:norm-Wk0} and~\eqref{eq:Vbar_k}, we have
    \begin{equation}
        \begin{split}
            \norm{I - \vUbar_k^T \vUbar_k}
            & \leq \bigO{\eps} \kappa^{\alpha_1}(\vVbar_k)
            \leq \bigO{\eps} \kappa^{\alpha_1}(\vVtil_k)
            \leq \bigO{\eps} \kappa^{\alpha_1}(\bXX_k).
        \end{split}
    \end{equation}
    Furthermore, $\norm{\vUbar_k} \leq 1 + \bigO{\eps} \kappa^{\alpha_1}(\bXX_k)$ and $\sigmin(\vUbar_k) \geq 1 - \bigO{\eps} \kappa^{\alpha_1}(\bXX_k)$.  Combining this with~\eqref{eq:proof-lem-bcgsiro-property-TT} and~\eqref{eq:proof-lem-bcgsiro-property-QT}, we bound $\norm{\Delta\bar{\TT}_{1:k-1,k}}$ and $\norm{\Delta\vHtil_k}$ to obtain \eqref{eq:lem-bcgsiro-property-ETT} and \eqref{eq:lem-bcgsiro-property-orth2}.
    
    Next, we aim to bound $\norm{\vX_k} \norm{\Sbar_{kk}^\inv}$ and therefore need to analyze the relationship between $\norm{\Sbar_{kk}^\inv}$ and $\sigmin(\bXX_k)$.  By~\eqref{eq:lem-bcgsiro-property-orth1} and Lemma~\ref{lem:norm-R}, $\norm{\Sbar_{kk}^\inv}$ can be bounded as follows:
    \begin{equation}
        \norm{\Sbar_{kk}^\inv} \leq \frac{\norm{\vUbar_k}}{\sigmin(\vVtil_k) - \norm{\Delta\vVtil_k}}
        \leq \frac{1 + \bigO{\eps} \kappa^{\alpha_1}(\bXX_k)}{\sigmin(\vVtil_k) - \bigO{\eps} \norm{\vX_k}}.
    \end{equation}
    Together with the assumption~\eqref{eq:lem-bcgsiro-property:assump} and the following implication of Lemma~\ref{lem:norm-Wk0},
    \begin{equation}
        \sigmin(\vVtil_k)
        \geq \sigmin(\bXX_k) - \epsXkp \kappa(\bXX_k),
    \end{equation}
    we bound $\norm{\vX_k} \norm{\Sbar_{kk}^\inv}$ as follows:
    \begin{equation}\label{eq:norm-XkSkkinv}
        \begin{split}
            \norm{\vX_k} \norm{\Sbar_{kk}^\inv}
            & \leq \frac{\norm{\vX_k}}{\sigmin(\bXX_k) - \bigO{\eps} \norm{\vX_k}- \epsXkp \norm{\bXX_{k-1}}} \\
            & \leq \frac{\kappa(\bXX_k)}{1 - \bigO{\eps} \kappa(\bXX_k) - \epsXkp \kappa(\bXX_k)} \\
            & \leq 2 \kappa(\bXX_k).
        \end{split}
    \end{equation}
    
    Now we seek a tighter bound on the LOO of $\vQbar_k$.  Recall the definition of $\vHtil_k$ (from this lemma's assumptions), and combine \eqref{eq:lem-bcgsiro-property-lossUk} with the assumption~\eqref{eq:lem-bcgsiro-property:assump} to obtain
    \begin{equation} \label{eq:Hbark}
        \begin{split}
        \vHbar_k &= (I - \bQQbar_{k-1} \bQQbar_{k-1}^T) \vUbar_k + \Delta\vHbar_k = \vHtil_k + \Delta\vHbar_k,\\
        \norm{\Delta\vHbar_k} &\leq \bigO{\eps}\norm{\vUbar_k} \leq \bigO{\eps}.    
        \end{split}
    \end{equation}
    Combining~\eqref{eq:Hbark} with the assumption on $\IOtwonoarg$, we have
    \begin{equation}
        \norm{I - \vQbar_k^T \vQbar_k}
        \leq \bigO{\eps} \kappa^{\alpha_2}(\vHbar_k)
        \leq \bigO{\eps} \kappa^{\alpha_2}(\vHtil_k).
    \end{equation}
    By the definitions of $\vVtil_k$ and $\vHtil_k$, as well as \eqref{eq:lem-bcgsiro-property-orth1}, we find that
    \begin{equation} \label{eq:proof-lem-bcgsiro-property-sigmamintildeW}
        \begin{split}
            & \sigmin(\vHtil_k) \\
            & \quad = \sigmin\left(\vUbar_k - \bQQbar_{k-1} \bQQbar_{k-1}^T(\vVtil_k + \Delta\vVtil_k) \Sbar_{kk}^\inv \right) \\
            & \quad = \sigmin\left(\vUbar_k
            - \bQQbar_{k-1}(\bQQbar_{k-1}^T \bQQbar_{k-1}- I) \bQQbar_{k-1}^T \vX_k \Sbar_{kk}^\inv
            - \bQQbar_{k-1} \bQQbar_{k-1}^T \Delta\vVtil_k \Sbar_{kk}^\inv \right) \\
            & \quad \geq \sigmin\left(\vUbar_k \right)
            - \norm{\bQQbar_{k-1}(\bQQbar_{k-1}^T \bQQbar_{k-1}- I) \bQQbar_{k-1}^T \vX_k \Sbar_{kk}^\inv
            + \bQQbar_{k-1} \bQQbar_{k-1}^T \Delta\vVtil_k \Sbar_{kk}^\inv} \\
            & \quad \geq \sigmin\left(\vUbar_k \right)
            - \left(\epsQkp + \bigO{\eps} \right)(1 + \epsQkp)^2 \norm{\vX_k} \norm{\Sbar_{kk}^\inv} \\
            & \quad \geq \sigmin\left(\vUbar_k \right)
            - 2 \left(\epsQkp + \bigO{\eps} \right)(1 + \epsQkp)^2 \kappa(\bXX_k).
        \end{split}
    \end{equation}
    Furthermore, together with
    \begin{equation} \label{eq:proof-lem-bcgsiro-property-tildeW}
        \norm{\vHtil_k}
        \leq \sqrt{m}\, (1 + \epsQkp) \norm{\vUbar_k}
        \leq \sqrt{m}\, (1 + \epsQkp) (1 + \bigO{\eps} \kappa^{\alpha_1}(\bXX_k)),
    \end{equation}
    it follows that
    \begin{equation}
        \begin{split}
            \norm{I - \vQbar_k^T \vQbar_k}
            & \leq \bigO{\eps} \left(\frac{\sqrt{m}\, (1 + \epsQkp) (1 + \bigO{\eps} \kappa^{\alpha_1}(\bXX_k))}
            {\sigmin\left(\vUbar_k \right)
            - 2 \left(\epsQkp + \bigO{\eps} \right)(1 + \epsQkp)^2 \kappa(\bXX_k)} \right)^{\alpha_2} \\
            & \leq \bigO{\eps} \left(\frac{\sqrt{m}\, (1 + \epsQkp) (1 + \bigO{\eps} \kappa^{\alpha_1}(\bXX_k))}
            {1 - \bigO{\eps} \kappa^{\alpha_1}(\bXX_k)
            - 2 \left(\epsQkp + \bigO{\eps} \right)(1 + \epsQkp)^2 \kappa(\bXX_k)} \right)^{\alpha_2} \\
            & \leq \bigO{\eps}.
        \end{split}
    \end{equation}
    
    A simple floating-point analysis on \eqref{eq:lem-bcgsiro-property-SS}, \eqref{eq:lem-bcgsiro-property-orth1}, and~\eqref{eq:lem-bcgsiro-property-ETT} leads to \eqref{eq:lem-bcgsiro-property-ERR}.  Similarly, the bound on $\norm{\Delta \Rbar_{kk}}$ can be obtained by combining \eqref{eq:lem-bcgsiro-property-SS} and~\eqref{eq:lem-bcgsiro-property-orth1} with \eqref{eq:lem-bcgsiro-property-orth2} and \eqref{eq:proof-lem-bcgsiro-property-tildeW}.
\end{proof}

We are now prepared to analyze the behavior of the $k$th inner loop of $\BCGSIROA$.
\begin{lemma} \label{lem:bcgsiroa-stepk}
    Assume that $\bQQbar_{k-1}$ satisfies $\norm{I - \bQQbar_{k-1}^T\bQQbar_{k-1}}\leq \epsQkp$ and for all $\vX \in \spR^{m \times s}$ with $\kappa(\vX) \leq \kappa(\bXX)$, the following hold for $[\vQbar, \Rbar] = \IOone{\vX}$:
    \begin{equation*}
        \begin{split}
            & \vX + \Delta \vX = \vQbar \Rbar,
            \quad \norm{\Delta \vX} \leq \bigO{\eps} \norm{\vX}, \\
            & \norm{I - \vQbar^T \vQbar} \leq \bigO{\eps} \kappa^{\alpha_1}(\vX).
        \end{split}
    \end{equation*}
    Similarly, assume the following hold for $[\vQbar, \Rbar] = \IOtwo{\vX}$:
    \begin{equation*}
        \begin{split}
            & \vX + \Delta \vX = \vQbar \Rbar,
            \quad \norm{\Delta \vX} \leq \bigO{\eps} \norm{\vX}, \\
            & \norm{I - \vQbar^T \vQbar} \leq \bigO{\eps} \kappa^{\alpha_2}(\vX).
        \end{split}
    \end{equation*}
    Furthermore, if
    \begin{equation}
        \begin{split}
            & \bXX_{k-1} + \Delta \bXX_{k-1} = \bQQbar_{k-1} \RRbar_{k-1},
            \quad \norm{\Delta \bXX_{k-1}}
            \leq \epsXkp \norm{\bXX_{k-1}},
        \end{split}
    \end{equation}
    as well as~\eqref{eq:lem-bcgsiro-property:assump} holds, then for the $k$th inner loop of Algorithm~\ref{alg:BCGSIROA} with any $k\geq 2$,
    \begin{equation*}
        \norm{I - \bQQbar_k^T \bQQbar_k}
        \leq \epsQkp
        + 2 \epsQkp(1 + \epsQkp) \left(\epsQkp + \bigO{\eps} \right) \kappa(\bXX_k)
        + \bigO{\eps}.
    \end{equation*}
\end{lemma}

\begin{proof}
    We apply Lemma~\ref{lem:bcgsiro-property} to bound $\norm{\bQQbar_{k-1}^T \vQbar_k}$ as follows:
    \begin{equation}
        \begin{split}
            \norm{\bQQbar_{k-1}^T \vQbar_k}
            & \leq \norm{\bQQbar_{k-1}^T \vHtil_k \Tbar_{kk}^\inv}
            + \norm{\bQQbar_{k-1}^T \Delta\vHtil_k \Tbar_{kk}^\inv} \\
            & \leq \norm{\bQQbar_{k-1}^T(I - \bQQbar_{k-1} \bQQbar_{k-1}^T) \vUbar_k \Tbar_{kk}^\inv}
            + \norm{\bQQbar_{k-1}^T \Delta\vHtil_k\Tbar_{kk}^\inv} \\
            & \leq \epsQkp \norm{\bQQbar_{k-1}^T \vUbar_k \Tbar_{kk}^\inv}
            + \norm{\bQQbar_{k-1}^T \Delta\vHtil_k \Tbar_{kk}^\inv} \\
            & \leq \epsQkp \left(\norm{\bQQbar_{k-1}^T \vVtil_k \Sbar_{kk}^\inv \Tbar_{kk}^\inv}
            + \norm{\bQQbar_{k-1}^T \Delta\vVtil_k \Sbar_{kk}^\inv \Tbar_{kk}^\inv} \right) \\
                & \quad + \norm{\bQQbar_{k-1}^T \Delta\vHtil_k \Tbar_{kk}^\inv} \\
            & \leq \epsQkp \norm{\bQQbar_{k-1}^T(I - \bQQbar_{k-1} \bQQbar_{k-1}^T) \vX_k \Sbar_{kk}^\inv \Tbar_{kk}^\inv} \\
                & \quad + \epsQkp \norm{\bQQbar_{k-1}^T \Delta\vVtil_k \Sbar_{kk}^\inv \Tbar_{kk}^\inv} + \norm{\bQQbar_{k-1}^T \Delta\vHtil_k \Tbar_{kk}^\inv} \\
            & \leq \epsQkp^2 \norm{\bQQbar_{k-1}^T \vX_k \Sbar_{kk}^\inv \Tbar_{kk}^\inv} \\
                & \quad + \epsQkp \norm{\bQQbar_{k-1}^T \Delta\vVtil_k \Sbar_{kk}^\inv \Tbar_{kk}^\inv} + \norm{\bQQbar_{k-1}^T \Delta\vHtil_k \Tbar_{kk}^\inv} \\
            & \leq \epsQkp(1 + \epsQkp) \left(\epsQkp + \bigO{\eps} \right) \norm{\vX_k} \norm{\Sbar_{kk}^\inv} \norm{\Tbar_{kk}^\inv} \\
                 & + (1 + \epsQkp) \bigO{\eps}(1 + \bigO{\eps} \kappa^{\alpha_1}(\bXX_k)) \norm{\Tbar_{kk}^\inv}.
        \end{split}
    \end{equation}
    By Lemma~\ref{lem:norm-R}, \eqref{eq:lem-bcgsiro-property-orth2}, \eqref{eq:lem-bcgsiro-property-Qkloss}, and~\eqref{eq:proof-lem-bcgsiro-property-sigmamintildeW}, we have
    \begin{equation*}
    \begin{split}
        \norm{\Tbar_{kk}^\inv}
        & \leq \frac{\norm{\vQbar_k}}
        {\sigmin(\vHtil_k) - \bigO{\eps}(1 + \bigO{\eps} \kappa^{\alpha_1}(\bXX_k))} \\
        & \leq \frac{1 + \bigO{\eps}}{1 - \bigO{\eps} \kappa^{\alpha_1}(\bXX_k)
        - 2 \left(\epsQkp + \bigO{\eps} \right)(1 + \epsQkp)^2 \kappa(\bXX_k)} \\
        & \leq 2+ \bigO{\eps}.
    \end{split}
    \end{equation*}
    Combining this with~\eqref{eq:norm-XkSkkinv} proved in Lemma~\ref{lem:bcgsiro-property}, we bound $\norm{\bQQbar_{k-1}^T \vQbar_k}$.  Finally, we conclude the proof because of
    \begin{equation*}
        \norm{I - \bQQbar_k^T \bQQbar_k}
        \leq \norm{I - \bQQbar_{k-1}^T \bQQbar_{k-1}}
        + 2 \norm{\bQQbar_{k-1}^T \vQbar_k}
        + \norm{I - \vQbar_k^T \vQbar_k}.
    \end{equation*}
    and \eqref{eq:lem-bcgsiro-property-Qkloss}.
\end{proof}

By induction on $k$, we achieve following theorem to show the LOO of $\BCGSIROA$.
\begin{theorem} \label{thm:bcgsiroa}
    Let $\bQQbar$ and $\RRbar$ denote the computed results of Algorithm~\ref{alg:BCGSIROA}.  Assume that for all $\vX \in \spR^{m \times s}$ with $\kappa(\vX) \leq \kappa(\bXX)$, the following hold for $[\vQbar, \Rbar] = \IOA{\vX}$:
    \begin{equation*}
        \begin{split}
            & \vX + \Delta \vX = \vQbar \Rbar,
            \quad \norm{\Delta \vX} \leq \bigO{\eps} \norm{\vX}, \\
            & \norm{I - \vQbar^T \vQbar} \leq \bigO{\eps}.
        \end{split}
    \end{equation*}
    Likewise, assume the following hold for $[\vQbar, \Rbar] = \IOone{\vX}$ and $[\vQbar, \Rbar] = \IOtwo{\vX}$, respectively:
    \begin{equation*}
        \begin{split}
            & \vX + \Delta \vX = \vQbar \Rbar,
            \quad \norm{\Delta \vX} \leq \bigO{\eps} \norm{\vX}, \\
            & \norm{I - \vQbar^T \vQbar} \leq \bigO{\eps} \kappa^{\alpha_1}(\vX);
        \end{split}
    \end{equation*}
    and
    \begin{equation*}
        \begin{split}
            & \vX + \Delta \vX = \vQbar \Rbar,
            \quad \norm{\Delta \vX} \leq \bigO{\eps} \norm{\vX}, \\
            & \norm{I - \vQbar^T \vQbar} \leq \bigO{\eps} \kappa^{\alpha_2}(\vX).
        \end{split}
    \end{equation*}
    If $\bigO{\eps} \kappa^{\theta}(\bXX) < 1$ for $\theta:= \max(\alpha_1, 1)$ is satisfied, then
    \begin{equation} \label{eq:thm:bcgsiroa:res}
        \bXX + \Delta\bXX = \bQQbar \RRbar,
        \quad \norm{\Delta\bXX} \leq \bigO{\eps} \norm{\bXX},
    \end{equation}
    and
    \begin{equation}
        \norm{I - \bQQbar^T \bQQbar} \leq \bigO{\eps}.
    \end{equation}
\end{theorem}

\begin{proof}
    By the assumption of $\IOAnoarg$, we have $\norm{I - \bQQbar_1^T \bQQbar_1} \leq \bigO{\eps}$.  Then we can draw the conclusion by induction on $k$ followed by Lemma~\ref{lem:bcgsiroa-stepk} if the residual bound~\eqref{eq:thm:bcgsiroa:res} can be satisfied.

    The assumptions of $\IOAnoarg$ directly give the base case.  Assume that $\bXX_{k-1} + \Delta \bXX_{k-1} = \bQQbar_{k-1} \RRbar_{k-1}$ with $\norm{\Delta \bXX_{k-1}} \leq \bigO{\eps} \norm{\bXX_{k-1}}$.  Then our aim is to prove that it holds for $k$.  By Lemma~\ref{lem:bcgsiro-property}, we have
    \begin{equation}
        \begin{split}
            & \bQQbar_{k-1} \RRbar_{1:k-1,k} + \vQbar_k \Rbar_{kk} \\
            & = \bQQbar_{k-1} \left(\bQQbar_{k-1}^T \vX_k + \Delta\SSbar_{1:k-1,k}
                + (\bQQbar_{k-1}^T \vUbar_k + \Delta\bar{\TT}_{1:k-1,k}) \Sbar_{kk} + \Delta\RRbar_{1:k-1,k} \right)\\            
                & \quad + \vQbar_k(\Tbar_{kk} \Sbar_{kk} + \Delta\Rbar_{kk}) \\
            & = \bQQbar_{k-1} \left(\bQQbar_{k-1}^T \vX_k + \Delta\SSbar_{1:k-1,k}
                + (\bQQbar_{k-1}^T \vUbar_k + \Delta\bar{\TT}_{1:k-1,k}) \Sbar_{kk} + \Delta\RRbar_{1:k-1,k} \right)\\
                & \quad + (\vHtil_k + \Delta\vHtil_k) \Sbar_{kk} + \vQbar_k \Delta\Rbar_{kk} \\
            & = \bQQbar_{k-1} \left(\bQQbar_{k-1}^T \vX_k + \Delta\SSbar_{1:k-1,k}
                + (\bQQbar_{k-1}^T \vUbar_k + \Delta\bar{\TT}_{1:k-1,k}) \Sbar_{kk} + \Delta\RRbar_{1:k-1,k} \right)\\
                & \quad + ((I - \bQQbar_{k-1} \bQQbar_{k-1}^T) \vUbar_k + \Delta\vHtil_k) \Sbar_{kk} + \vQbar_k \Delta\Rbar_{kk} \\
            & = \bQQbar_{k-1} \left(\bQQbar_{k-1}^T \vX_k + \Delta\SSbar_{1:k-1,k}
                + \Delta\bar{\TT}_{1:k-1,k} \Sbar_{kk} + \Delta\RRbar_{1:k-1,k} \right)+ (\vUbar_k+ \Delta\vHtil_k) \Sbar_{kk}\\
                & \quad + \vQbar_k \Delta\Rbar_{kk} \\
            & = \bQQbar_{k-1} \bQQbar_{k-1}^T \vX_k
                + \vVtil_k + \Delta\vVtil_k
                + \Delta\vHtil_k \Sbar_{kk} + \vQbar_k \Delta\Rbar_{kk}
                + \bQQbar_{k-1} \Delta\SSbar_{1:k-1,k}\\
                & \quad + \bQQbar_{k-1} \Delta\bar{\TT}_{1:k-1,k} \Sbar_{kk}
                + \bQQbar_{k-1} \Delta\RRbar_{1:k-1,k} \\
            & = \vX_k+ \Delta\vVtil_k
                + \Delta\vHtil_k \Sbar_{kk} + \vQbar_k \Delta\Rbar_{kk}
                + \bQQbar_{k-1} \Delta\SSbar_{1:k-1,k}
                + \bQQbar_{k-1} \Delta\bar{\TT}_{1:k-1,k} \Sbar_{kk} \\
                & \quad + \bQQbar_{k-1} \Delta\RRbar_{1:k-1,k}.
        \end{split}
    \end{equation}
    Let
    \begin{align*}
        \Delta\vX_k & = \Delta\vVtil_k
            + \Delta\vHtil_k\Sbar_{kk} + \vQbar_k \Delta\Rbar_{kk}
            + \bQQbar_{k-1} \Delta\SSbar_{1:k-1,k} \\
            & \quad + \bQQbar_{k-1} \Delta\bar{\TT}_{1:k-1,k} \Sbar_{kk} 
            + \bQQbar_{k-1} \Delta\RRbar_{1:k-1,k}.
    \end{align*}
    Then can we conclude the proof of the residual because
    \begin{equation*}
        \begin{split}
            \bXX_k + \Delta \bXX_k 
            & = \bmat{\bXX_{k-1} + \Delta \bXX_{k-1} & \vX_k+ \Delta\vX_k} \\
            & = \bmat{\bQQbar_{k-1} \RRbar_{k-1} & \bQQbar_{k-1} \RRbar_{1:k-1,k} + \vQbar_k \Rbar_{kk}}
        \end{split}
    \end{equation*}
    with
    \[
        \norm{\Delta \bXX_k} \leq \bigO{\eps} \norm{\bXX_{k-1}} + \bigO{\eps} \norm{\vX_k}
        \leq \bigO{\eps} \norm{\bXX_k}.
    \]

    Next we aim to prove the LOO using Lemma~\ref{lem:bcgsiroa-stepk}.  The assumptions of $\IOAnoarg$ directly give the base case for the LOO, i.e., $\omega_1\leq \bigO{\eps}$.
    Assume that $\epsQkp\leq \bigO{\eps}$. Then we obtain that \eqref{eq:lem-bcgsiro-property:assump} holds.
    Using Lemma~\ref{lem:bcgsiroa-stepk} and the assumption of $\IOonenoarg$ and $\IOtwonoarg$ we conclude the proof because
    \begin{equation*}
        \begin{split}
            \norm{I - \bQQbar_k^T \bQQbar_k}
            &\leq \epsQkp + \bigO{1}\epsQkp^2\kappa(\bXX_{k}) + \bigO{\eps}
            \leq \bigO{\eps}.
        \end{split}
    \end{equation*}
\end{proof}

Theorem~\ref{thm:bcgsiroa} reproduces the main result of \cite{BarS13}, which analyzes $\BCGSIRO$, or equivalently $\BCGSIROA$ with $\IOAnoarg = \IOonenoarg = \IOtwonoarg$ in our nomenclature.  Barlow and Smoktunowicz require that all \IOnoargs be as stable as \HouseQR.  In contrast, Theorem~\ref{thm:bcgsiroa} shows that the choice of $\IOtwonoarg$ has no effect on the LOO of $\BCGSIROA$, while $\IOonenoarg$ only limits the conditioning of $\bXX$ for which we can guarantee $\bigO{\eps}$ LOO.  Recently, Barlow proved a similar result for special cases of \BCGSIROA, where $\IOAnoarg = \HouseQR$, $\IOonenoarg$ is either $\HouseQR$ or a reorthogonalized \CholQR, and $\IOtwonoarg = \CholQR$ \cite{Bar24}.  Indeed, Theorem~\ref{thm:bcgsiroa} generalizes \cite{Bar24} and reveals additional possibilities that would further reduce the number of sync points.  Consider, for example, \BCGSIROA with $\IOAnoarg = \IOonenoarg = \TSQR$ and $\IOtwonoarg = \CholQR$.  Such an algorithm would only need 4 sync points per block column, as all \IOnoargs need only one global communication, and still achieve $\bigO{\eps}$ LOO without any additional restriction on $\kappa(\bXX)$.

Unfortunately, Theorem~\ref{thm:bcgsiroa} cannot guarantee stability for $\BCGSIROA \circ \CholQR$ (i.e., $\IOonenoarg = \IOtwonoarg = \CholQR$) when $\kappa(\bXX) > \frac{1}{\sqrt{\eps}} \approx 10^8$, because then $\theta = 2$.  Figure~\ref{fig:roadmap_2_piled} shows $\BCGSIROA \circ \CholQR$ deviating from $\bigO{\eps}$ after $\kappa(\bXX) > 10^8$ for a class of \piled\footnote{Formed as $\bXX = \bmat{\vX_1 & \vX_2 & \cdots & \vX_p}$, where $\vX_1$ has small condition number and for $k \in \{2,\ldots,p\}$, $\vX_k = \vX_{k-1} + \vZ_k$, where each $\vZ_k$ has the same condition number for all $k$. Toggling the condition numbers of $\vX_1$ and $\{\vZ_k\}_{k=2}^p$ controls the overall conditioning of the test matrix.} matrices, which are designed to highlight such edge-case behavior. At the same time, the LOO for $\BCGSIRO\circ\CholQR$ is even more extreme; cf. Figure~\ref{fig:roadmap_2} as well.  Practically speaking, if the application can tolerate $\kappa(\bXX) \leq 10^8$, $\BCGSIROA \circ \CholQR$ would be the superior algorithm here, as \CholQR only requires one sync point per block vector.

\begin{figure}[htbp!]
	\begin{center}
	    \begin{tabular}{cc}
	         \resizebox{.355\textwidth}{!}{\includegraphics[trim={0 0 190pt 0},clip]{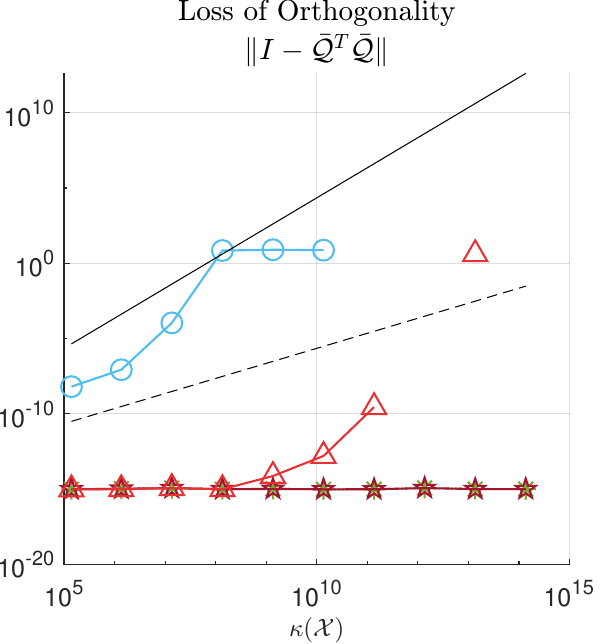}} &
	         \resizebox{.58\textwidth}{!}{\includegraphics{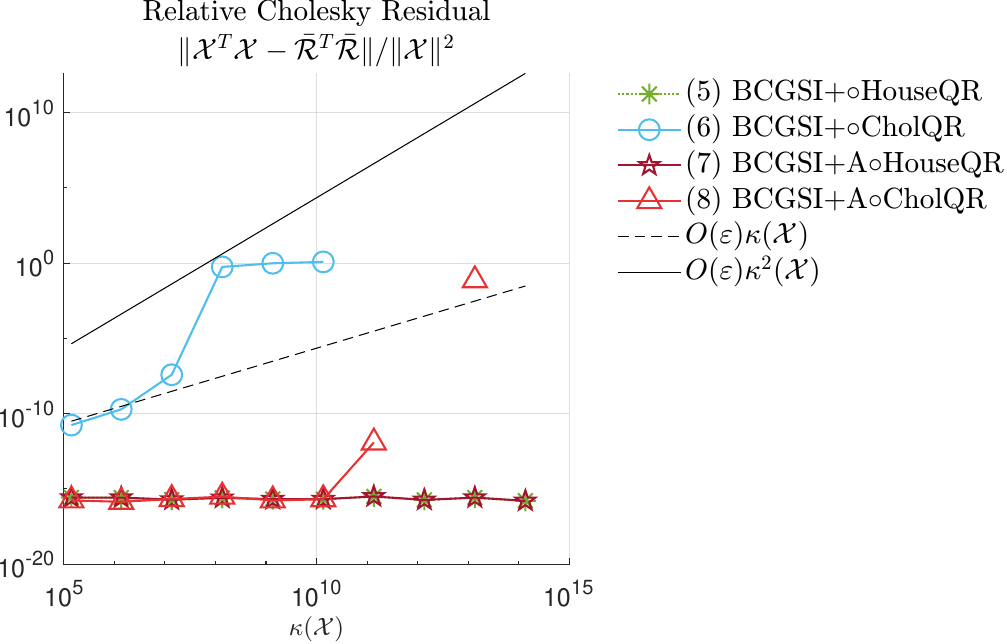}}
	    \end{tabular}
	\end{center}
	\caption{Comparison between \BCGSIRO (i.e., Algorithm~\ref{alg:BCGSIROA} with all \IOnoargs equal) and \BCGSIROA ($\IOAnoarg = \HouseQR$ and $\IOonenoarg = \IOtwonoarg$) on a class of \piled matrices. \label{fig:roadmap_2_piled}}
\end{figure}
\section{Derivation of a one-sync, reorthogonalized, block Gram-Schmidt method} \label{sec:roadmap}
\BCGSIRO and its counterpart \BCGSIROA both require 4 sync points per iteration, which is a disadvantage compared to \BCGS/\BCGSA, which can have demonstrably worse LOO than $\bigO{\eps} \kappa^2(\bXX)$.  Ideally, we would like to reduce the sync points to 2 or even 1 per iteration while keeping a small LOO.  DCGS2 from \cite{BieLTetal22} boasts both 1 sync point per column as well as $\bigO{\eps}$ LOO, at least according to their numerical experiments.  With an eye towards achieving 1 sync point per block column, we will generalize and adapt their derivation, starting from \BCGSIROA. We note that \BCGSIROAthree and \BCGSIROAtwo are considered to be intermediate steps towards the derivation of \BCGSIROAone and its theory. This section aims to show how to achieve a one-sync algorithm using the four-sync approach and to illustrate the impact of reducing each sync point on stability. Unfortunately, \BCGSIROAone has instability issues when $s>1$, that is, its LOO is affected by the square of the condition number, as will be shown in Section~\ref{sec:ls_proofs}. Consequently, this block variant (and thus the equivalent block variant proposed in \cite[Figure~3]{YamTHetal20}) is not the best choice for practical for use. We note that the deficiencies of this variant motivated the development of a stable one-sync variant in~\cite{CarMa2025}.

We can eliminate one sync point by skipping the first normalization step, denoted as $k.1.2$ in \BCGSIROA; consequently, we drop the distinction between $\IOonenoarg$ and $\IOtwonoarg$.  This leads to \BCGSIROAthree, summarized as Algorithm~\ref{alg:BCGSIROA3S}; note again the three colors for each phase.  Despite the small change-- after all, \BCGSIROAthree still projects the basis twice and normalizes the projected block vector in step $k.2.2$-- the effect on the stability behavior is notable.  Figure~\ref{fig:roadmap_3} demonstrates a small LOO when \HouseQR is the \IOnoarg and up to $\bigO{\eps}\kappa^2(\bXX)$ for when $\IOnoarg = \CholQR$.  Furthermore, the relative Cholesky residual of \BCGSIROAthree begins to increase as well.  $\BCGSIROAthree\circ\CholQR$ in particular cannot handle $\kappa(\bXX) > \bigO{\frac{1}{\sqrt{\eps}}}$, due to the Cholesky subroutine being applied to negative semidefinite matrices at some point.
\begin{algorithm}[htbp!]
	\caption{$[\bQQ, \RR] = \BCGSIROAthree(\bXX, \IOAnoarg, \IOnoarg)$ \label{alg:BCGSIROA3S}}
	\begin{algorithmic}[1]
		\State{Allocate memory for $\bQQ$, $\RR$} 
		\State{$[\vQ_1, R_{11}] = \IOA{\vX_1}$} 
		\For{$k = 2, \ldots,p$}
		    \State \first{$\SS_{1:k-1,k} = \bQQ_{k-1}^T \vX_k$}
             \label{line:bcgsiroa3s-proj-begin}
             \SComment{step k.1.1 -- first projection}
		    \State \first{$\vV_k = \vX_k - \bQQ_{k-1} \SS_{1:k-1,k}$} \SComment{\textbf{step k.1.2' -- skip normalization}}
		    \State \second{$\YY_{1:k-1,k} = \bQQ_{k-1}^T \vV_k$}  \label{line:bcgsiroa3s-proj-end}
             \SComment{step k.2.1 -- second projection}
		    \State \second{$[\vQ_k, Y_{kk}] = \IO{\vV_k - \bQQ_{k-1} \YY_{1:k-1,k}}$} \SComment{step k.2.2 -- second normalization}
		    \State \combo{$\RR_{1:k-1,k} = \SS_{1:k-1,k} + \YY_{1:k-1,k}$} \SComment{step k.3.1 -- form upper $\RR$ column}
		    \State \combo{$R_{kk} = Y_{kk}$} \SComment{step k.3.2 -- form $\RR$ diagonal entry}
		\EndFor
		\State \Return{$\bQQ = [\vQ_1, \ldots, \vQ_p]$, $\RR = (R_{ij})$}
	\end{algorithmic}
\end{algorithm}

\begin{figure}[htbp!]
	\begin{center}
	    \begin{tabular}{cc}
	         \resizebox{.325\textwidth}{!}{\includegraphics[trim={0 0 210pt 0},clip]{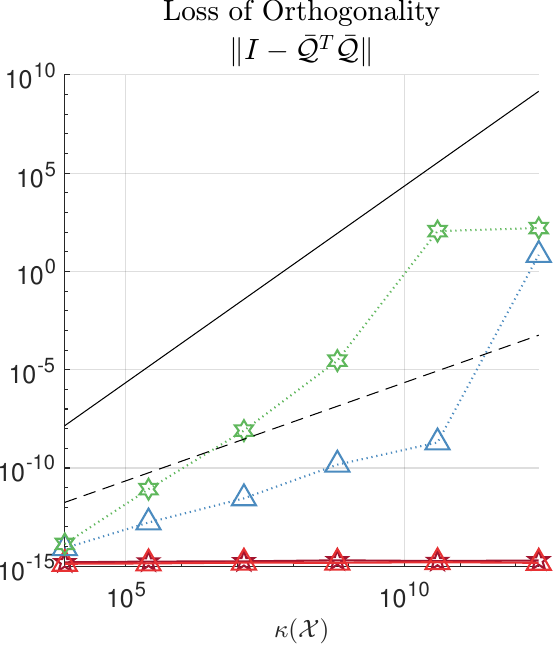}} &
	         \resizebox{.58\textwidth}{!}{\includegraphics{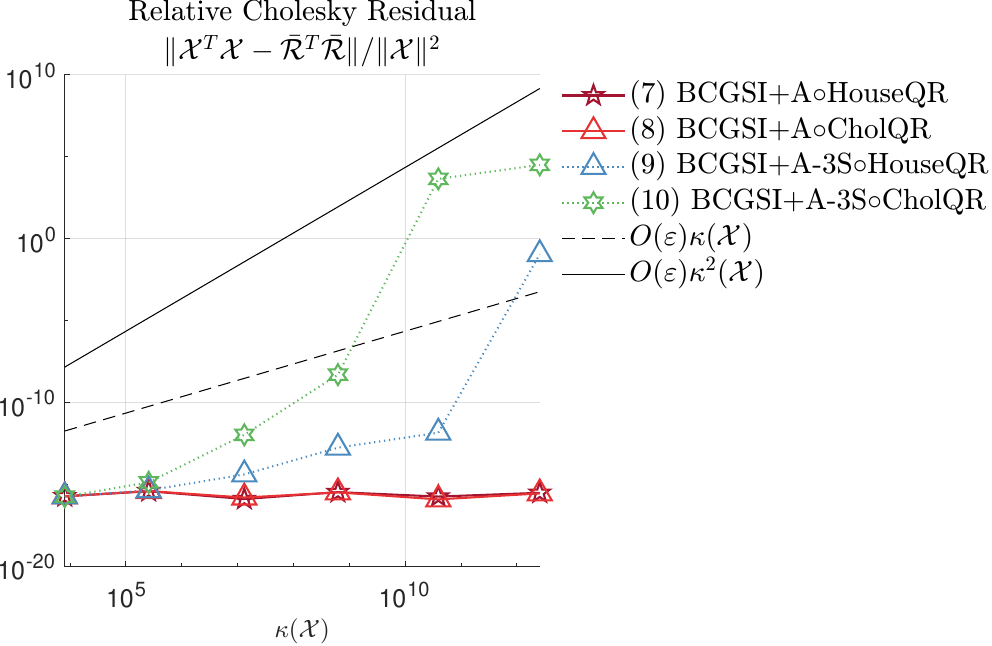}}
	    \end{tabular}
	\end{center}
    \caption{Comparison between \BCGSIROA and \BCGSIROAthree on a class of \monomial matrices. Note that $\IOAnoarg$ is fixed as \HouseQR, and $\IOnoarg = \IOonenoarg = \IOtwonoarg$. \label{fig:roadmap_3}}
\end{figure}

\BCGSIROAthree can also be interpreted as a generalization of the ``continuous projection" approach introduced in \cite{Zou23} as \BCGSCP. The only real difference is that we allow for \texttt{stable\_qr} (i.e., the choice of \IOnoarg) to differ between the first and subsequent steps, thus adding a little extra flexibility.  Zou does not carry out a floating-point analysis of \BCGSCP, which we do in Section~\ref{sec:BCGSIROA3S}

By fixing the \IOnoarg to be \CholQR for iterations $2$ through $p$ and batching the inner products, we arrive at a 2-sync method, \BCGSIROAtwo, displayed in Algorithm~\ref{alg:BCGSIROA2S}.  The block Pythagorean theorem (cf.~\cite{CarLR21}) could also be used to derive this step, but we take a more straightforward approach here by just expanding the inner product implicit in $[\vQ_k, Y_{kk}] = \CholQR(\vV_k - \bQQ_{k-1} \YY_{1:k-1,k})$ (and keeping in mind that we're still working in exact arithmetic for the derivation).  In particular, we find that
\begin{align*}
    Y_{kk}
    &= \vV_k^T\vV_k - \underbrace{\vV_k^T \bQQ_{k-1}}_{= \YY_{1:k-1,k}^T} \YY_{1:k-1,k} - \YY_{1:k-1,k}^T \underbrace{\bQQ_{k-1}^T \vV_k}_{= \YY_{1:k-1,k}} + \YY_{1:k-1,k}^T \underbrace{\bQQ_{k-1}^T \bQQ_{k-1}}_{= I} \YY_{1:k-1,k} \\
    &= \underbrace{\vV_k^T\vV_k}_{= \Omega_k} - \YY_{1:k-1,k}^T \YY_{1:k-1,k}
\end{align*}
and $\vQ_k = \left( \vV_k - \bQQ_{k-1} \YY_{1:k-1,k} \right) Y_{kk}^\inv$.

We assume that batching itself introduces only $\bigO{\eps}$ errors for the matrix products, with the exact floating-point error depending on how the hardware and low-level libraries handle matrix-matrix multiplication.  In which case, $\BCGSIROAtwo$ can be regarded as equivalent to $\BCGSIROAthree\circ\CholQR$ in floating-point error.

\begin{algorithm}[htbp!]
    \caption{$[\bQQ, \RR] = \BCGSIROAtwo(\bXX, \IOAnoarg)$ \label{alg:BCGSIROA2S}}
    \begin{algorithmic}[1]
        \State{$[\vQ_1, R_{11}] = \IOA{\vX_1}$} 
        \For{$k = 2, \ldots,p$}
            \State \first{$\SS_{1:k-1,k} = \bQQ_{k-1}^T \vX_k$}  \SComment{step k.1.1 -- first projection}
            \State \first{$\vV_k = \vX_k - \bQQ_{k-1} \SS_{1:k-1,k}$} \SComment{step k.1.2' -- skip normalization}
            \State \second{$\begin{bmatrix} \YY_{1:k-1,k} \\ \Omega_k \end{bmatrix} = \begin{bmatrix}\bQQ_{k-1} & \vV_k \end{bmatrix}^T \vV_k$} \label{line:bcgsiroa2s-QR-begin}  \SComment{\textbf{step k.2.1' -- second projection, part of \CholQR}}
            \State \second{$Y_{kk} = \chol(\Omega_k - \YY_{1:k-1,k}^T \YY_{1:k-1,k})$} \label{line:bcgsiroa2s-chol}
            \State \second{$\vQ_k = (\vV_k - \bQQ_{k-1} \YY_{1:k-1,k}) Y_{kk}^\inv$} \label{line:bcgsiroa2s-QR-end} \SComment{step k.2.2 -- second normalization}
            \State \combo{$\RR_{1:k-1,k} = \SS_{1:k-1,k} + \YY_{1:k-1,k}$} \SComment{step k.3.1 -- form upper $\RR$ column}
            \State \combo{$R_{kk} = Y_{kk}$} \SComment{step k.3.2 -- form $\RR$ diagonal entry}
        \EndFor
        \State \Return{$\bQQ = [\vQ_1, \ldots, \vQ_p]$, $\RR = (R_{ij})$}
    \end{algorithmic}
\end{algorithm}

Deriving \BCGSIROAone (Algorithm~\ref{alg:BCGSIROA1S}) from \BCGSIROAtwo requires shifting the window over which the for-loop iterates.  First, we bring out steps 2.1.1 and 2.1.2, which leaves $\vV_2$ and $S_{12}$ to initialize the for-loop.  We then batch all the inner products and define some intermediate quantities $\vZ_{k-1}$ and $P_k$.  Consequently, we cannot compute $\SS_{1:k,k+1}$ directly but rather have to reverse-engineer it from what has been computed in line~\ref{line:batched_IP}:
\begin{align*}
     \SS_{1:k,k+1}
     & = \bQQ_k^T \vX_{k+1}\\
     & = \begin{bmatrix} \bQQ_{k-1}  & \vQ_k \end{bmatrix}^T \vX_{k+1}\\
     & = \begin{bmatrix} \vZ_{k-1}   & \vQ_k^T \vX_{k+1} \end{bmatrix}
\end{align*}
By line~\ref{line:Qk}, we have $\vQ_k$, but not its projection onto $\vX_{k+1}$, which we cannot get until the next iteration.  However, from the same line, we can compute $\vQ_k^T \vX_{k+1}$, as it is composed of pieces that can be pulled from line~\ref{line:batched_IP}:
\begin{equation*}
    \vQ_k^T \vX_{k+1}
    = Y_{kk}^\tinv \left( \underbrace{\vV_k^T \vX_{k+1}}_{=P_k} - \YY_{1:k-1,k}^T \underbrace{\bQQ_{k-1}^T \vX_{k+1}}_{=\vZ_{k-1}} \right).
\end{equation*}
After the loop we have to complete the final step $p$.  Interestingly, note that $\bXX$ is no longer needed for the final inner product in line~\ref{alg:BCGSIROA1S:final_IP}.  We highlight again the colorful chaos of the pseudocode: it helps to illustrate what Bielich et al.~\cite{BieLTetal22} and \'{S}wirodowicz et al.~\cite{SwiLAetal21} called ``lagging", in the sense that ``earlier" calculations in \first{blue} now take place after the ``later" ones in \second{red} and \combo{purple} within the for-loop.

\begin{algorithm}[htbp!]
    \caption{$[\bQQ, \RR] = \BCGSIROAone(\bXX, \IOAnoarg)$ \label{alg:BCGSIROA1S}}
    \begin{algorithmic}[1]
        \State{$[\vQ_1, R_{11}] = \IOA{\vX_1}$} \SComment{Implicitly, initialize $\SS = I_{ps}$}
        \State \first{$S_{12} = \vQ_1^T \vX_2$}  \SComment{step 2.1.1 -- first projection}
        \State \first{$\vV_2 = \vX_2 - \vQ_1 S_{12}$} \SComment{step 2.1.2' -- skip normalization}
        \For{$k = 2, \ldots,p-1$}
            \State {$\begin{bmatrix} \second{\YY_{1:k-1,k}} & \first{\vZ_{k-1}} \label{line:bcgsiroa1s:proj-addit}\\ \second{\Omega_k} & \first{P_k} \end{bmatrix} =
                \begin{bmatrix}\bQQ_{k-1} & \vV_k \end{bmatrix}^T \begin{bmatrix}\vV_k & \vX_{k+1} \end{bmatrix}$}
                \SComment{\textbf{step k.2.1', part of (k+1).1.1'}} \label{line:batched_IP}
            \State \second{$Y_{kk} = \chol(\Omega_k - \YY_{1:k-1,k}^T \YY_{1:k-1,k})$}
            \State \second{$\vQ_k = (\vV_k - \bQQ_{k-1} \YY_{1:k-1,k}) Y_{kk}^\inv$} \SComment{step k.2.2 -- second normalization} \label{line:Qk}
            \State \combo{$\RR_{1:k-1,k} = \SS_{1:k-1,k} + \YY_{1:k-1,k}$} \label{line:bcgsiroa1s:proj-begin}\SComment{step k.3.1 -- form upper $\RR$ column}
            \State \combo{$R_{kk} = Y_{kk}$} \SComment{step k.3.2 -- form $\RR$ diagonal entry}
            \State \first{$\SS_{1:k,k+1} =
                \begin{bmatrix} \vZ_{k-1} \\
                Y_{kk}^\tinv \left( P_k - \YY_{1:k-1,k}^T \vZ_{k-1} \right) \end{bmatrix}$}
                \SComment{\textbf{step (k+1).1.1' -- reverse-engineer $\SS$}}
            \State \first{$\vV_{k+1} = \vX_{k+1} - \bQQ_k \SS_{1:k,k+1}$} \label{line:bcgsiroa1s:proj-end}\SComment{step (k+1).1.2' -- skip normalization}
        \EndFor
        \State \second{$\begin{bmatrix} \YY_{1:p-1,p} \\ \Omega_p \end{bmatrix} =
            \begin{bmatrix}\bQQ_{p-1} & \vV_p \end{bmatrix}^T \vV_p$}
            \SComment{step p.2.1'} \label{alg:BCGSIROA1S:final_IP}
        \State \second{$Y_{pp} = \chol(\Omega_p - \YY_{1:p-1,p}^T \YY_{1:p-1,p})$}
        \State \second{$\vQ_p = (\vV_p - \bQQ_{p-1} \YY_{1:p-1,p}) Y_{pp}^\inv$} \SComment{step p.2.2 -- second normalization}
        \State \combo{$\RR_{1:p-1,p} = \SS_{1:p-1,p} + \YY_{1:p-1,p}$} \SComment{step p.3.1 -- form upper $\RR$ column}
        \State \combo{$R_{pp} = Y_{pp}$} \SComment{step p.3.2 -- form $\RR$ diagonal entry}
        \State \Return{$\bQQ = [\vQ_1, \ldots, \vQ_p]$, $\RR = (R_{ij})$}
    \end{algorithmic}
\end{algorithm}

A comparison of the 3-sync, 2-sync, and 1-sync variants is provided in Figure~\ref{fig:roadmap_4}.  \BCGSIROAthree remains under $\bigO{\eps} \kappa^2(\bXX)$, while both \BCGSIROAtwo and \BCGSIROAone explode dramatically once $\kappa(\bXX) > 10^9$.

\begin{figure}[htbp!]
	\begin{center}
	    \begin{tabular}{cc}
	         \resizebox{.325\textwidth}{!}{\includegraphics[trim={0 0 205pt 0},clip]{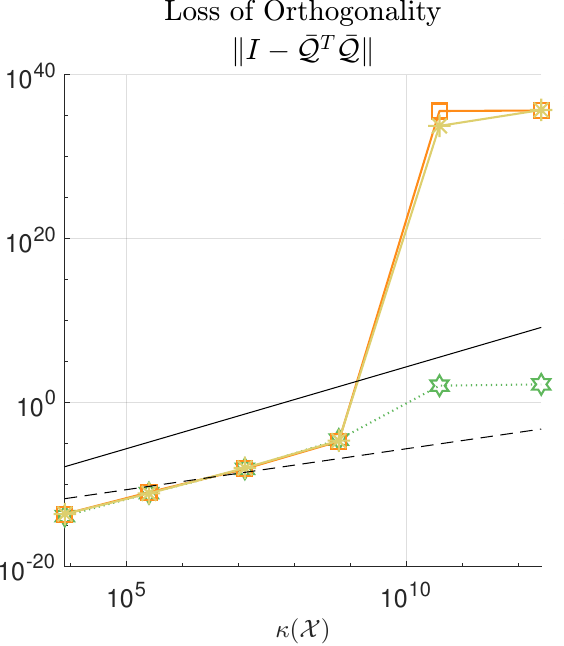}} &
	         \resizebox{.58\textwidth}{!}{\includegraphics{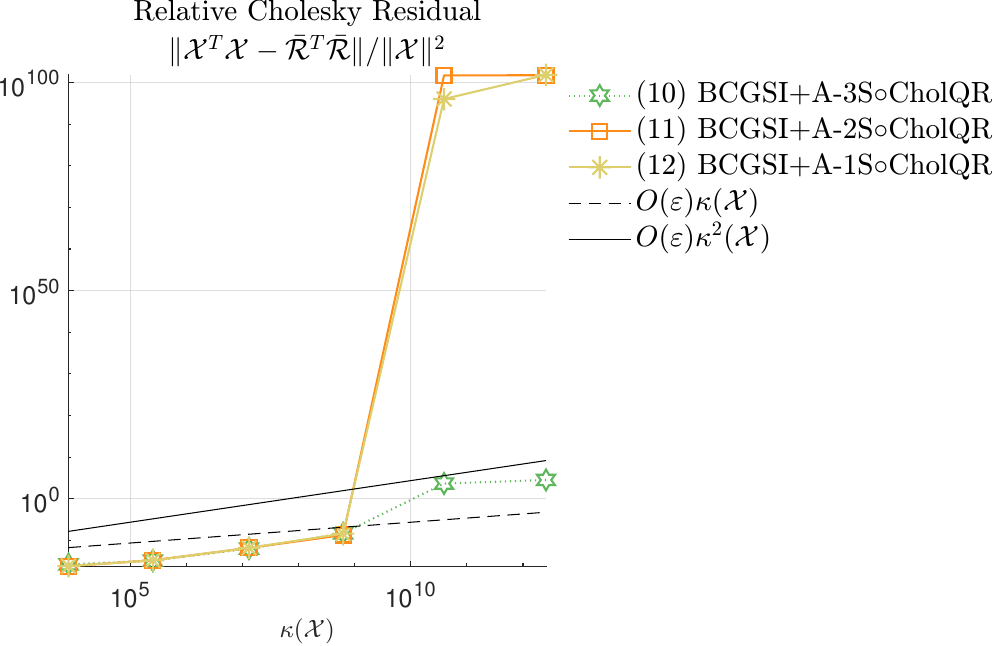}}
	    \end{tabular}
	\end{center}
    \caption{Comparison among low-sync versions of \BCGSIROA on a class of \monomial matrices. Note that $\IOAnoarg$ is fixed as \HouseQR. \label{fig:roadmap_4}}
\end{figure}

This version of 1-sync \BCGSIROA is aesthetically quite different from \cite[Algorithm~7]{CarLRetal22} or \cite[Figure~3]{YamTHetal20}, as well as the column-wise versions of \cite{BieLTetal22} and \cite{SwiLAetal21}.  For one, a general $\IOAnoarg$ is used in the first step.  However, the core of the algorithm-- i.e., everything in the for-loop-- is fundamentally the same, up to $\bigO{\eps}$ rounding errors.  Our derivation for \BCGSIROAone provides an alternative perspective from just writing out the first few steps of the for-loop, batching the inner products, and reverse-engineering the next column of $\SS$ from the most recently computed inner product.

The \monomial example used in Figures~\ref{fig:roadmap_1}-\ref{fig:roadmap_4}-- which are combined in Figure~\ref{fig:roadmap_monomial}-- paints a pessimistic picture for methods with reduced sync points.  The \monomial matrices are not especially extreme matrices; they are in fact designed to mimic $s$-step Krylov subspace methods and are built from powers of a well-conditioned operator.  There are certainly cases where \BCGSIROAone may be good enough; see Figure~\ref{fig:roadmap_default} for comparisons on the \texttt{default} matrices, which are built by explicitly defining a singular value decomposition from diagonal matrices with logarithmically spaced entries.  Clearly all methods discussed so far appear more stable than \BCGS on these simple matrices.
\begin{figure}[htbp!]
	\begin{center}
	    \begin{tabular}{cc}
	         \resizebox{.33\textwidth}{!}{\includegraphics[trim={0 0 210pt 0},clip]{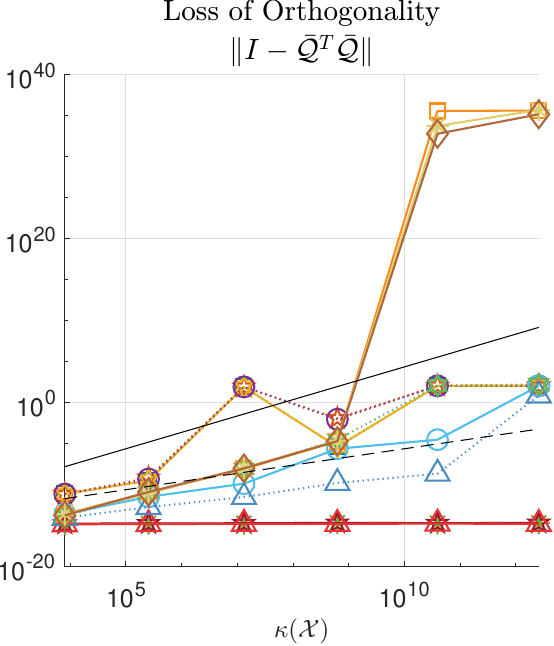}} &
	         \resizebox{.58\textwidth}{!}{\includegraphics{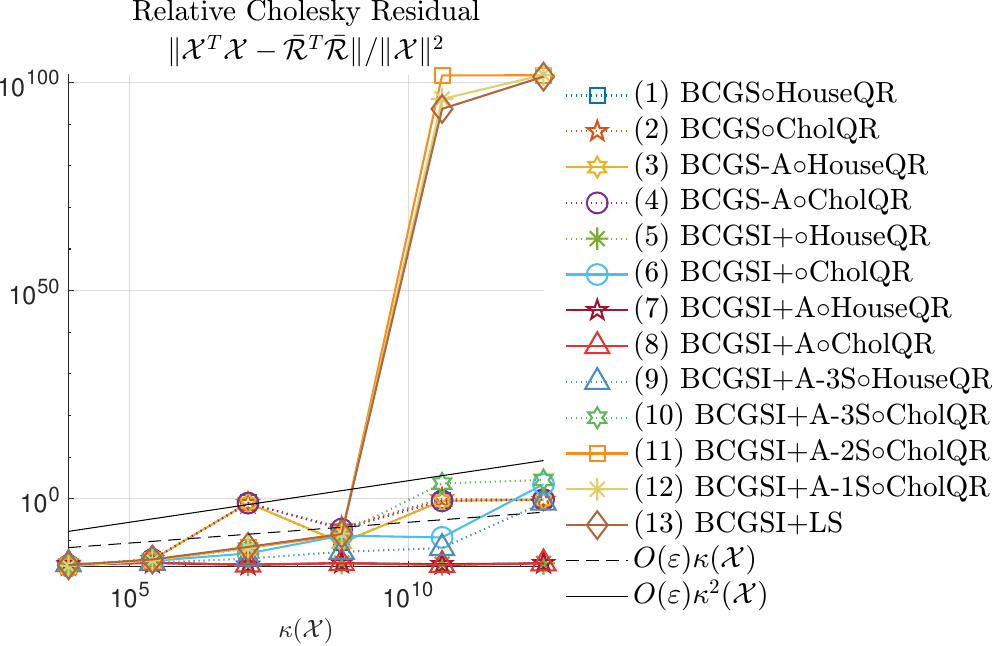}}
	    \end{tabular}
	\end{center}
    \caption{Comparison among \BCGS, \BCGSIRO, \BCGSIROA, and low-sync variants thereof on \monomial matrices. \BCGSIROLS is Algorithm~7 from \cite{CarLRetal22}. Note that $\IOAnoarg$ is fixed as \HouseQR in \texttt{BlockStab}. \label{fig:roadmap_monomial}}
\end{figure}

\begin{figure}[htbp!]
	\begin{center}
	    \begin{tabular}{cc}
	         \resizebox{.33\textwidth}{!}{\includegraphics[trim={0 0 210pt 0},clip]{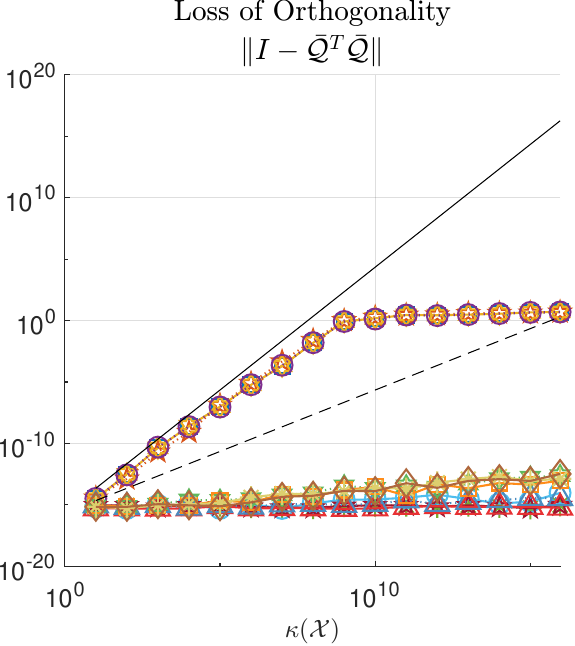}} &
	         \resizebox{.58\textwidth}{!}{\includegraphics{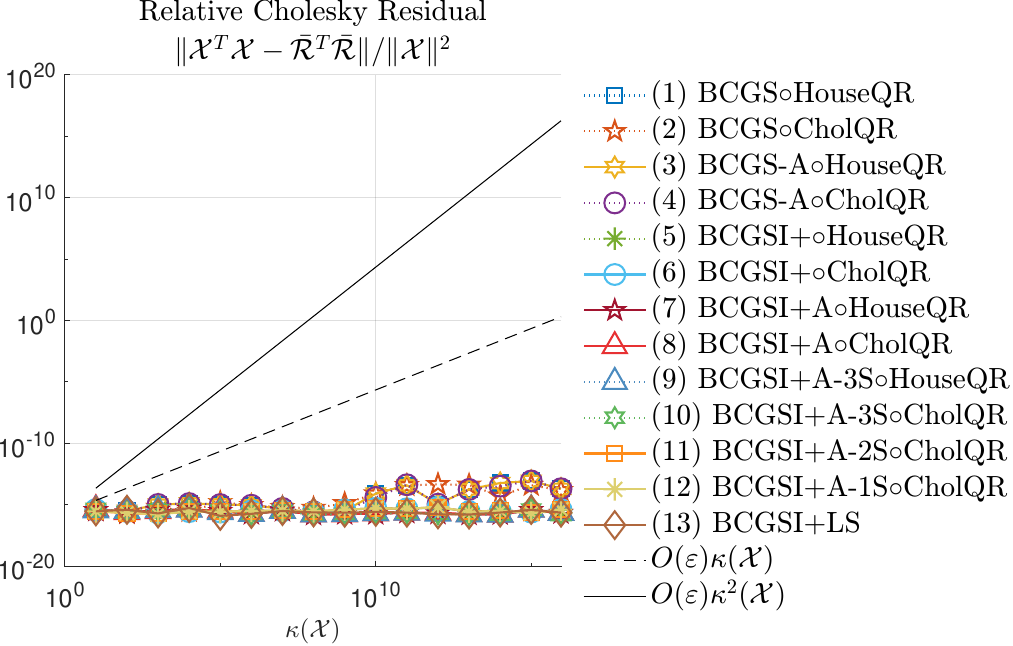}}
	    \end{tabular}
	\end{center}
    \caption{Comparison among \BCGS, \BCGSIRO, \BCGSIROA, and low-sync variants thereof on \texttt{default} matrices. \BCGSIROLS is Algorithm~7 from \cite{CarLRetal22}. Note that $\IOAnoarg$ is fixed as \HouseQR in \texttt{BlockStab}. \label{fig:roadmap_default}}
\end{figure}

We remind the reader that Figures~\ref{fig:roadmap_1}-\ref{fig:roadmap_default} are all generated by the MATLAB script \texttt{test\_roadmap.m} in \texttt{BlockStab}, which we recommend downloading and interacting with for a better understanding of the nuances of these BGS variants.
\section{Loss of orthogonality of low-sync versions of \texttt{BCGSI+A}} \label{sec:ls_proofs}
Figures~\ref{fig:roadmap_3}-\ref{fig:roadmap_4} reinforce the challenges of proving tight upper bounds for the LOO of Algorithms~\ref{alg:BCGSIROA3S}-\ref{alg:BCGSIROA1S}.  None of the variants has a small relative Cholesky residual, meaning that we cannot rely on the technique from \cite{CarLR21}, which inducts over the $\RR$ factor to obtain an LOO bound. We also know that the LOO can be worse than $\bigO{\eps} \kappa^2(\bXX)$, for $\bigO{\eps} \kappa^2(\bXX) \leq 1$.  However, we can still employ the general framework from Section~\ref{sec:framework} to prove some insightful bounds.

\subsection{\texttt{BCGSI+A-3S}} \label{sec:BCGSIROA3S}
For Algorithm~\ref{alg:BCGSIROA3S}, the projection and intraortho stage can be written respectively as
\begin{equation}
    \vG = \Proj{\vX, \bQQ} := (I - \bQQ \bQQ^T)(I - \bQQ \bQQ^T)\vX
    \quad\text{and}\quad \QR{\vG} := \IO{\vG}. 
\end{equation}
Specifically, for the $k$th inner loop of Algorithm~\ref{alg:BCGSIROA3S},
\begin{align}
    & \vW_k = \Proj{\vX_k, \bQQ_{k-1}} = (I - \bQQ_{k-1} \bQQ_{k-1}^T)(I - \bQQ_{k-1} \bQQ_{k-1}^T) \vX_k, \label{eq:two-proj} \\
    &[\vQ_k, R_{kk}] = \QR{\vW_k} = \IO{\vW_k}. \label{eq:bcgsiroa3s-orth}
\end{align}
Then we define
\begin{equation} \label{eq:definition-Wtilk}
    \vWtil_k = (I - \bQQbar_{k-1} \bQQbar_{k-1}^T)(I - \bQQbar_{k-1} \bQQbar_{k-1}^T) \vX_k,
\end{equation}
where $\bQQbar_{k-1}$ satisfies \eqref{eq:def-Vtilk-epsQkp}.
Furthermore, $\vWbar_k$ denotes the computed result of $\vWtil_k$.

In analogue to the analysis of \BCGSA, we first estimate $\sigmin(\vWtil_k)$ and $\kappa(\vWtil_k)$ in Lemma~\ref{lem:norm-Wk1}, and then analyze the specific $\epsproj := \epsprojbcgsiroathree_k$ satisfying $\norm{\vWbar_k - \vWtil_k} \leq \epsprojbcgsiroathree_k$ for the $k$th inner loop and $\norm{\bQQbar_{k-1}^T \vWtil_k \Rbar_{kk}^\inv}$ that are related only to the projection stage in Lemma~\ref{lem:quantities-two-projection}.

\begin{lemma} \label{lem:norm-Wk1}
    Let $\vWtil_k$ and $\bQQbar_{k-1}$ satisfy~\eqref{eq:definition-Wtilk} and~\eqref{eq:def-Vtilk-epsQkp}.
    Assume that
    \begin{equation} \label{eq:lem-norm-Wk1:assump}
        \bXX_{k-1} + \Delta \bXX_{k-1} = \bQQbar_{k-1} \RRbar_{k-1},
        \quad \norm{\Delta \bXX_{k-1}}
        \leq \epsXkp \norm{\bXX_{k-1}}
    \end{equation}
    with $\epsXkp \kappa(\bXX_k)<1$ and $\RRbar_{k-1}$ nonsingular.
    Then for any $k \geq 2$
    \begin{align}
        & \norm{\vWtil_k} \leq \bigl(1 + \epsQkp + (1 + \epsQkp)^2 \epsQkp\bigr) \norm{\vX_k}, \\
        & \sigmin(\vWtil_k) \geq \sigmin(\bXX_k) - \epsXkp \norm{\bXX_{k-1}}, \mbox{ and}\\
        & \kappa(\vWtil_k) \leq \frac{\bigl(1 + \epsQkp + (1 + \epsQkp)^2 \epsQkp\bigr) \kappa(\bXX_k)}{1 - \epsXkp \kappa(\bXX_k)}.
    \end{align}
\end{lemma}

\begin{proof}
    Similarly to the proof of Lemma~\ref{lem:norm-Wk0}, by the assumption~\eqref{eq:lem-norm-Wk1:assump} and
    \begin{equation}
        \begin{split}
            \vWtil_k
            & = \vX_k - \bQQbar_{k-1} \RRbar_{k-1} \RRbar_{k-1}^\inv(\bQQbar_{k-1}^T+ \bQQbar_{k-1}^T(I - \bQQbar_{k-1} \bQQbar_{k-1}^T)) \vX_k \\
            & = \bmat{\bXX_{k-1} + \Delta \bXX_{k-1}& \vX_k} \vD^T
        \end{split}
    \end{equation}
    with $\vD:= \bmat{- \left(\RRbar_{k-1}^\inv(\bQQbar_{k-1}^T+ \bQQbar_{k-1}^T(I - \bQQbar_{k-1} \bQQbar_{k-1}^T)) \vX_k \right)^T& I}$,
    we obtain
    \begin{equation} \label{eq:tildeW-smallest-sva}
        \begin{split}
            \sigmin(\vWtil_k)
            & \geq \sigmin(\bXX_k) - \epsXkp \norm{\bXX_{k-1}}
            \geq \sigmin(\bXX_k) - \epsXkp \norm{\bXX_k}.
        \end{split}
    \end{equation}
    Combining \eqref{eq:tildeW-smallest-sva} with \eqref{eq:lem:norm-I-QQT} and
    \begin{equation}
        \begin{split}
            \norm{\vWtil_k}
            & \leq \left(\norm{I - \bQQbar_{k-1} \bQQbar_{k-1}^T}
            + \norm{\bQQbar_{k-1}}^2 \norm{I - \bQQbar_{k-1}^T \bQQbar_{k-1}} \right) \norm{\vX_k} \\
            & \leq \bigl(1 + \epsQkp + (1 + \epsQkp)^2 \epsQkp\bigr) \norm{\vX_k},
        \end{split}
    \end{equation}
    we have
    \begin{equation}
        \kappa(\vWtil_k)
        \leq \frac{\bigl(1 + \epsQkp + (1 + \epsQkp)^2 \epsQkp\bigr) \kappa(\bXX_k)}{1 - \epsXkp \kappa(\bXX_k)}.
    \end{equation}
\end{proof}

\begin{lemma} \label{lem:quantities-two-projection}
    Let $\vWtil_k$ and $\bQQbar_{k-1}$ satisfy~\eqref{eq:definition-Wtilk} and~\eqref{eq:def-Vtilk-epsQkp}, and $\vWbar_k$ be the computed result of $\vWtil_k$.  For the projection stage~\eqref{eq:two-proj} computed by lines~\ref{line:bcgsiroa3s-proj-begin}--\ref{line:bcgsiroa3s-proj-end} in Algorithm \ref{alg:BCGSIROA3S}, with any $k \geq 2$, it holds that
    \begin{equation*}
        \begin{split}
            & \norm{\vWbar_k - \vWtil_k}\leq \epsprojbcgsiroathree_k \leq \bigO{\eps} \norm{\vWtil_k} + \bigO{\eps^2} \norm{\vX_k} + \bigO{\eps} \epsQkp \norm{\vX_k}, \\
            & \norm{\bQQbar_{k-1}^T \vWtil_k \Rbar_{kk}^\inv}
            \leq(1 + \epsQkp) \epsQkp^2 \norm{\vX_k} \norm{\Rbar_{kk}^\inv}.
        \end{split}
    \end{equation*}
\end{lemma}

\begin{proof}
    By Lemma~\ref{lem:quantities-one-projection}, we have
    \begin{equation} \label{eq:Wk0bar}
    \begin{split}
        & \vVbar_k = \vVtil_k + \Delta\vVtil_k,
        \quad \norm{\Delta\vVtil_k} \leq \bigO{\eps} \norm{\vX_k}, \\
        & \vWbar_k = \vWtil_k + \Delta\vWtil_k,
        \quad \norm{\Delta\vWtil_k} \leq \bigO{\eps} \norm{\vVbar_k}
        \leq \bigO{\eps} \norm{\vVtil_k} + \bigO{\eps^2} \norm{\vX_k}
    \end{split}
    \end{equation}
    with the definition~\eqref{eq:definition-Vtilk} of $\vVtil_k$. Noticing that with
    \[
        \vWtil_k = \vVtil_k - \bQQbar_{k-1}(I - \bQQbar_{k-1}^T \bQQbar_{k-1}) \bQQbar_{k-1}^T \vX_k,
    \]
    we bound $\epsprojbcgsiroathree_k$ because
    \begin{equation*}
        \norm{\vVtil_k} \leq \norm{\vWtil_k} +(1 + \epsQkp)^2\epsQkp \norm{\vX_k}.
    \end{equation*}
    Then by~\eqref{eq:definition-Wtilk}, $\norm{\bQQbar_{k-1}^T \vWtil_k \Rbar_{kk}^\inv}$ is bounded as follows:
    \begin{align*}
        \norm{\bQQbar_{k-1}^T \vWtil_k \Rbar_{kk}^\inv}
        &= \norm{\bQQbar_{k-1}^T \left(I - \bQQbar_{k-1} \bQQbar_{k-1}^T \right) \left(I - \bQQbar_{k-1} \bQQbar_{k-1}^T \right) \vX_k \Rbar_{kk}^\inv} \\
        & \leq \norm{\left(I - \bQQbar_{k-1}^T \bQQbar_{k-1} \right) \left(I - \bQQbar_{k-1}^T \bQQbar_{k-1} \right) \bQQbar_{k-1}^T \vX_k \Rbar_{kk}^\inv} \\
        & \leq \norm{I - \bQQbar_{k-1}^T \bQQbar_{k-1}}^2
        \norm{\bQQbar_{k-1}} \norm{\vX_k} \norm{\Rbar_{kk}^\inv} \\
        & \leq(1 + \epsQkp) \epsQkp^2\norm{\vX_k} \norm{\Rbar_{kk}^\inv}.
    \end{align*}
\end{proof}

The following lemma analyzes the behavior of the $k$th inner loop of $\BCGSIROAthree$.

\begin{lemma} \label{lem:bcgsiroa3s-kinnerloop}
    Assume that $\bigO{\eps} \kappa(\bXX_k) \leq\frac{1}{2}$, $\bQQbar_{k-1}$ satisfies~\eqref{eq:def-Vtilk-epsQkp},
    and
    \[
        \bXX_{k-1} + \Delta \bXX_{k-1} = \bQQbar_{k-1} \RRbar_{k-1, k-1},
        \quad\norm{\Delta \bXX_{k-1}}
        \leq \bigO{\eps} \norm{\bXX_{k-1}}.
    \]
    Then for the $k$th inner loop of Algorithm~\ref{alg:BCGSIROA3S} with any $k \geq 2$:
    \begin{equation*}
        \norm{I - \bQQbar_k^T \bQQbar_k}
        \leq \epsQkp + \bigO{\eps} \kappa(\bXX_k) +4(1 + \epsQkp) \epsQkp^2 (1 + \epsQ) \kappa(\bXX_k) + \epsQ.
    \end{equation*}
\end{lemma}

\begin{proof}
    By Lemma~\ref{lem:epsQk-relation} and Lemma~\ref{lem:quantities-two-projection}, we have
    \begin{equation} \label{eq:lem:bcgsiroa3s-kinnerloop}
        \begin{split}
            \norm{I - \bQQbar_k^T \bQQbar_k}
            & \leq \epsQkp
            + \frac{\bigO{\eps} \kappa(\vWtil_k) + (\bigO{\eps^2} + \bigO{\eps} \epsQkp) \norm{\vX_k}/\sigmin(\vWtil_k)}{1- \bigO{\eps} \norm{\vX_k}/\sigmin(\vWtil_k)- \bigO{\eps} \kappa(\vWtil_k)} \\
            & + \frac{2 (1 + \epsQkp) \epsQkp^2 (1 + \epsQ) \norm{\vX_k}}{\sigmin(\vWtil_k)- \bigO{\eps} \norm{\vX_k} - \bigO{\eps} \norm{\vWtil_k}} + \epsQ.
        \end{split}
    \end{equation}
    Combining \eqref{eq:lem:bcgsiroa3s-kinnerloop} with Lemma~\ref{lem:norm-Wk1} and the assumption~$\bigO{\eps} \kappa(\bXX_k) \leq\frac{1}{2}$, we conclude the proof because
    \begin{equation}\label{eq:lem:bcgsiroa:epsQk}
        \begin{split}
            \norm{I - \bQQbar_k^T \bQQbar_k}
            & \leq \epsQkp
            + \frac{\bigO{\eps} \kappa(\vWtil_k) + (\bigO{\eps^2} + \bigO{\eps} \epsQkp)  \kappa(\bXX_k)}{1- \bigO{\eps} \kappa(\bXX_k)} \\
            & + \frac{2 (1 + \epsQkp) \epsQkp^2 (1 + \epsQ) \kappa(\bXX_k)}{1- \bigO{\eps} \kappa(\bXX_k)} + \epsQ.
        \end{split}
    \end{equation}
\end{proof}

By induction on $k$, we obtain the following theorem to show the loss of orthogonality of $\BCGSIROAthree$.

\begin{theorem} \label{thm:bcgsiroa3s}
    Let $\bQQbar$ and $\RRbar$ denote the computed results of Algorithm~\ref{alg:BCGSIROA3S}. Assume that for all $\vX \in \spR^{m \times s}$ with $\kappa(\vX) \leq \kappa(\bXX)$, the following hold for $[\vQbar, \Rbar] = \IOA{\vX}$:
    \begin{equation*}
        \begin{split}
            & \vX + \Delta \vX = \vQbar\Rbar,
            \quad \norm{\Delta \vX} \leq \bigO{\eps} \norm{\vX}, \\
            & \norm{I - \vQbar^T \vQbar} \leq \bigO{\eps} \kappa^{\alpha_A}(\vX),
        \end{split}
    \end{equation*}
    and for $[\vQbar, \Rbar] = \IO{\vX}$, it holds that
    \begin{equation*}
        \begin{split}
            & \vX + \Delta \vX = \vQbar\Rbar,
            \quad \norm{\Delta \vX} \leq \bigO{\eps} \norm{\vX}, \\
            & \norm{I - \vQbar^T \vQbar} \leq \bigO{\eps} \kappa^{\constQ}(\vX).
        \end{split}
    \end{equation*}
    If $\alpha_A\leq \constQ$ and $\bigO{\eps} \kappa^{\theta}(\bXX) \leq \frac{1}{2}$ with $\theta := \max(\constQ + 1, 2)$, is satisfied, then
    \[
        \bXX + \Delta\bXX = \bQQbar\RRbar,
        \quad\norm{\Delta\bXX} \leq \bigO{\eps} \norm{\bXX}
    \]
    and
    \begin{equation}
        \norm{I - \bQQbar^T \bQQbar} \leq \bigO{\eps} \left(\kappa(\bXX) \right)^{\max\{\constQ, 1\}}.
    \end{equation}
\end{theorem}

\begin{proof}
    We only need to verify the assumptions of Lemma~\ref{lem:bcgsiroa3s-kinnerloop}, which are guaranteed by the assumption $\bigO{\eps} \kappa^{\theta}(\bXX) \leq \frac{1}{2}$ and the residual bound that we establish in the rest of the proof.

    The assumptions on $\IOAnoarg$ directly give the base case of the residual bound.
    Assume that $\bXX_{k-1} + \Delta \bXX_{k-1} = \bQQbar_{k-1} \RRbar_{k-1,k-1}$ with $\norm{\Delta \bXX_{k-1}}\leq \bigO{\eps}\norm{\bXX_{k-1}}$.
    By~\cite{Hig02} and the assumption on $\IOAnoarg$ and $\IOnoarg$, we obtain
    \begin{equation}
        \begin{split}
            \RRbar_{1:k-1,k} &= \SSbar_{1:k-1,k}  + \YYbar_{1:k-1,k} + \Delta \RRbar_{1:k-1,k}\\
            &= \bigl(\bQQbar_{k-1}^T \vX_k + \Delta \SSbar_{1:k-1,k}\bigr) + \bigl(\bQQbar_{k-1}^T \vVbar_k + \Delta \YYbar_{1:k-1,k}\bigr) + \Delta \RRbar_{1:k-1,k}\\
            &= \bQQbar_{k-1}^T \vX_k + \bQQbar_{k-1}^T \vVbar_k + \Delta \vE_{\RR},
        \end{split}
    \end{equation}
    where $\Delta \vE_{\RR} = \Delta \SSbar_{1:k-1,k} + \Delta \YYbar_{1:k-1,k} + \Delta \RRbar_{1:k-1,k}$ satisfies
    \begin{equation}\label{eq:bcgsiroathree:deltaRR}
        \norm{\Delta \vE_{\RR}} \leq \bigO{\eps} \norm{\vX_k}.
    \end{equation}
    Then recalling the definitions~\eqref{eq:definition-Vtilk} and~\eqref{eq:definition-Wtilk}, we have
    \begin{equation}\label{eq:bcgsiroathree:res-k-1}
        \begin{split}
            \bQQbar_{k-1} \RRbar_{1:k-1,k} + \vQbar_k \Rbar_{kk}
            &= \bQQbar_{k-1} \left(\bQQbar_{k-1}^T (\vX_k + \vVtil_k + \Delta \vVtil_k) + \Delta \vE_{\RR}\right) + \vWtil_k \\
            & \quad + \vEproj + \Delta\vG_k \\
            &= \vX_k + \bQQbar_{k-1} \bQQbar_{k-1}^T \Delta\vVtil_k
            + \bQQbar_{k-1} \Delta \vE_{\RR} \\
            & \quad + \vEproj + \Delta\vG_k,
        \end{split}
    \end{equation}
    where $\vEproj = \vWbar_k - \vWtil_k$ and $\Delta\vG_k$ satisfies $\vWbar_k + \Delta\vG_k = \vQbar_k \Rbar_{kk}$. Combining~\eqref{eq:bcgsiroathree:res-k-1} with~\eqref{eq:epsproj}, \eqref{eq:epsqr}, \eqref{eq:bcgsiroathree:deltaRR}, and Lemma~\ref{lem:quantities-two-projection}, we draw the conclusion followed by the proof of Theorem~\ref{thm:bcgs} because
    \begin{equation*}
        \begin{split}
            \bXX_k + \Delta \bXX_k 
            &= \bmat{\bXX_{k-1} + \Delta \bXX_{k-1}& \vX_k+ \Delta\vX_k} \\
            &= \bmat{\bQQbar_{k-1} \RRbar_{k-1,k-1}& \bQQbar_{k-1} \RRbar_{1:k-1,k} + \vQbar_k \Rbar_{kk}},
        \end{split}
    \end{equation*}
    and 
    \[
    \bQQbar_{k-1} \RRbar_{1:k-1,k} + \vQbar_k \Rbar_{kk}
    = \vX_k + \bQQbar_{k-1} \bQQbar_{k-1}^T \Delta \vVtil_k + \bQQbar_{k-1} \Delta \vE_{\RR} + \vEproj + \Delta \vG_k,
    \]
    where $\Delta \vX_k = \bQQbar_{k-1} \bQQbar_{k-1}^T \Delta \vVtil_k + \bQQbar_{k-1} \Delta \vE_{\RR} + \vEproj + \Delta \vG_k$ satisfies $\norm{\Delta \vX_k} \leq \bigO{\eps} \norm{\vX_k}$.  By induction on $k$, the residual bound has been proved.

    Next we aim to prove the LOO using Lemma~\ref{lem:bcgsiroa3s-kinnerloop}.  The assumptions on $\IOAnoarg$ directly give the base case for the LOO, i.e.,
    \[
        \omega_1
        \leq \bigO{\eps} \left(\kappa(\bXX) \right)^{\constQA}
        \leq \bigO{\eps} \left(\kappa(\bXX) \right)^{\max\{\constQ, 1\}}.
    \]
    Now we assume that $\epsQkp\leq \bigO{\eps}\left(\kappa(\bXX) \right)^{\max\{\constQ, 1\}}$.
    Using Lemma~\ref{lem:bcgsiroa3s-kinnerloop} and the assumption on $\IOnoarg$, we conclude that
    \begin{equation*}
        \begin{split}
            \norm{I - \bQQbar_k^T \bQQbar_k}
            &\leq \epsQkp \bigl(4(1 + \epsQkp) (1 + \bigO{\eps}) \epsQkp \kappa(\bXX_k) + 1 \bigr) + \bigO{\eps} \kappa(\bXX_k) \\
            &\leq \bigO{\eps} \left(\kappa(\bXX_k) \right)^{\max\{\constQ, 1\}}.
        \end{split}
    \end{equation*}
    Note that $4(1 + \epsQkp) (1 + \bigO{\eps}) \epsQkp \kappa(\bXX_k) \leq 1$ requires $\bigO{\eps} \kappa^{\alpha+1}(\bXX_k) \leq 1$, which is guaranteed by the assumption $\bigO{\eps} \kappa^{\theta}(\bXX) \leq \frac{1}{2}$.
\end{proof}

With Theorem~\ref{thm:bcgsiroa3s} we have proven our observations from Figure~\ref{fig:roadmap_3} in Section~\ref{sec:roadmap}.  By removing the first \IOnoarg in the inner loop, we implicitly impose a restriction on $\kappa(\bXX)$ dictated by the remaining \IOnoarg.  In particular, for $\BCGSIROAthree\circ\HouseQR$, $\theta = 2$, and for $\BCGSIROAthree\circ\CholQR$, $\theta = 3$.  Practically speaking, in double precision, the first translates to the requirement that $\kappa(\bXX) \leq 10^8$, and the latter to $\kappa(\bXX) \leq 10^{5.\overline{3}}$.

\subsection{\texttt{BCGSI+A-2S}} \label{sec:BCGSIROA2S}
For the $k$th inner loop of Algorithm~\ref{alg:BCGSIROA2S}, the projection stage also satisfies \eqref{eq:two-proj}.
Comparing $\BCGSIROAtwo$ and $\BCGSIROAthree\circ\CholQR$, the only difference is that for the intraortho stage, $\BCGSIROAtwo$ applies a Cholesky factorization (implicitly via line~\ref{line:bcgsiroa2s-chol}) to
\[
    \vV_k^T \vV_k - (\bQQbar_{k-1}^T \vV_k)^T (\bQQbar_{k-1}^T \vV_k),
\]
with $\vV_k = (I - \bQQ_{k-1} \bQQ_{k-1}^T) \vX_k$.
This means that we need to estimate the terms related to the intraortho stage, namely $\epsQ$, in Lemma~\ref{lem:bcgsiroa3s-kinnerloop} and then we derive the loss of orthogonality of $\BCGSIROAtwo$ directly by applying Lemma~\ref{lem:bcgsiroa3s-kinnerloop} in a manner similar to the proof of Theorem~\ref{thm:bcgsiroa3s}.
In the following lemma, we give a bound on $\epsQ$.

\begin{lemma} \label{lem:bcgspipiro}
    Let $\vVtil_k$, $\vWtil_k$, and $\bQQbar_{k-1}$ satisfy \eqref{eq:definition-Vtilk}, \eqref{eq:definition-Wtilk}, and~\eqref{eq:def-Vtilk-epsQkp}, and $\vVbar_k$ and $\vWbar_k$ be the computed results of $\vVtil_k$ and $\vWtil_k$.
    Assume that~\eqref{eq:lem-norm-Wk1:assump} is satisfied with
    \begin{equation} \label{eq:lem:bcgsprpiro-assump}
        2 \bigl((1 + \epsQkp)^2 \epsQkp^3 + \bigO{\eps}\bigr) \kappa^2(\bXX_k) + \bigl(\bigO{\eps} (1 + \epsQkp) + \epsXkp\bigr) \kappa(\bXX_k) \leq \frac{1}{2}.
    \end{equation}
    Then for $\vQbar_k$ and $\Ybar_{kk}$ with any $k \geq 2$ computed by lines~\ref{line:bcgsiroa2s-QR-begin}--\ref{line:bcgsiroa2s-QR-end} in Algorithm~\ref{alg:BCGSIROA2S},
    \begin{equation*}
        \begin{split}
            & \vWbar_k + \Delta \vW_k = \vQbar_k \Ybar_{kk} = \vQbar_k \Rbar_{kk},
            \quad \norm{\Delta \vW_k} \leq \bigO{\eps} \norm{\vWbar_k}, \\
            & \norm{\vQbar_k^T \vQbar_k - I}
            \leq \bigl(2 \cdot (1 + \epsQkp)^2 \epsQkp^3 + \bigO{\eps}\bigr)\kappa(\bXX_k).
        \end{split}
    \end{equation*}
\end{lemma}

\begin{proof}
    By~\cite[Theorem~8.5 and Lemma~6.6]{Hig02}, there exists $\Delta R_i$ such that
    \begin{equation} 
        (\Ybar_{kk}^T + \Delta R_i^T) \vQbar(i, :)^T= \vWbar_k(i, :)^T,
        \quad \norm{\Delta R_i} \leq \bigO{\eps} \norm{\Ybar_{kk}},
    \end{equation}
    for $i = 1, 2, \dotsc, m$, where have used MATLAB notation to access rows.
    Then we obtain
    \begin{equation} \label{eq:backward-error-triangular-sys}
        \vWbar_k + \Delta \vW_k = \vQbar_k \Ybar_{kk},
        \quad \Delta\vW_k(i, :)^T = \Delta R_i^T \vQbar_k(i, :)^T,
    \end{equation}
    where $\Delta\vW_k$ satisfies
    \begin{equation} \label{eq:backward-error-deltaW}
        \norm{\Delta\vW_k} \leq \bigO{\eps} \norm{\vQbar_k} \norm{\Ybar_{kk}}.
    \end{equation}
    We therefore need to estimate $\norm{\Ybar_{kk}}$ and $\norm{\vQbar_k}$.
    
    By standard rounding-error bounds and~\eqref{eq:Qkp-2norm},
    \begin{align}
        & \bar{\YY}_{1:k-1,k} = \bQQbar_{k-1}^T \vVbar_k + \Delta \YY_k,
        \quad \norm{\Delta \YY_k} \leq \bigO{\eps} \norm{\vVbar_k},
        \label{eq:bcgsiroa2s-YY}\\
        & \Ombar_k  = \vVbar_k^T \vVbar_k + \Delta\Omega_k,
        \quad \norm{\Delta\Omega_k} \leq \bigO{\eps} \norm{\vVbar_k}^2,
        \label{eq:bcgsiroa2s-omega}\\
        & \vWbar_k = (I - \bQQbar_{k-1} \bQQbar_{k-1}^T ) \vVbar_k + \Delta \vE_W,
        \quad \norm{\Delta \vE_W} \leq \bigO{\eps} \norm{\vVbar_k}.
        \label{eq:bcgsiroa2s-W}
    \end{align}
    Applying \cite[Theorem~10.3]{Hig02} to line~\ref{line:bcgsiroa2s-chol} of Algorithm~\ref{alg:BCGSIROA2S} leads to
    \begin{equation} \label{eq:thm:PIPI+:Skk_chol}
        \Ybar_{kk}^T\Ybar_{kk} = \Ombar_k - \YYbar_{1:k-1,k}^T\YYbar_{1:k-1,k} + \Delta F_k + \Delta C_k,
    \end{equation}
    where $\Delta F_k$ is the floating-point error from the computations $\Ombar_k - \YYbar_{1:k-1,k}^T \YYbar_{1:k-1,k}$, while $\Delta C_k$ denotes the Cholesky error.  Furthermore, the following bounds hold:
    \begin{equation} \label{eq:lem:PIPI+:DeltaFk_DeltaCk}
        \norm{\Delta F_k} \leq \bigO{\eps}\norm{\vVbar_k}^2 \mbox{ and }
        \norm{\Delta C_k} \leq \bigO{\eps}\norm{\vVbar_k}^2.
    \end{equation}
    Combining~\eqref{eq:lem:PIPI+:DeltaFk_DeltaCk} with~\eqref{eq:bcgsiroa2s-YY} and~\eqref{eq:bcgsiroa2s-omega}, we have
    \begin{equation}
        \Ybar_{kk}^T\Ybar_{kk} = \vVbar_k^T \vVbar_k - \vVbar_k^T \bQQbar_{k-1} \bQQbar_{k-1}^T \vVbar_k + \Delta\vY_k,
        \quad\norm{\Delta\vY_k} \leq \bigO{\eps} \norm{\vVbar_k}^2.
    \end{equation}
    Let
    \begin{align*}
        \vM_k^{(1)} & = \vVbar_k^T (I - \bQQbar_{k-1} \bQQbar_{k-1}^T) \vVbar_k, \mbox{ and} \\
        \vM_k^{(2)} & = \vVbar_k^T (I - \bQQbar_{k-1} \bQQbar_{k-1}^T)(I - \bQQbar_{k-1} \bQQbar_{k-1}^T) \vVbar_k.
    \end{align*}
    Notice that
    \begin{equation*}
        \begin{split}
            & \vM_k^{(1)}
            = \vM_k^{(2)} + \vVbar_k^T \bQQbar_{k-1} (I - \bQQbar_{k-1}^T \bQQbar_{k-1}) \bQQbar_{k-1}^T \vVbar_k, \mbox{ and}\\
            & \vM_k^{(2)} = \vWbar_k^T \vWbar_k - \Delta \vE_W^T \vWbar_k - \vWbar_k^T \Delta \vE_W
            + \Delta \vE_W^T \Delta \vE_W.
        \end{split}
    \end{equation*}
    Then
    \begin{equation} \label{eq:hatZ}
        \Ybar_{kk}^T\Ybar_{kk} = \vM_k^{(2)} + \vVbar_k^T \bQQbar_{k-1} (I - \bQQbar_{k-1}^T \bQQbar_{k-1}) \bQQbar_{k-1}^T \vVbar_k + \Delta\vY_k
        = \vWbar_k^T \vWbar_k + \Delta \vM,
    \end{equation}
    where
    \begin{align*}
        \Delta\vM
        & = - \Delta \vE_W^T \vWbar_k - \vWbar_k^T \Delta \vE_W + \Delta \vE_W^T \Delta \vE_W \\
        & \quad +\vVbar_k^T \bQQbar_{k-1} (I - \bQQbar_{k-1}^T \bQQbar_{k-1}) \bQQbar_{k-1}^T \vVbar_k + \Delta\vY_k.
    \end{align*}
    From Lemmas~\ref{lem:norm-Wk0}, \ref{lem:quantities-one-projection}, \ref{lem:norm-Wk1}, and~\ref{lem:quantities-two-projection},  $\norm{\vVbar_k^T \bQQbar_{k-1} (I - \bQQbar_{k-1}^T \bQQbar_{k-1}) \bQQbar_{k-1}^T \vVbar_k}$ can be bounded by
    \begin{equation}
        \begin{split}
            & \norm{\vVbar_k^T \bQQbar_{k-1} (I - \bQQbar_{k-1}^T \bQQbar_{k-1}) \bQQbar_{k-1}^T \vVbar_k} \\
            & \quad \leq \norm{\vVtil_k^T \bQQbar_{k-1} (I - \bQQbar_{k-1}^T \bQQbar_{k-1}) \bQQbar_{k-1}^T \vVtil_k} + \bigO{\eps}\norm{\vX_k}^2\\
            & \quad \leq \norm{\vX_k^T \bQQbar_{k-1} (I - \bQQbar_{k-1}^T \bQQbar_{k-1})^3 \bQQbar_{k-1}^T \vX_k} + \bigO{\eps}\norm{\vX_k}^2\\
            & \quad \leq \bigl((1 + \epsQkp)^2 \epsQkp^3 + \bigO{\eps}\bigr)\norm{\vX_k}^2,
        \end{split}
    \end{equation}
    and further together with~\eqref{eq:lem:norm-I-QQT} and~\eqref{eq:thm:PIPI+:Skk_chol},
    \begin{equation}
        \begin{split} \label{eq:deltaM}
            \norm{\Delta\vM}
            & \leq \bigl((1 + \epsQkp)^2 \epsQkp^3 + \bigO{\eps}\bigr) \bigr) \norm{\vX_k}^2
            + \bigO{\eps} \norm{\vVbar_k}^2 \\
            & \leq \bigl((1 + \epsQkp)^2 \epsQkp^3 + \bigO{\eps}\bigr) \bigr) \norm{\vX_k}^2.
        \end{split}
    \end{equation}
    Note that also from Lemmas~\ref{lem:norm-Wk1}, \ref{lem:quantities-two-projection}, and~\eqref{eq:lem:norm-I-QQT}, we have
    \begin{equation} \label{eq:forproof:Ykk-sigma-kappa}
        \begin{split}
            & \frac{\bigl((1 + \epsQkp)^2 \epsQkp^3 + \bigO{\eps}\bigr) \norm{\vX_k}^2}{\sigmin^2 (\vWbar_k)} \\
            & \qquad \leq \frac{\bigl((1 + \epsQkp)^2 \epsQkp^3 + \bigO{\eps}\bigr) \norm{\vX_k}^2}{\sigmin^2 (\vWtil_k) - \bigO{\eps} \norm{\vX_k} - \bigO{\eps} \epsQkp \norm{\vX_k}} \\
            & \qquad \leq \frac{\bigl((1 + \epsQkp)^2 \epsQkp^3 + \bigO{\eps}\bigr) \norm{\vX_k}^2}{\sigmin(\bXX_k) - \epsXkp \norm{\bXX_{k-1}} - \bigO{\eps} \norm{\vX_k} - \bigO{\eps} \epsQkp \norm{\vX_k}} \\
            & \qquad \leq \frac{\bigl((1 + \epsQkp)^2 \epsQkp^3 + \bigO{\eps}\bigr) \kappa^2(\bXX_k)}{1 - \epsXkp \kappa(\bXX_k) - \bigO{\eps} (1 + \epsQkp) \kappa(\bXX_k)} \\
            & \qquad \leq \frac{1}{2}
        \end{split}
    \end{equation}
    by the assumption~\eqref{eq:lem:bcgsprpiro-assump}. From~\eqref{eq:hatZ} and~\eqref{eq:deltaM}, it then follows that
    \begin{equation}
        \begin{split} \label{eq:Ykk-sigma-kappa:norm}
             \norm{\Ybar_{kk}}^2
             & \leq \norm{\vWbar_k}^2 + \bigl((1 + \epsQkp)^2 \epsQkp^3 + \bigO{\eps}\bigr) \norm{\vX_k}^2 \\
             & \leq \norm{\vWbar_k}^2\Biggl(1 + \frac{\bigl((1 + \epsQkp)^2 \epsQkp^3 + \bigO{\eps}\bigr)\norm{\vX_k}^2}{\sigmin^2(\vWbar_k)}\Biggr)\\
             & \leq 2\norm{\vWbar_k}^2
        \end{split}
    \end{equation}
    and
    \begin{equation} \label{eq:Ykk-sigma-kappa:sigmin}
        \begin{split}
             \sigmin^2 (\Ybar_{kk})
             & \geq \sigmin^2 (\vWbar_k) - \bigl((1 + \epsQkp)^2 \epsQkp^3 + \bigO{\eps}\bigr) \norm{\vX_k}^2.
        \end{split}
    \end{equation}
    Furthermore, we then have
    \begin{equation} \label{eq:Ykk-sigma-kappa:kappa}
        \begin{split}
             \kappa(\Ybar_{kk})
             & \leq \sqrt{\frac{2\norm{\vWbar_k}^2}{\sigmin^2(\vWbar_k) -  \bigl((1 + \epsQkp)^2 \epsQkp^3 + \bigO{\eps}\bigr) \norm{\vX_k}^2}} \\
             & = \sqrt{\frac{2\kappa(\vWbar_k)^2}{1 -  \bigl((1 + \epsQkp)^2 \epsQkp^3 + \bigO{\eps}\bigr) \norm{\vX_k}^2/\sigmin^2 (\vWbar_k)}} \\
             & \leq 2 \kappa(\vWbar_k).
        \end{split}
    \end{equation}
    From~\eqref{eq:backward-error-triangular-sys}, it follows that
    \begin{equation} \label{eq:bcgsiroa2s-QkTQk}
        \begin{split}
            \vQbar_k^T \vQbar_k
            & = \Ybar_{kk}^\tinv \vWbar_k^T \vWbar_k \Ybar_{kk}^\inv
            + \Ybar_{kk}^\tinv \Delta\vW_k^T \vWbar_k \Ybar_{kk}^\inv\\
            & \quad + \Ybar_{kk}^\tinv \vWbar_k^T \Delta\vW_k \Ybar_{kk}^\inv
            + \Ybar_{kk}^\tinv \Delta\vW_k^T \Delta\vW_k \Ybar_{kk}^\inv.
        \end{split}
    \end{equation}
    Combining~\eqref{eq:bcgsiroa2s-QkTQk} with~\eqref{eq:backward-error-deltaW}, \eqref{eq:hatZ}, \eqref{eq:forproof:Ykk-sigma-kappa}, \eqref{eq:Ykk-sigma-kappa:sigmin}, and \eqref{eq:Ykk-sigma-kappa:kappa}, it holds that
    \begin{equation*}
        \begin{split}
            & \norm{\Ybar_{kk}^\tinv \vWbar_k^T \vWbar_k \Ybar_{kk}^\inv}
            \leq  1 + \frac{\bigl((1 + \epsQkp)^2 \epsQkp^3 + \bigO{\eps}\bigr) \bigr) \norm{\vX_k}^2}{\sigmin (\vWbar_k)^2 - \bigl((1 + \epsQkp)^2 \epsQkp^3 + \bigO{\eps}\bigr) \bigr) \norm{\vX_k}^2}
            \leq 2, \\
            & \norm{\vW_k \Ybar_{kk}^\inv}
            \leq \norm{\vQbar_k} + \bigO{\eps} \norm{\vQbar_k} \kappa(\vWbar_k), \mbox{ and} \\
            & \norm{\Delta\vW_k \Ybar_{kk}^\inv}
            \leq \bigO{\eps} \norm{\vQbar_k} \kappa(\Ybar_{kk})
            \leq \bigO{\eps} \norm{\vQbar_k} \kappa(\vWbar_k),
        \end{split}
    \end{equation*}
    which implies that
    \begin{equation*}
        \begin{split}
            \norm{\vQbar_k}^2
            & \leq \norm{\Ybar_{kk}^\tinv \vWbar_k^T \vWbar_k \Ybar_{kk}^\inv} + \bigO{\eps} \norm{\vQbar_k}^2 \kappa(\vWbar_k)  + \bigl(\bigO{\eps}\kappa(\vWbar_k)\bigr)^2 \norm{\vQbar_k}^2  \\
            & \leq 2 + \bigO{\eps}\kappa(\vWbar_k)  \norm{\vQbar_k}^2.
        \end{split}
    \end{equation*}
    Thus we have
    \begin{equation*}
        \norm{\vQbar_k}
        \leq \sqrt{\frac{2}{1 - \bigO{\eps} \kappa(\vWbar_k)}}
        \leq 2.
    \end{equation*}
    Together with~\eqref{eq:backward-error-deltaW} and~\eqref{eq:Ykk-sigma-kappa:norm}, we bound $\Delta\vW_k$ by $\bigO{\eps} \norm{\vWbar_k}$.
    From~\eqref{eq:backward-error-deltaW} and \eqref{eq:bcgsiroa2s-QkTQk}, we conclude
    \begin{align}
        \norm{I - \vQbar_k^T \vQbar_k}
        & \leq \frac{\bigl((1 + \epsQkp)^2 \epsQkp^3 + \bigO{\eps}\bigr) \bigr) \norm{\vX_k}^2}{\sigmin (\vWbar_k)^2 - \bigl((1 + \epsQkp)^2 \epsQkp^3 + \bigO{\eps}\bigr) \norm{\vX_k}^2}
        + \bigO{\eps} \kappa(\vWbar_k). \label{eq:bcgsiroa2s-LOO}
    \end{align}
\end{proof}

Assuming that $\epsQkp\leq \bigO{\eps}\kappa^2(\bXX_{k-1})$, the assumption~\eqref{eq:lem:bcgsprpiro-assump} of Lemma~\ref{lem:bcgspipiro} can be guaranteed by $\bigO{\eps} \kappa^{8/3}(\bXX_k) \leq \frac{1}{2}$.  It further follows that
\[
    \norm{I - \vQbar_k^T \vQbar_k}
    \leq \bigO{\eps} \epsQkp^3 \kappa(\bXX_k) + \bigO{\eps} \kappa(\bXX_k)
    \leq \bigO{\eps}\kappa^2(\bXX_k),
\]
because the requirement $\bigO{\eps} \kappa^{8/3}(\bXX_k) \leq \frac{1}{2}$ can imply $\epsQkp^2\kappa(\bXX_k)\leq 1$.  Theorem \ref{thm:bcgsiroa2s} summarizes the results for \BCGSIROAtwo using Theorem \ref{thm:bcgsiroa3s} with Lemmas \ref{lem:bcgsiroa3s-kinnerloop} and \ref{lem:bcgspipiro}.
\begin{theorem} \label{thm:bcgsiroa2s}
    Let $\bQQbar$ and $\RRbar$ denote the computed result of Algorithm~\ref{alg:BCGSIROA2S}. Assume that for all $\vX \in \spR^{m \times s}$ with $\kappa(\vX) \leq \kappa(\bXX)$, the following hold for $[\vQbar, \Rbar] = \IOA{\vX}$:
    \begin{equation*}
        \begin{split}
            & \vX + \Delta \vX = \vQbar\Rbar,
            \quad \norm{\Delta \vX} \leq \bigO{\eps} \norm{\vX}, \\
            & \norm{I - \vQbar^T \vQbar} \leq \bigO{\eps} \kappa^{\alpha_A}(\vX).
        \end{split}
    \end{equation*}
    If $\alpha_A\leq 2$ and $\bigO{\eps} \kappa^{3}(\bXX) \leq \frac{1}{2}$ is satisfied, then
    \[
        \bXX + \Delta\bXX = \bQQbar\RRbar,
        \quad\norm{\Delta\bXX} \leq \bigO{\eps} \norm{\bXX}
    \]
    and
    \begin{equation}
        \norm{I - \bQQbar^T \bQQbar} \leq \bigO{\eps} \kappa^2(\bXX).
    \end{equation}
\end{theorem}

Similarly to Theorem~\ref{thm:bcgsiroa3s}, Theorem~\ref{thm:bcgsiroa2s} reifies observations from Figure~\ref{fig:roadmap_4}, most notably the common behavior between $\BCGSIROAthree\circ\CholQR$ and $\BCGSIROAtwo\circ\CholQR$.  In particular, one cannot expect the two-sync variant to be better than $\BCGSIROAthree\circ\CholQR$, and indeed, the exponent on $\kappa(\bXX)$ is fixed to $3$ now, meaning that in double precision, we cannot prove stability when $\kappa(\bXX) > 10^{5.\overline{3}}$.

\subsection{\texttt{BCGSI+A-1S}} \label{sec:BCGSIROA1S}
The main difference between $\BCGSIROAone$ and $\BCGSIROAtwo$ is how $\vW_k$ defined in~\eqref{eq:two-proj} is computed, particularly $\vV_k$.  Thus, just as in the proof of $\BCGSIROAtwo$, we only need to estimate the specific $\epsproj$ for the $k$th inner loop of Algorithm~\ref{alg:BCGSIROA1S}, i.e., $\epsprojbcgsiroaone_k$, and then we can derive the LOO for $\BCGSIROAone$ using the same logic as in Theorem~\ref{thm:bcgsiroa2s}.

\begin{lemma} \label{lem:bcgsiroa1s}
    Let $\vVtil_k$, $\vWtil_k$, and $\bQQbar_{k-1}$ satisfy \eqref{eq:definition-Vtilk}, \eqref{eq:definition-Wtilk}, and~\eqref{eq:def-Vtilk-epsQkp}, and $\vVbar_k$ and $\vWbar_k$ be the computed results of $\vVtil_k$ and $\vWtil_k$ by Algorithm~\ref{alg:BCGSIROA1S}.
    Assume that~\eqref{eq:lem-norm-Wk1:assump} is satisfied with
    \begin{equation} \label{lem:bcgsiroa1s:assump}
        \begin{split}
            & 2 \bigl((1 + \epsQkp)^2 \epsQkp^3 \kappa^2(\bXX_k) + \bigO{\eps} \kappa^3(\bXX_k)\bigr) \\
            &+ \bigl(\bigO{\eps} (1 + \epsQkp) + \epsXkp\bigr) \kappa(\bXX_k) \leq \frac{1}{2}.
        \end{split}
    \end{equation}
    Then for the projection stage~\eqref{eq:two-proj} with any $k \geq 2$ computed by lines~\ref{line:bcgsiroa1s:proj-begin}--\ref{line:bcgsiroa1s:proj-end} and \ref{line:bcgsiroa1s:proj-addit} in Algorithm~\ref{alg:BCGSIROA1S},
    \begin{equation*}
        \norm{\vWbar_k - \vWtil_k}\leq \epsprojbcgsiroaone_k \leq \bigO{\eps} \kappa(\bXX_{k-1}) \norm{\vX_k}.
    \end{equation*}
\end{lemma}

\begin{proof}
    We estimate $\epsprojbcgsiroaone_k$ satisfying $\norm{\vWbar_k - \vWtil_k}\leq \epsprojbcgsiroaone_k$ by analyzing the rounding error of computing $\SSbar_{k,k+1}$, i.e., $\norm{\SSbar_{k,k+1} - \vQbar_k^T\vX_{k+1}}$. Then our aim is to prove
    \begin{equation*}
        \begin{split}
            & \SSbar_{k,k+1} = \vQbar_k \vX_{k+1} + \Delta\SS_{k,k+1} \quad
            \norm{\Delta\SS_{k,k+1}} \leq \bigO{\eps} \kappa(\bXX_k) \norm{\vX_{k+1}}; \\
            & \vVbar_{k+1} = \vVtil_{k+1} + \Delta \vVtil_{k+1},
            \quad \norm{\Delta \vVtil_{k+1}} \leq \bigO{\eps} \kappa(\bXX_k) \norm{\vX_{k+1}}; \mbox{ and} \\
            & \vWbar_{k+1} = \vWtil_{k+1} + \Delta\vWtil_{k+1}, \quad
        \norm{\Delta\vWtil_{k+1}} \leq \bigO{\eps} \kappa(\bXX_k) \norm{\vX_{k+1}}
        \end{split}
    \end{equation*}
    by induction.
    Standard rounding-error analysis and Lemma~\ref{lem:quantities-two-projection} give the base case.
    Now we assume that the following $i-1$ case with $3 \leq i \leq k$ hold:
    \begin{equation} \label{eq:lem:bcgsiroa1s:i-1case}
        \begin{split}
            & \SSbar_{i-1,i} = \vQbar_{i-1} \vX_i + \Delta\SS_{i-1,i}, \quad
            \norm{\Delta\SS_{i-1,i}} \leq \bigO{\eps} \kappa(\bXX_{i-1}) \norm{\vX_i}; \\
            & \vVbar_i = \vVtil_i + \Delta \vVtil_i,
            \quad \norm{\Delta \vVtil_i} \leq \bigO{\eps} \kappa(\bXX_{i-1}) \norm{\vX_i}; \mbox{ and} \\
            & \vWbar_i = \vWtil_i + \Delta\vWtil_i, \quad
        \norm{\Delta\vWtil_i} \leq \bigO{\eps} \kappa(\bXX_{i-1}) \norm{\vX_i},
        \end{split}
    \end{equation}
    and then aim to prove the above also hold for \(i\).
    Note that the base case $i=2$ has already been proved.
    A straightforward rounding-error analysis gives
    \begin{align}
        & \vZbar_{i-1}= \bQQbar_{i-1}^T \vX_{i+1} + \Delta \vZ_{i-1},
        \quad \norm{\Delta \vZ_{i-1}} \leq \bigO{\eps} \norm{\vX_{i+1}}, \\
        & \Pbar_i = \vVbar_i^T \vX_{i+1} + \Delta P_i,
        \quad \norm{\Delta P_i} \leq \bigO{\eps} \norm{\vVbar_i} \norm{\vX_{i+1}}, \mbox{ and} \\
        & \YYbar_{1:i-1,i}= \bQQbar_{i-1}^T \vVbar_i+ \Delta \YY_{1:i-1,i}
        \quad \norm{\Delta \YY_{1:i-1,i}} \leq \bigO{\eps} \norm{\vVbar_i}.
    \end{align}
    Furthermore, similarly to~\eqref{eq:backward-error-triangular-sys}, we obtain
    \begin{equation} \label{eq:bcgsiroa1s-YkkSkk+1}
        \begin{split}
            \Ybar_{ii}^T \SSbar_{i,i+1} & = \Pbar_i - \YYbar_{1:i-1,i}^T \vZbar_{i-1} + \Delta E_i + \Delta\SS_{i,i+1}, \\
            & = \vVbar_i^T \vX_{i+1} + \Delta P_i- \YYbar_{1:i-1,i}^T (\bQQbar_{i-1}^T \vX_{i+1} + \Delta\vZ_{i-1}) \\
            & \quad + \Delta E_i + \Delta\SS_{i,i+1} \\
            & = \vVbar_i^T \vX_{i+1} - \YYbar_{1:i-1,i}^T \bQQbar_{i-1}^T \vX_{i+1}
            + \Delta P_i- \vVbar_i^T \bQQbar_{i-1} \Delta\vZ_{i-1} \\
            & \quad + \Delta E_i + \Delta\SS_{i,i+1},
        \end{split}
    \end{equation}
    where $\Delta E_i$ is the floating-point error from the sum and product of $\Pbar_i- \YYbar_{1:i-1,i}^T \vZbar_{i-1}$ and $\Delta\SS_{i,i+1}$ is from solving the triangular system $\Ybar_{ii}^{-T}\bigl(\Pbar_i-\YYbar_{1:i-1,i}^T \vZbar_{i-1} + \Delta E_i\bigr)$.
    Thus, the following bounds hold:
    \begin{equation}
        \begin{split}
            \norm{\Delta F_i} \leq \bigO{\eps} \norm{\vVbar_i} \norm{\vX_{i+1}}
            \quad\text{and}\quad
            \norm{\Delta\SS_{i,i+1}} &\leq \bigO{\eps} \norm{\SSbar_{i,i+1}}\norm{\Ybar_{ii}}
        \end{split}
    \end{equation}
    with $\Delta F_i:= \Delta P_i - \vVbar_i^T \bQQbar_{i-1} \Delta\vZ_{i-1} + \Delta E_i$.
    Similarly to \eqref{eq:hatZ} and \eqref{eq:deltaM}, we have
    \begin{equation}
        \begin{split}\label{eq:bcgsiroa1s-YkkTYkk}
            \Ybar_{ii}^T \Ybar_{ii} & = \vWbar_i^T \vWbar_i + \Delta \vM, \mbox{ with}\\
            \norm{\Delta\vM} & \leq \bigl((1 + \epsQip)^2 \epsQip^3 + \bigO{\eps}\kappa(\bXX_{i-1})\bigr) \norm{\vX_i}^2.
        \end{split}
    \end{equation}
    Furthermore, in analogy to \eqref{eq:Ykk-sigma-kappa:sigmin} and \eqref{eq:Ykk-sigma-kappa:kappa}, it holds that
    \begin{equation} \label{eq:bcgsiroa1s-normYkk}
        \begin{split}
            \norm{\Ybar_{ii}}^2 & \leq 2 \norm{\vWbar_i}, \\
            \sigmin^2 (\Ybar_{ii}) & \geq \sigmin^2(\vWtil_i) - \bigl((1 + \epsQip)^2 \epsQip^3 + \bigO{\eps}\kappa(\bXX_{i-1})\bigr) \norm{\vX_i}^2, \mbox{ and}\\
            \kappa(\Ybar_{ii})
            & \leq 2 \kappa(\vWbar_i)
            \leq 2 \kappa(\vWtil_i) + \bigO{\eps} \kappa(\bXX_{i-1})\norm{\vX_i},
        \end{split}
    \end{equation}
    which relies on the assumption~\eqref{lem:bcgsiroa1s:assump}. Combining~\eqref{eq:bcgsiroa1s-YkkSkk+1} with~\eqref{eq:bcgsiroa1s-normYkk},
    we obtain
    \begin{equation} \label{eq:Skk+1}
        \SSbar_{i,i+1} = \Ybar_{ii}^\tinv \left(\vVbar_i^T \vX_{i+1} - \YYbar_{1:i-1,i}^T \bQQbar_{i-1}^T \vX_{i+1} \right)
        + \Ybar_{ii}^\tinv \Delta F_i + \Ybar_{ii}^\tinv \Delta\SS_{i,i+1},
    \end{equation}
    with
    \begin{equation} \label{eq:bcgsiroa1s-normYkkinvFk}
        \begin{split}
            & \norm{\Ybar_{ii}^\tinv \Delta F_i} \leq \bigO{\eps} \norm{\vVbar_i} \norm{\vX_{i+1}} \norm{\Ybar_{ii}^\inv}, \mbox{ and}\\
            & \norm{\Ybar_{ii}^\tinv \Delta\SS_{i,i+1}} \leq \bigO{\eps} \norm{\SSbar_{i,i+1}}(2 \kappa(\vWtil_i) + \bigO{\eps} \kappa(\bXX_{i-1})\norm{\vX_i}).
        \end{split}
    \end{equation}

    Now we bound the distance between $\SSbar_{i,i+1}$ and $\vQbar_i^T \vX_{i+1}$. Following the same logic as with~\eqref{eq:bcgsiroa1s-YkkSkk+1}, we have
    \begin{align*}
        \vVbar_i & - \bQQbar_{i-1} \YYbar_{1:i-1,i} + \Delta \vC_i = \vQbar_i\Ybar_{ii}, \quad
         \norm{\Delta\vC_i} \leq \bigO{\eps} \norm{\vVbar_i} + \bigO{\eps} \norm{\Ybar_{ii}}.
    \end{align*}
    Furthermore, by~\eqref{eq:bcgsiroa1s-normYkk} and the assumption~\eqref{lem:bcgsiroa1s:assump}, it follows that
    \begin{equation}
        (\vVbar_i - \bQQbar_{i-1} \YYbar_{1:i-1,i}) \Ybar_{ii}^\inv + \Delta\vD_i = \vQbar_i
    \end{equation}
    with $\Delta\vD_i  = \Delta\vC_i  \Ybar_{ii}^\inv$ satisfying
    \begin{equation}
        \begin{split}
            \norm{\Delta\vD_i}
            & \leq \frac{\bigO{\eps} \norm{\vVbar_i}}{\sigmin(\vWtil_i)}
            + \bigO{\eps} \kappa(\Ybar_{ii})
            \leq \bigO{\eps} \kappa(\bXX_i).
        \end{split}
    \end{equation}
    Together with~\eqref{eq:Skk+1} and~\eqref{eq:bcgsiroa1s-normYkkinvFk}, we arrive at
    \begin{equation*}
        \begin{split}
            \SSbar_{i,i+1}
            & = \vQbar_i^T \vX_{i+1} - \Delta\vD_i^T \vX_{i+1}
            + \Ybar_{ii}^\tinv \Delta F_i + \Ybar_{ii}^\tinv \Delta\SS_{i,i+1} \\
            & = \vQbar_i^T \vX_{i+1} + \Delta \SSbar_{i,i+1},
        \end{split}
    \end{equation*}
    where $\Delta \SSbar_{i,i+1} = - \Delta\vD_i^T \vX_{i+1} + \Ybar_{ii}^\tinv \Delta F_i + \Ybar_{ii}^\tinv \Delta\SS_{i,i+1}$ satisfies
    \begin{equation*}
        \norm{\Delta \SSbar_{i,i+1}} \leq \bigO{\eps} \kappa(\bXX_i) \norm{\vX_{i+1}}.
    \end{equation*}
    Then standard floating-point analysis yields
    \begin{equation}
        \vVbar_{i+1} = \vVtil_{i+1} + \Delta \vVtil_{i+1},
        \quad\norm{\Delta \vVtil_{i+1}} \leq \bigO{\eps} \kappa(\bXX_i) \norm{\vX_{i+1}},
    \end{equation}
    and further,
    \begin{equation}
        \vWbar_{i+1} = \vWtil_{i+1} + \Delta\vWtil_{i+1}, \quad
        \norm{\Delta\vWtil_{i+1}} \leq \bigO{\eps} \kappa(\bXX_i) \norm{\vX_{i+1}},
    \end{equation}
    which gives the bound on $\epsprojbcgsiroaone_k$ by induction on $i$ and noticing $\norm{\vWbar_k - \vWtil_k}\leq \epsprojbcgsiroaone_k$.
\end{proof}

By imitating the proof of $\BCGSIROAthree$, Lemma~\ref{lem:epsQk-relation} leads to the following result on the LOO of $\BCGSIROAone$.

\begin{theorem} \label{thm:bcgsiroa1s}
    Let $\bQQbar$ and $\RRbar$ denote the computed results of Algorithm~\ref{alg:BCGSIROA1S}.  Assume that for all $\vX \in \spR^{m \times s}$ with $\kappa(\vX) \leq \kappa(\bXX)$, the following hold for $[\vQbar, \Rbar] = \IOA{\vX}$:
    \begin{equation*}
        \begin{split}
            & \vX + \Delta \vX = \vQbar\Rbar,
            \quad \norm{\Delta \vX} \leq \bigO{\eps} \norm{\vX}, \\
            & \norm{I - \vQbar^T \vQbar} \leq \bigO{\eps} \kappa^{\alpha_A}(\vX).
        \end{split}
    \end{equation*}
    If $\alpha_A \leq 2$ and $\bigO{\eps} \kappa^{3}(\bXX) \leq \frac{1}{2}$
    are satisfied, then
    \[
        \bXX + \Delta\bXX = \bQQbar \RRbar,
        \quad \norm{\Delta\bXX} \leq \bigO{\eps} \norm{\bXX}
    \]
    and
    \begin{equation}
        \norm{I - \bQQbar^T \bQQbar} \leq \bigO{\eps} \kappa^2(\bXX).
    \end{equation}
\end{theorem}

Theorem~\ref{thm:bcgsiroa1s} concludes the journey through sync-point reductions and, much like Theorems~\ref{thm:bcgsiroa3s} and \ref{thm:bcgsiroa2s}, confirms the observations from Figure~\ref{fig:roadmap_4} in Section~\ref{sec:roadmap}.  It is clear that shifting the window of the for-loop is not to blame for any LOO; the problem stems already from the eliminated \IOnoarg in \BCGSIROAthree and fixing \CholQR as the remaining \IOnoarg in \BCGSIROAtwo.

\section{Summary and consequences of bounds} \label{sec:recap}
Table~\ref{tab:recap} summarizes the key assumptions and conclusions of the main theorems from Sections~\ref{sec:bcgs_iro_a} and \ref{sec:ls_proofs}.  A key feature of our results is the range of $\kappa(\bXX)$ for which LOO bounds are provable (and attainable).  All bounds require at least that $\bXX$ is numerically full rank; as we reduce the number of sync points, that restriction becomes tighter.  \BCGSIROAthree requires $\kappa(\bXX) \lesssim 10^8$ in double precision (or worse if $\alpha > 1$), while \BCGSIROAtwo and \BCGSIROAone need $\kappa(\bXX) \lesssim 10^{5.3}$.  Figure~\ref{fig:roadmap_3_piled} demonstrates $\BCGSIROAthree\circ\HouseQR$ ($\alpha = 0$, $\theta = 2$) deviating from $\bigO{\eps} \kappa^2(\bXX)$ after $\kappa(\bXX)$ exceeds $10^9$ and $\BCGSIROAthree\circ\CholQR$ ($\alpha =2$, $\theta = 3$) deviating much earlier, once $\kappa(\bXX)$ exceeds $10^6$.  Figure~\ref{fig:roadmap_4_piled} tells a similar story for \BCGSIROAtwo and \BCGSIROAone: both begin to deviate after $\kappa(\bXX) = 10^6$.

\begin{table}[htbp!]
    \caption{Summary of all major theorems, their assumptions, and their LOO bounds for \BCGSIROA and lower sync variants thereof.  For the column $\IOAnoarg$, we state only the exponent such that $\IOAnoarg$ has LOO bounded by $\bigO{\eps} \kappa^{\alpha_A}(\vX)$ for block vectors $\vX$.  See Table~\ref{tab:muscles} for examples of $\alpha_*$.} \label{tab:recap}
    \centering
    \begin{tabular}{c|c|c|c|c} \hline
         Variant & Theorem & $\IOAnoarg$ & $\bigO{\eps} \kappa^\theta(\bXX) \leq \frac{1}{2}$ & $\norm{I - \bQQbar^T \bQQbar}$ \\ \hline
         \BCGSIROA  & \ref{thm:bcgsiroa}    & $\alpha_A = 0$    & $\theta = \max(\alpha_1, 1)$  & $\bigO{\eps}$ \\
         \BCGSIROAthree  & \ref{thm:bcgsiroa3s}  & $\alpha_A \leq \alpha$   & $\theta = \max(\alpha + 1, 2)$    & $\bigO{\eps} \kappa^{\max(\alpha,1)}(\bXX)$ \\
         \BCGSIROAtwo   & \ref{thm:bcgsiroa2s}  & $\alpha_A \leq 2$ & $\theta = 3$  & $\bigO{\eps} \kappa^2(\bXX)$ \\
         \BCGSIROAone & \ref{thm:bcgsiroa1s}  & $\alpha_A \leq 2$ & $\theta = 3$    & $\bigO{\eps} \kappa^2(\bXX)$\\
    \end{tabular}
\end{table}

\begin{figure}[htbp!]
	\begin{center}
	    \begin{tabular}{cc}
	         \resizebox{.325\textwidth}{!}{\includegraphics[trim={0 0 210pt 0},clip]{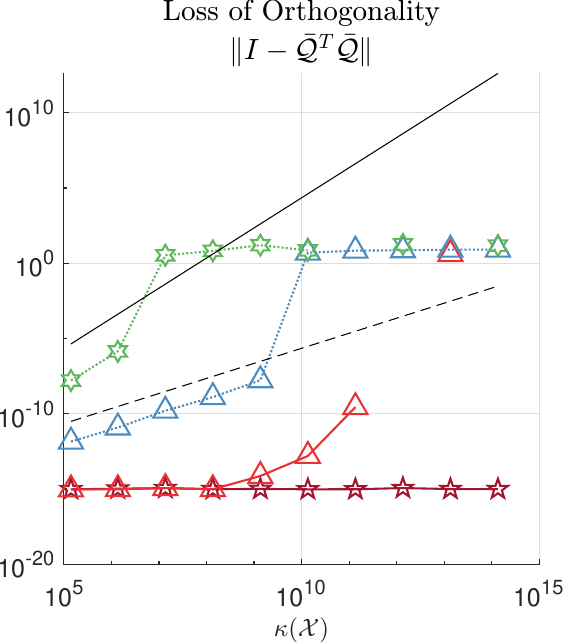}} &
	         \resizebox{.58\textwidth}{!}{\includegraphics{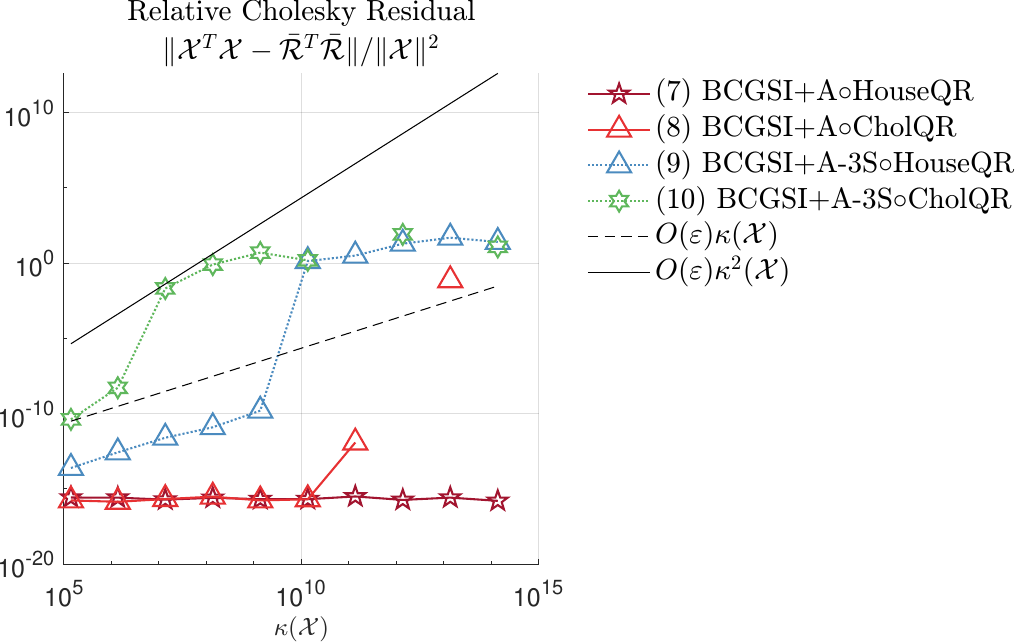}}
	    \end{tabular}
	\end{center}
    \caption{Comparison between \BCGSIROA and \BCGSIROAthree on a class of \piled matrices. Note that $\IOAnoarg$ is fixed as \HouseQR, and $\IOnoarg = \IOonenoarg = \IOtwonoarg$. \label{fig:roadmap_3_piled}}
\end{figure}

\begin{figure}[htbp!]
	\begin{center}
	    \begin{tabular}{cc}
	         \resizebox{.325\textwidth}{!}{\includegraphics[trim={0 0 205pt 0},clip]{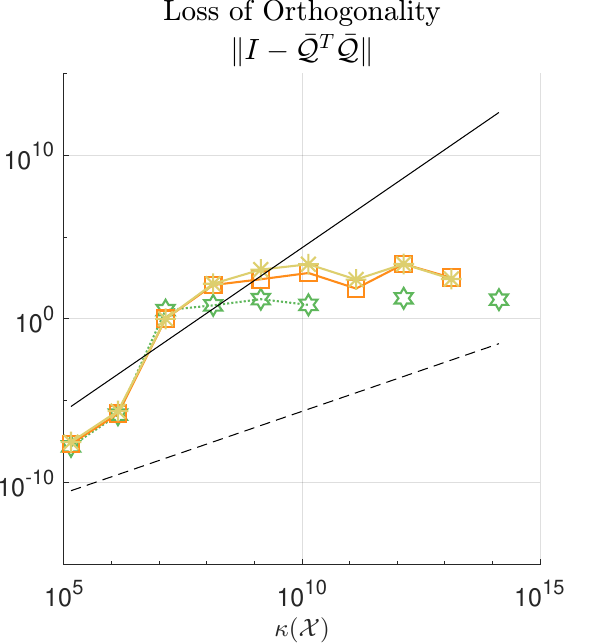}} &
	         \resizebox{.58\textwidth}{!}{\includegraphics{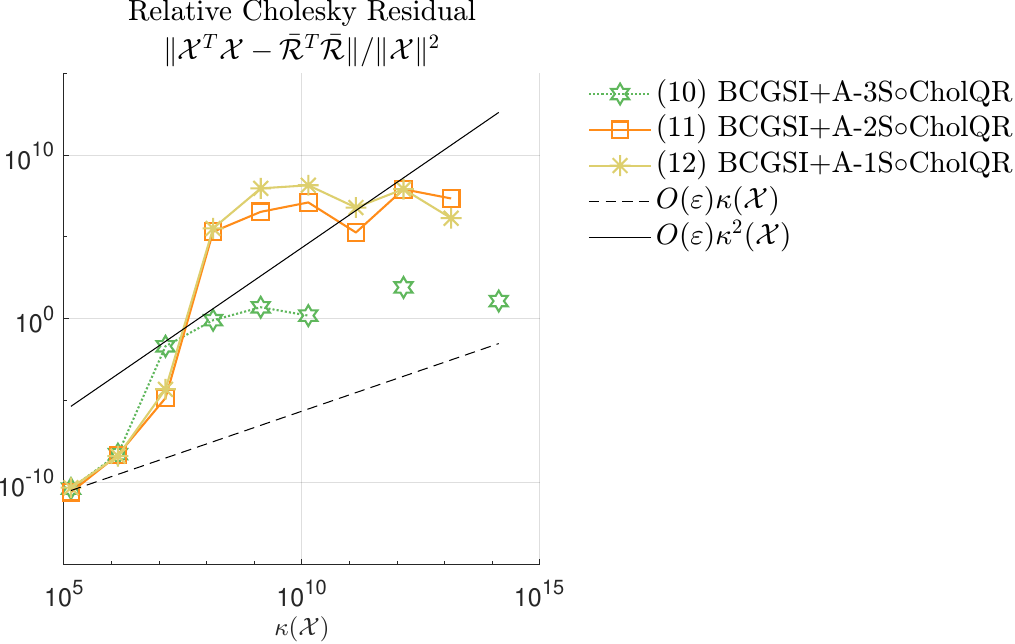}}
	    \end{tabular}
	\end{center}
    \caption{Comparison among low-sync versions of \BCGSIROA on a class of \piled matrices. Note that $\IOAnoarg$ is fixed as \HouseQR. \label{fig:roadmap_4_piled}}
\end{figure}

For completeness, we also provide Figure~\ref{fig:roadmap_piled}, which aggregates plots for all methods discussed on the \piled matrices.  Compare with Figures~\ref{fig:roadmap_monomial} and \ref{fig:roadmap_default}.  In particular, the \piled matrices are quite tough for \BCGS and \BCGSA.

\begin{figure}[htbp!]
	\begin{center}
	    \begin{tabular}{cc}
	         \resizebox{.33\textwidth}{!}{\includegraphics[trim={0 0 210pt 0},clip]{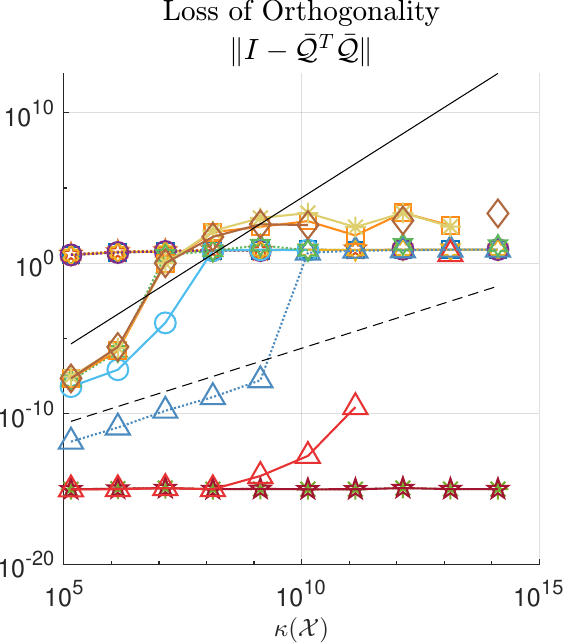}} &
	         \resizebox{.58\textwidth}{!}{\includegraphics{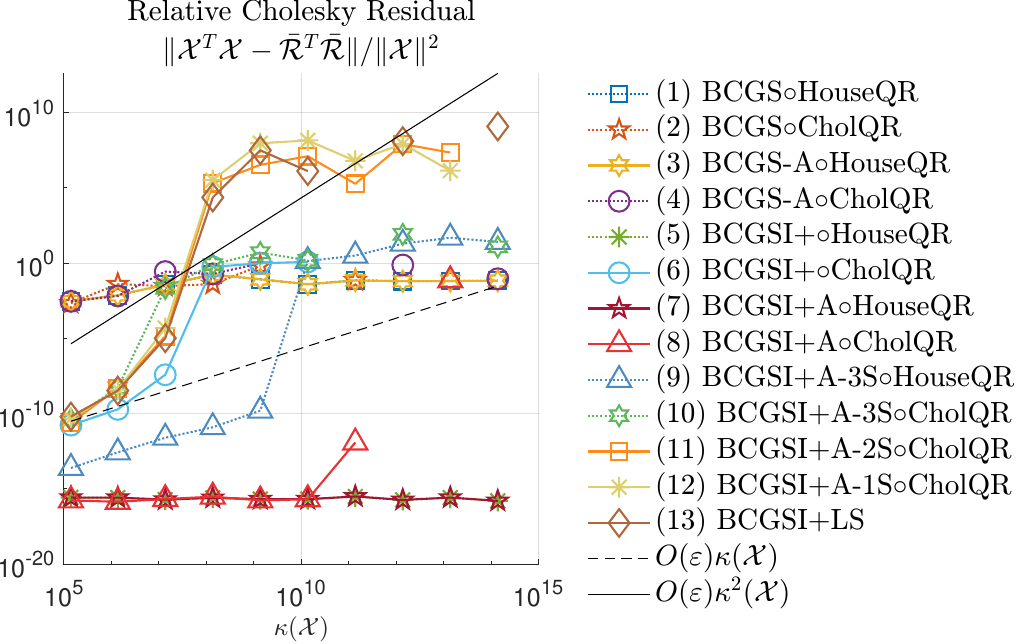}}
	    \end{tabular}
	\end{center}
    \caption{Comparison among \BCGS, \BCGSIRO, \BCGSIROA, and low-sync variants thereof on \piled matrices. \BCGSIROLS is Algorithm~7 from \cite{CarLRetal22}. Note that $\IOAnoarg$ is fixed as \HouseQR in \texttt{BlockStab}. \label{fig:roadmap_piled}}
\end{figure}


Despite the rather pessimistic behavior for block variants we have observed, there is yet good news for the column version of \BCGSIROAone, i.e., when $s = 1$.  Our results apply trivially to the column version and in fact address the open issue of the stability of a one-sync, reorthogonalized \CGS, first introduced as \cite[Algorithm~3]{SwiLAetal21} and revisited as DCGS2 in \cite{BieLTetal22}.

\begin{corollary}
    Let $\bQQbar$ and $\RRbar$ denote the computed result of Algorithm~\ref{alg:BCGSIROA3S}, \ref{alg:BCGSIROA2S}, or~\ref{alg:BCGSIROA1S} with $s = 1$.  If $\bigO{\eps} \kappa(\bXX) \leq \frac{1}{2}$ is satisfied, then
    \begin{equation} \label{eq:cor:backwarderror}
        \bXX + \Delta\bXX = \bQQbar\RRbar,
        \quad\norm{\Delta\bXX} \leq \bigO{\eps} \norm{\bXX}
    \end{equation}
    and
    \begin{equation} \label{eq:cor:loo}
        \norm{I - \bQQbar^T \bQQbar} \leq \bigO{\eps}.
    \end{equation}
\end{corollary}

\begin{proof}
    First note that the proof of Theorem~\ref{thm:bcgsiroa3s} is based on Lemma~\ref{lem:bcgsiroa3s-kinnerloop}.  Now we aim to derive a new version of Lemma~\ref{lem:bcgsiroa3s-kinnerloop} for $s = 1$. Notice that for $s = 1$, $\vX_k$ has only one column and furthermore $\vWtil_k$ defined by~\eqref{eq:definition-Wtilk} also has only one column, which trivially implies that $\kappa(\vWtil_k) = 1$. Combining this realization with~\eqref{eq:lem:bcgsiroa:epsQk}, we obtain
    \begin{equation}\label{eq:s1:epsQk}
        \begin{split}
            \norm{I - \bQQbar_k^T \bQQbar_k}
            &\leq \epsQkp + \bigO{\eps} + \bigO{\eps} \epsQkp \kappa(\bXX_k) \\
            & \quad + 4 (1 + \epsQkp) \epsQkp^2 (1 + \epsQ) \kappa(\bXX_k) + \epsQ
        \end{split}
    \end{equation}
    with the assumption $\bigO{\eps} \kappa(\bXX_k) \leq\frac{1}{2}$. Then we use~\eqref{eq:s1:epsQk} instead of Lemma~\ref{lem:bcgsiroa3s-kinnerloop}, similarly to the proof of Theorem~\ref{thm:bcgsiroa3s}, to conclude that~\eqref{eq:cor:backwarderror} and~\eqref{eq:cor:loo} hold for $\BCGSIROAthree$ with $s = 1$. 

    Recalling Section~\ref{sec:BCGSIROA2S}, the only difference between $\BCGSIROAtwo$ and $\BCGSIROAthree\circ\CholQR$ is the estimation of $\epsQ$, i.e., $\norm{I - \vQbar_k^T \vQbar_k}$, which has been bounded in Lemma~\ref{lem:bcgspipiro}. When $s = 1$, $\kappa(\vWbar_k) = 1$, and by~\eqref{eq:bcgsiroa2s-LOO},
    \[
        \sigmin (\vWbar_k) = \norm{\vWbar_k} \geq \norm{I - \bQQbar_{k-1} \bQQbar_{k-1}^T}^2 \sigmin(\vX_k)
        \geq \sigmin(\vX_k) = \norm{\vX_k}.
    \]
    We then have
    \begin{align*}
        \norm{I - \vQbar_k^T \vQbar_k}
        & \leq 2 \cdot (1 + \epsQkp)^2 \epsQkp^3 + \bigO{\eps}.
    \end{align*}
    Similarly to the proof of Theorem~\ref{thm:bcgsiroa2s}, we can prove that~\eqref{eq:cor:backwarderror} and~\eqref{eq:cor:loo} hold for $\BCGSIROAtwo$ with $s = 1$.

    For $\BCGSIROAone$, we only need to rewrite Lemma~\ref{lem:bcgsiroa1s} for $s = 1$. Since  $\kappa(\bXX_{1}) = \kappa(\vX_1) = 1$, $\kappa(\bXX_{1})$ can be eliminated from the upper bounds of $\Delta\SS_{1,2}$, $\Delta \vVtil_2$, and $\Delta\vWtil_2$ for the base case, i.e., \eqref{eq:lem:bcgsiroa1s:i-1case} with $i = 2$. Furthermore, $\kappa(\bXX_{i-1})$ can be eliminated from the upper bounds of $\Delta\SS_{i-1,i}$, $\Delta \vVtil_i$, and $\Delta\vWtil_i$ in~\eqref{eq:lem:bcgsiroa1s:i-1case}. Then following the same process of Lemma~\ref{lem:bcgsiroa1s}, it holds that 
    \[
        \norm{\vWbar_k - \vWtil_k}\leq \epsprojbcgsiroaone_k \leq \bigO{\eps} \norm{\vX_k}.
    \]
    In analogue to the proof of Theorem~\ref{thm:bcgsiroa1s}, we can conclude the proof for $\BCGSIROAone$ with $s = 1$.
\end{proof}

Figure~\ref{fig:roadmap_piled_column} shows a comparison among column versions of the methods discussed here.  Note that for all methods $\IOAnoarg = \IOonenoarg = \IOtwonoarg$, as all QR subroutines reduce to scaling a column vector by its norm.  Consequently, all versions of \BCGS and \BCGSA are equivalent, as well as \BCGSIRO and \BCGSIROA, etc.  Such redundancies have been removed to make the figure more legible.  Clearly all variants except \BCGS/\BCGSA achieve $\bigO{\eps}$ LOO and relative Cholesky residual.

\begin{figure}[htbp!]
	\begin{center}
	    \begin{tabular}{cc}
	         \resizebox{.33\textwidth}{!}{\includegraphics[trim={0 0 210pt 0},clip]{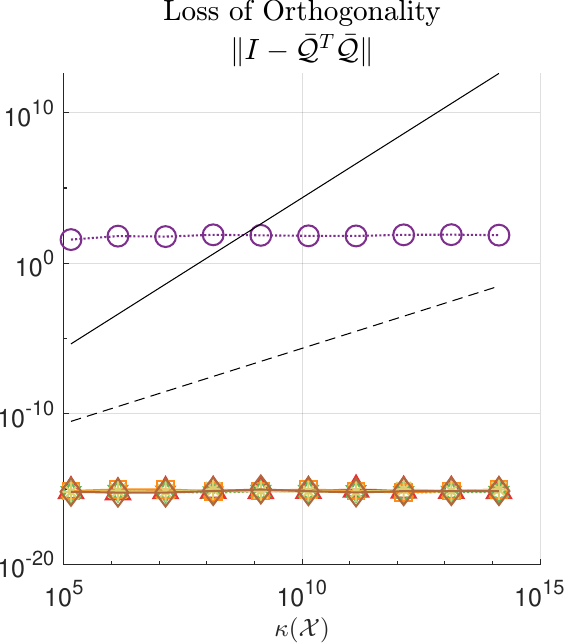}} &
	         \resizebox{.58\textwidth}{!}{\includegraphics{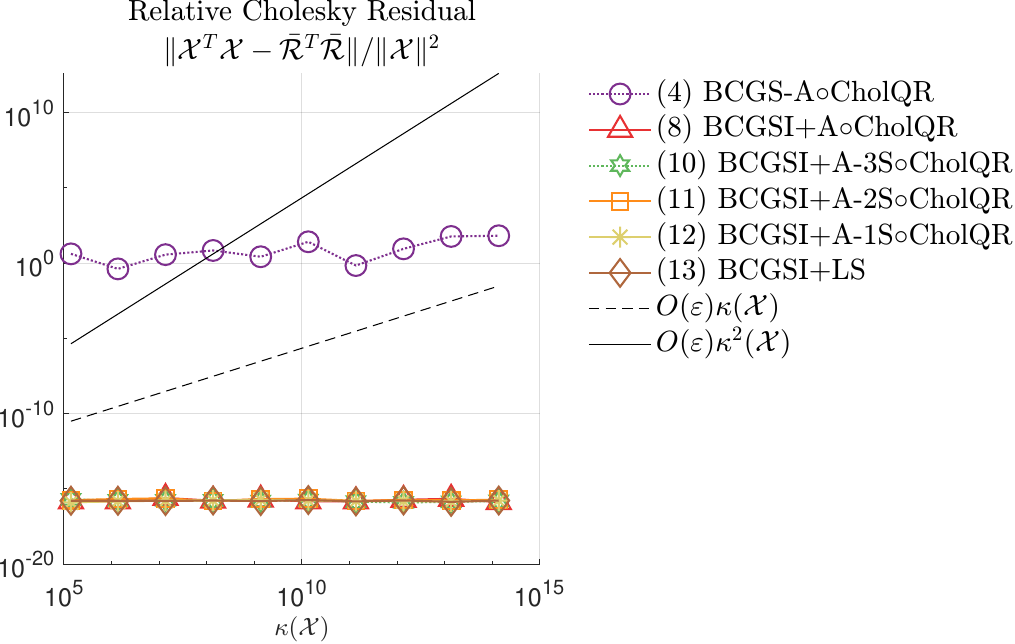}}
	    \end{tabular}
	\end{center}
    \caption{Comparison among column variants ($s=1$) for \piled matrices. \BCGSIROLS is Algorithm~7 from \cite{CarLRetal22}. Note that $\IOAnoarg$ is fixed as \HouseQR in \texttt{BlockStab}. \label{fig:roadmap_piled_column}}
\end{figure}
\section{Conclusions} \label{sec:conclusions}

In this work, we provide a number of new results on the loss of orthogonality in variants of block Gram-Schmidt. To enable a uniform approach to our analyses, we introduce an abstract framework which considers the effects of the projection and intraorthogonalization
stages of a \BCGS method. 

We first introduce a modification of the basic \BCGS method, called \BCGSA, in which the \IOnoarg used for the first block vector can be different than that used for subsequent blocks. We prove a bound on the resulting LOO. As a side effect, this gives the first known bound on the LOO for \BCGS in the literature. 

We then introduce a reorthogonalized variant of \BCGSA, \BCGSIROA, and prove a bound on its LOO. Our results reproduce the bound given by Barlow and Smoktunowicz for the case that a single \IOnoarg is used throughout \cite{BarS13}. Further, our analysis provides a valuable insight: we only need a ``strong'' \IOnoarg on the very first block. After this, less expensive and less stable \IOnoargs can be used. The first \IOnoarg used in the main loop  determines the constraint on the condition number of $\bXX$. The second \IOnoarg, used for the reorthogonalization, has no effect. 

The resulting \BCGSIROA has four synchronization points. We then demonstrate, through a series of steps, how each sequential removal of a synchronization point affects the LOO and constraints on the condition number for which the bound holds. We eventually reach a low-sync version with only a single synchronization point, equivalent to methods previously proposed in the literature, and we show that unfortunately, the LOO depends on the square of the condition number, which has been conjectured previously in \cite{CarLRetal22}. 

Despite the unsatisfactory results for the block variant, our analysis also gives bounds for column (non-block) one-sync variants which have been proposed in the literature \cite{BieLTetal22, SwiLAetal21}, and it is shown that these attain a LOO on the level of the unit roundoff.


\section*{Acknowledgments}%
\addcontentsline{toc}{section}{Acknowledgments}
The second author would like to thank the Computational Methods in Systems and Control Theory group at the Max Planck Institute for Dynamics of Complex Technical Systems for funding the fourth author's visit in March 2023. The fourth author would like to thank the Chemnitz University of Technology for funding the second author's visit in July 2023.

The first, third, and fourth authors are supported by the European Union (ERC, inEXASCALE, 101075632). Views and opinions expressed are those of the authors only and do not necessarily reflect those of the European Union or the European Research Council. Neither the European Union nor the granting authority can be held responsible for them. The first and the fourth authors acknowledge support from the Charles University GAUK project No. 202722 and the Exascale Computing Project (17-SC-20-SC), a collaborative effort of the U.S. Department of Energy Office of Science and the National Nuclear Security Administration. The first author acknowledges support from the Charles University Research Centre program No. UNCE/24/SCI/005.



\addcontentsline{toc}{section}{References}
\bibliographystyle{abbrvurl}
\bibliography{BlockStab}

\end{document}